\documentclass{amsart} 
\usepackage{geometry}
\usepackage{graphicx} % Required for inserting images
\usepackage{url}
\usepackage{amsmath}
\usepackage{amssymb}
\usepackage{bbold}
\usepackage{amsthm}
\usepackage{enumitem}
\usepackage{mathrsfs}
\usepackage{stmaryrd}
\usepackage{proof}
\usepackage[all]{xy}
\usepackage{tikz-cd}
\usepackage{mathtools}
\usepackage{ebproof}

\theoremstyle{definition}
\newtheorem{definition}{Definition}
\newtheorem{example}[definition]{Example}
\newtheorem{lemma}[definition]{Lemma}
\newtheorem{proposition}[definition]{Proposition}
\newtheorem{theorem}[definition]{Theorem}

\newtheorem{remark}[definition]{Remark}

\newcommand{\justification}[1]{\mathbf{#1}}
\newcommand{\seq}{\vartriangleright}
\newcommand{\derivation}[1]{\mathcal{#1}}

\newcommand{\base}[1]{\mathcal{#1}}
\newcommand{\baseA}{\base{A}}
\newcommand{\baseB}{\base{B}}
\newcommand{\baseC}{\base{C}}
\newcommand{\baseD}{\base{D}}
\newcommand{\baseE}{\base{E}}
\newcommand{\baseP}{\base{P}}

\newcommand{\atRule}[1]{\mathcal{#1}}

\newcommand{\baseop}[1]{{\baseB}{#1}} % atoms to base rules

\newcommand{\Pfrak}{\mathfrak{P}} % sets of inferonic atoms

\newcommand{\inferon}[4]{\llangle {#1}_{#4} , \mathcal{#2}_{#4} , \mathfrak{#3}_{#4} \rrangle}

\newcommand{\inferonPlain}[4]{\llangle {#1}_{#4} , {#2}_{#4} , \mathfrak{#3}_{#4} \rrangle}

\newcommand{\inferonSub}[6]{\llangle {#1}_{#2} , \mathcal{#3}_{#4} , \mathfrak{#5}_{#6} \rrangle}

\newcommand{\inferonSubSup}[7]{\llangle {#1}_{#2} , \mathcal{#3}_{#4}^{#5} , \mathfrak{#6}_{#7} \rrangle}

\newcommand{\inferatom}[3]{\llangle {#1}_{#3} , \mathfrak{#2}_{#3} \rrangle}

\newcommand{\agentd}[1]{\langle #1 \rangle}  
\newcommand{\agentb}[1]{[{#1}]}

\newcommand{\notequiv}{\neq}
\newcommand{\setofterms}{\mathcal{T}}
\newcommand{\cl}[1]{\mathrm{Cl}({#1})} % set of [terms etc.] that are closed

\makeatletter
\newsavebox{\@brx}
\newcommand{\llangle}[1][]{\savebox{\@brx}{\(\m@th{#1\langle}\)}% left double angle-bracket
  \mathopen{\copy\@brx\mkern2mu\kern-0.9\wd\@brx\usebox{\@brx}}}
\newcommand{\rrangle}[1][]{\savebox{\@brx}{\(\m@th{#1\rangle}\)}% right double angle-bracket
  \mathclose{\copy\@brx\mkern2mu\kern-0.9\wd\@brx\usebox{\@brx}}}
\makeatother

\usepackage{xcolor}

\newcommand{\fillBox}{\hfill$\Box$}

\newcommand{\predON}[1]{\mathrm{ON}{#1}} % Flashlight Example
\newcommand{\predLIT}[1]{\mathrm{LIT}{#1}} % Example
\newcommand{\predONflash}[1]{\mathrm{fON}{#1}} % Example
\newcommand{\predLITflash}[1]{\mathrm{fLIT}{#1}} % Example

\begin{document}

% \title[Inferonics]{Inferonics: An Inferentialist Account of Information Through Proof-theoretic Semantics}
\title[Towards an Inferentialist Account of Information]{Towards an Inferentialist Account of Information Through Proof-theoretic Semantics} 
\author{Matthew Collinson$^*$}%\thanks{$^*$Corresponding author.} 
\address{University of Aberdeen, King's College, Aberdeen AB24 3FX, Scotland, UK}
\email{matthew.collinson@aberdeen.ac.uk}
\author{Timo Eckhardt$^*$}%\thanks{$^*$Corresponding author.} 
\address{UCL \& Institute of Philosophy, University of London, Senate House, Malet Street, London WC1E 7HU, England, UK}
\email{t.echkhardt@ucl.ac.uk}
\author{David Pym$^*$}\thanks{$^*$Corresponding author.} 
\address{UCL \& Institute of Philosophy, University of London, Senate House, Malet Street, London WC1E 7HU, England, UK {\tt https://www.cantab.net/users/david.pym/}}
\email{david.pym@sas.ac.uk}

\date{}

\begin{abstract}
Information is one of the most widely-discussed concepts of the current era. 
However, a great deal of insightful work notwithstanding, it is yet to be given wholly convincing logical or mathematical foundations. Without them, we lack adequate reasoning tools for understanding the complex ecosystems of systems upon which the society depends. We seek to rectify this by taking a first step towards developing an inferentialist semantic theory of information. There are three key interacting components. First, conceptual analysis: the metaphysics of information. Dretske expressed the key concepts of information in terms of intentionality, truth, and transmissibility. We replace truth with inferability, and trace the consequences of this replacement. Second, logic: proof-theoretic semantics (P-tS) provides a mathematical-logical realization of inferentialist reasoning. Using P-tS, we develop the first steps towards a mathematical-logical theory of an inferentialist primitive unit of information, the `inferon', which 
is an inferentialist counterpart to the concept of `infon' previously 
studied in situation theory. This proof-theoretic approach counterpoints the model-theoretic view of information articulated in situation theory. Furthermore, we argue that it facilitates addressing all three components of van Benthem and Martinez's categorization of the understandings of information, as range, as correlation, and as code. Our focus is on information-as-correlation. Third, systems: the P-tS tools we develop provide the basis for a mathematical account of distributed systems modelling --- a key tool from informatics for understanding the organization of information processing systems. This yields a reasoning-based theory of information flow in models of distributed systems. Overall, we seek to give a conceptually rigorous mathematical-logical account of information and its role within informatics, grounded in inference and reasoning.
\end{abstract}

\maketitle

\section{Introduction} \label{sec:introduction}

% \begin{enumerate}
% \item[--] knowledge, background, \cite{Dretske1999}
% \item[--] a basic review of Dretske's position on metaphysics of information, 
%    with a brief, first-cut discussion of the effect of 
%    replacing truth with inference

%    motivation? 

% \item[--] a basic review of the key points of situation theory in the 
%    sense of barwise/devlin/seligman 

% \item[--] a brief summary of the salient bits of B-eS 

% \item[--] basic set up of inferons, with examples, and discussion of 
%    design choices? what can we say about proof systems for the logic(s) of inferons? s and c theorems 

%    agency? intensional vs extensional?  

% \item[--] basics of the implied approach to information flow, with
%    examples

%    relate to B-eS dist. syst. modelling set up? 

% \item[--] inferomorphism equiv theorem 

% \item[--] directions --- plenty, especially epistemic
% \end{enumerate}

Information is one of the most widely-discussed concepts of the current era. 
However, a great deal of insightful work notwithstanding, it is yet to be given wholly convincing logical or mathematical foundations. Without them, we lack adequate reasoning tools for understanding the complex ecosystems of systems upon which the society depends. We seek to rectify this by taking a first step towards developing an inferentialist semantic theory of information. There are three key interacting components. First, conceptual analysis: the metaphysics of information. Dretske expressed the key concepts of information in terms of intentionality, truth, and transmissibility. We replace truth with inferability, and trace the consequences of this replacement. Second, logic: proof-theoretic semantics (P-tS) provides a mathematical-logical realization of inferentialist reasoning. Using P-tS, we develop the first steps towards a mathematical-logical theory of an inferentialist primitive unit of information, the `inferon', which 
is an inferentialist counterpart to the concept of `infon' previously studied in situation theory. This 
proof-theoretic approach counterpoints the model-theoretic view of information articulated in situation theory, and represents a major conceptual shift. Furthermore, we argue that it facilitates addressing all three components of van Benthem and Martinez's categorization of the understandings of information, as range, as correlation, and as code. Our focus is on information-as-correlation. Third, systems: the P-tS tools we develop provide the basis for a mathematical account of distributed systems modelling --- a key tool from informatics for understanding the organization of information processing systems. This yields a reasoning-based theory of information flow in models of distributed systems. Overall, we seek to give a conceptually rigorous mathematical-logical account of information and its role within informatics, grounded in inference and reasoning.

Barwise and Seligman, \emph{Information Flow: The Logic of Distributed Systems}, Cambridge University Press, 1997, page~22:  
\begin{quote}
Inference has something crucial to do with information. On the one hand, inference is often characterized as the extraction of implicit information from explicit information. This seems reasonable, because it is clearly possible to
obtain information by deduction from premisses. In carrying out these inferences in real life, we take as premisses some background theory of how the world works. On the other hand, getting information typically requires inference.
\end{quote} 

Our starting point is the position of inferentialism --- that the meaning of linguistic entities, including propositions and proofs, is determined by their use. Inferentialism has been articulated with clarity and elegance by Brandom (and also Hlobil) in a sequence of major publications, most notably \cite{Brandom1994, Brandom2000,Brandom2024}. We adopt an inferentialist perspective on, and reinterpretation of Dretske's metaphysics of information and employ in our formal development a mathematical-logical realization of inferentialism known as proof-theoretic semantics. 

We begin the story in Section~\ref{sec:metaphysics} where we summarize Dretske's metaphysics of information \cite{Dretske2008}, explaining how we seek to replace his requirement of truth with a requirement of inferability. We also discuss van Benthem and Martinez's characterization of formal theories of information into the categories information-as-range, information-as-correlation, and information-as-code~\cite{BenthemMartinez2008stories} and how our approach, while largely within information-as-correlation, accounts for significant parts of all three categories.

Next, we have two important background sections. The first, 
Section~\ref{sec:situations}, recalls the basic ideas of situation theory that are relevant to our development. The second, Section~\ref{sec:P-tS}, introduces the key ideas of proof-theoretic semantics, stressing what is widely know as base-extension semantics, which is an essential tool in out later development. In Section~\ref{sec:situations}, we summarize the key ideas of situation theory, as developed logically by Barwise, Etchemendy, Devlin, Perry, Seligman, and others \cite{BarwisePerry1981,BarwisePerry1983,Barwise1986,BarwiseEtchemendy1990,Devlin1991,BarwiseSeligman1997,Seligman2014}. The core idea of situation theory in our context is that of the (atomic) \emph{infon} 
--- a basic unit of information that is very closely related to a first-order atomic predicate, which derives its meaning from a first-order model. 

In Section~\ref{sec:P-tS}, we introduce the key ideas of proof-theoretic semantics \cite{Prawitz1971ideas,Schroeder2007modelvsproof,SEP-PtS,Dummett1991}, stressing the role 
of base-extension semantics (B-eS) \cite{Sandqvist2015IPL,PiechaPS-H2016-AtomicSystems,SandqvistWLD2022,Sandqvist2025,GP2025-P-tS-S-oL,GheorghiuPym2025}, in particular as formulated for intuitionistic logic, which provides our logical point of departure. The key idea here is that of 
a \emph{base} of atomic rules, such as  %\notemc{something missing here} Yep! 
    \[
    \dfrac{\qquad}{p} \qquad 
    \dfrac{p_1 \ldots p_k}{p} \qquad \mbox{and} \qquad 
        \dfrac{\begin{array}{ccc}
               [P_1] &        & [P_k] \\ 
                p_1  & \ldots &  p_k \\ 
               \end{array}}{p} 
    \]
where each $p$ and $p_i$ is an atom, not a formula or a variable ranging over formul{\ae}, and $P_i$ a set of such atoms. Conceptually, the key point here is that base rules are pre-logical --- they involve no use of logical operators. We discuss how bases work in Section~\ref{sec:P-tS}. 
%\notemc{I think here is a good place also to say something like (below the rules):  ``where each $p$ and $p_i$ here is an atom, not a formula or a variable ranging over formulae.'' It can be confusing on first contact for a non-pts person.} YEP, that's missing. 

In Section~\ref{sec:LogicalBasis}, we get to the core of our ideas about information, giving a range of simple examples in Section~\ref{sec:first-examples}. Central 
to this development is the concept of the inferentialist 
infon, or \emph{inferon} --- this too is a basic unit of information and also resembles a first-order atomic predicate, but its meaning is derived not from a first-order model, but rather from a global base of atomic rules, 
describing the context for the inferon, together with the 
inferon's own local base of atomic rules, describing the 
inferential capacity associated with the inferon. 

Then, in Section~\ref{sec:morphisms-flow}, we develop an account of information flow in the inferentialist setting that is analogous to the theory presented by Barwise and Seligman \cite{BarwiseSeligman1997}.      Section~\ref{sec:flow} relates the theory of information flow introduced in Section~\ref{sec:morphisms-flow} to a quite generic, yet practically motivated, approach to modelling distributed systems. Finally, in Section~\ref{sec:discussion}, we consider a range of further developments of the programme initiated by the work presented herein.  

In summary, we begin to develop an inferentialist theory of information through a reconstruction of situation theory (Barwise et al.) in the inferentialist logical setting provided by proof-theoretic semantics (in particular, starting from the base-extension semantics of intuitionistic propositional logic \cite{Sandqvist2015IPL}). Our contribution, through what is a work of philosophical logic, is to begin to establish the viability and value of this approach. Much work, both mathematical and philosophical, remains. 

\section{The Metaphysics of Information: Dretske and van~Benthem} \label{sec:metaphysics}

Building on `Knowledge and the Flow of Information' \cite{Dretske1999}, Dretske gives a basic account of the metaphysics of information in \cite{Dretske2008}. 

Dretske \cite{Dretske2008} identifies that information is an important concept in many aspects of science and everyday life that, despite that importance, lacks a general theory. This is an especially important and difficult problem as the use of the word has become quite different and specialized as it has been integrated into different disciplines. To address this Dretske aims to identify the concept that stands behind these specific uses and constitutes our shared understanding of information. A statistician, a computing scientist, a philosopher, and a layman, for example, might all talk about information differently, but with each them we understand that what they are talking about is, in fact, information.

Dretske identifies three properties of information:
\begin{itemize}
\item[--] Intentionality 
\item[--] Truth  
\item[--] Transmissibility.  
\end{itemize}
Intentionality requires information to be about something. With any information we have to be able to ask the question `information about what?'. It follows that pieces of information must be semantic objects. A letter is not information even if the message conveyed by it might be. Furthermore, this means that the same signal might carry different information depending on the sender, the receiver, and the general circumstances, with obvious examples being words like `you' and `I'.

Truth simply requires the information to be correct. This is a much discussed property called the \textit{Veridicality Thesis} (see, for example, \cite{Floridi2011-PhilInf}). Dretske argues that misinformation or false information are not actually information, but rather are mistakenly represented as information, much like a fake diamond is not actually a diamond. The truth of information is important because of its epistemological role: Information is a necessary condition for knowledge. A false proposition works just as well to check the validity of an argument as a true one. However, as Dretske puts it `What good is flawless reasoning if everything you conclude is false?' \cite{Dretske2008}. Information is sought after exactly because of its truth and that truth, together with correct reasoning, allows us to draw true conclusions. Israel and Perry \cite{IsraelPerry1990}, in their treatment of `information reports', assume that information conveyed within such reports must be `factive'; that is, true.

Transmissibility requires information to have the potential to be sent and received.  Meaning can, for example, be transmitted by simply writing it down. For information, however, more is required. As information is only real when it is true, it is not enough to transmit the truth, but you have to do so in a way that the receiver can know the truth you transmit. For example, looking at a thermometer conveys information about the temperature because the receiver already considers the thermometer to be a reasonably accurate tool for measuring temperature. Similarly a trusted expert can convey information on a topic in a way a random person cannot. A message from an unknown source can, of course, still be information if it includes the necessary evidence to generate knowledge in the receiver.

We aim to give an inferentialist account of information and so differ from Dretske in various ways. The most significant such difference is with the notion of truth. Instead of being given in terms of truth and falsity, information is given in terms of inferability. In our specific case, this means the derivability of basic units of information from a set of atomic rules called a base. This has a significant impact on the discussion of veridicality. We conjecture that an inferentialist version of the veridicality thesis holds. Our approach requires inferences to be valid in order to count as information and so the idea that information has to be correct remains, fundamentally, preserved. As this is supposed to be a first step, a more in-depth discussion of the details and intricacies is outside the scope of this paper and left for future work (see Section~\ref{sec:discussion}). However, we discuss some initial thoughts in Remarks \ref{rem:verid-1}, \ref{rem:verid-2}, and \ref{rem:verid-3}.

Less obvious is how the inferentialist view impacts the idea of transmissibilty. Information still has to be sent and received, of course, but we can give a precise account of the the idea that information has to be transmitted so that the receiver can know it. We take a piece of information to include a base, as mentioned above, so that the statement the information is about can be derived from the rules in that base together with the rules the receiver already had access to. This gives an easy way of checking whether a message fails the transmissibility criterion: whenever this inference is not possible. Only intentionality remains largely unchanged. However, derivability, more so than truth, asks the question `for whom?' and so puts further importance on the question of the r{\^o}le of the agents.

% -- VT at level of informational judgements 

% -- inluding agents' inferences about situations 

% -- 

% Us: 
% \begin{itemize}
% \item[--] Intentionality and agency 
% \item[--] Inference and validity  
% \item[--] Transmissibility through flow  
% \end{itemize}

% Intentionality. 

% Dretske: 
% \begin{quote}
% A brick isn’t information, but it might carry information. It might, for instance, carry information in its shape, color, size, or material composition, about its place and manner of manufacture. Or it might, by its deliberate placement and orientation, tell (inform) us where an accomplice is hiding. Similarly, we can use a fence post as a rough sundial --- thus using the post (and its shadow) as a source of information about the time of day. Clearly, though, the information 
% provided is not the brick, the post, or the shadow. It is what the brick, post, or shadow tell us about something else --- in this case the time of day, place of manufacture, or location of an accomplice.
% \end{quote}

% \note{Generally, we should highlight (later) how our set-up fits the inferentialist versions of Dreske's ideas.}

For the purposes of this paper, we have taken Dretske's metaphysics of information \cite{Dretske1999,Dretske2008} as our point of departure. There is a significant literature on the philosophy and logic of information --- for example, Bimb\'o \cite{Bimbo2016}, Carnap \cite{Carnap1935,Carnap1951}, Dunn \cite{Dunn2001} (and `A ``Reply'' to my ``Critics''\,' in  \cite{DunnIBL}), Floridi \cite{floridi2005semantic,Floridi2011-PhilInf,Floridi2019-LogInf} to name just a few contributors of very many. 

In particular, Johan van~Benthem has written a great deal about the relationship between logic and information. In \cite{BenthemMartinez2008stories}, `The Stories of Logic and Information', van~Benthem and Martinez characterize three categories of formal theories of information:   
\begin{enumerate} 
\item \emph{Information-as-range}. In this view, information is modelled by the range of possibilities that are consistent with it --- 
the more options one has to consider, the less information one has.  From a logical point  view, epistemic logics  \cite{Meyer1995,DEL2008} beginning with \cite{vonWright1951,Hintikka1962} and including Public Announcement Logic (PAL) \cite{Baltag2016} provide the basic logical tools. 

\item \emph{Information-as-correlation}. This view takes what information is \emph{about} as the focus of study. 
Information is treated as existing in different situations with \emph{dependencies} between such situations --- situation theory, and `infon logic' \cite{BarwisePerry1983,Devlin1991,Devlin2006,BarwiseSeligman1997,Abramsky2009-ABRFIT} 
are primary tools for this view. An alternative approach is the Modal Logic of Dependence \cite{BaltagvanBenthem2021,vanBenthemLIA2025}, also providing unification with an information-as-range theory.

\item \emph{Information-as-code}. Here information is thought of as \emph{encoded} in sentences and messages. This allows for a syntactical investigation of information which relies on proof theory rather than semantics. The inferences possible given a certain message tell us its informational content --- Shannon's work \cite{Shannon1949} and Abramsky's process-oriented approach \cite{Abramsky2008} provide good examples of this view.   
\end{enumerate}
These ideas are also conveyed by Adriaans and van~Benthem in \cite{ADRIAANS20083}, `Information is What Information Does'.  

While the work discussed above, and much more besides,\footnote{For example, see \cite{Kahre2002} for a discussion of Shannon's theory, some derivatives of it, and some other approaches  \cite{BarHillel1964,Chaitin1987,Cover1991,Stonier1997,Nielsen2000,Calude2002,Kohlas2003,vonBaeyer2004}. More detailed discussion of these is beyond our present scope.} has considered many aspects of the philosophy, logic, and algorithmics of information, and while this work exposes many challenging issues and has attracted much comment too various to discuss here, it has not addressed the challenge of establishing a unified logical theory. By contrast, through our approach, we seek to be able to address at least some of the concerns of each of, as indicated below. 

Our first point of departure, as discussed above, is Dretske's metaphysics of information and we propose an inferentialist variation of Dretske's analysis in which the requirement of truth is replaced by a requirement of inferability. Our second point of departure, also as discussed above, is van~Benthem and Martinez's categorization of formal theories of information. We have argued that van~Benthem and Martinez's categorization can be seen as being underpinned by Dretske's metaphysics. 
Our inferentialist view of van~Benthem and Martinez's categorization of formal theories of information --- in which information-as-correlation is based not on model-theoretic infons and situations, but rather on proof-theoretic inferons and situations, grounded in proof-theoretic semantics --- can evidently be seen as being underpinned by our inferentialist variation of Dretske's analysis. 

While the primary purpose of the present paper is to develop an inferentialist theory of informa\-tion-as-correlation, as described in terms of inferons and situations, we conjecture that our approach based on proof-theoretic semantics is also able to account for at least significant parts of information-as-range and information-as-code: 
\begin{enumerate}
\item \emph{Information-as-range}. Rather than considering a range of models or states that model facts, this understanding of information can be seen as employing a range of bases that support inferons (see Section~\ref{sec:LogicalBasis} for the meaning of this). 

We can expect initially to work with epistemic logics given by B-eS. We have established background work for intuitionistic modal logics (K through to S4) \cite{BuzokuPym-Modal2026}, classical modal logics (K through to S4 and S5) \cite{EckhardtPym-Modal2024,EckhardtPym2025S5}, and Public Announcement Logic (PAL) \cite{EckhardtPym2024PAL}. However, we will seek more general accounts of epistemic logics and agency --- see Section~\ref{sec:LogicalBasis} for some first thoughts on agency --- in the setting of the account of information given by proof-theoretic semantics. We can expect to find substantive  
connections with van Benthem's work on logical dynamics and information \cite{vanBLDII2011,vanBenthemTracking2016} and perhaps Floridi's \cite{Floridi2006LogicBeingInformed}, 
though this beyond our present scope.  
\item \emph{Information-as-correlation}. The inferentialist formulation of this through proof-theoretic semantics is the main topic of this paper. 
\item \emph{Information-as-code}. Our proof-theoretic approach to information is able to address some aspects of the information-as-code understanding, in which the inferences that are possible from a given message tell us its informational content. In our setting, such inferences are part of the inferential semantics ---  
characterized as support for an inferon with respect to a base (again, see Section~\ref{sec:LogicalBasis} for the meaning of this) --- rather than being accessible only through meta-theoretical reasoning. 
% \notemc{I think this is a very important point for the whole programme. Slight re-word? Some thing like ``In our setting, inferences about available information are not only possible via meta-theoretical reasoning about formal systems. Inferrable information is characterised semantically as support for an inferon with respect to a base.''} Done. 

Three worthwhile observations are the following: 
    \begin{itemize} 
    \item[--] For the first-order inferonic logic, we can work with the theory of Peano Arithmetic (PA) \cite{GP2025-P-tS-S-oL,gheorghiu2026classicallogicbivalance} and so represent numerical measure of information within our semantics. 
    \item[--] We can develop second-order inferonic logic (based on \cite{GP2025-P-tS-S-oL}), with the theory of second-order arithmetic (building on \cite{GP2025-P-tS-S-oL}), so allowing aspects of set theory and analysis to be deployed. 
    \item[--] We can explore a proof-relevant inferon logic, where proof-objects are associated with the clauses of the support relation (cf. BHK semantics, as discussed in connexion with proof-theoretic validity in \cite{Schroeder2007modelvsproof}; see also \cite{ArtemovNogina2005}). 
    \end{itemize}
\end{enumerate}
We discuss our intended approach to these further in Section~\ref{sec:discussion}. This is somewhat related to Prawitz's original conception of proof-theoretic validity, which we discuss briefly in Section~\ref{sec:P-tS}. 

Finally, a remark on agency: in each of the three categories, but especially in `range' and `correlation', notions of agency are relevant. In epistemic logics (e.g., \cite{Meyer1995,DEL2008,vanBenthemLIA2025}), this is commonplace notion and in our setting we should seek to explore, echoing Dretske, `knowledge and the flow of information'. In our present development, we introduce, in Section~\ref{sec:LogicalBasis}, a notion of agency through `action' or `agent'  modalities. Further development of these ideas, to which we return briefly in Section~\ref{sec:discussion}, is deferred to future work.  

Before we proceed to develop our inferentialist theory of inferons and situations, we give, in Section~\ref{sec:situations}, a brief overview of situation theory as presented by Barwise, Devlin, Etchemendy, Perry, Seligman, and others.

\section{Situations} 
\label{sec:situations}

%\subsection{Situation Theory and Situation Semantics}
%\label{subsec:SituationsDetail}

%\note{revisit this section after finishing discussion of our variation of Dretske}

% \note{states of affairs? too Kripke-ish?}

\emph{Situation theory} is a mathematical theory of information and information flow: see, among other sources,  \cite{BarwisePerry1981,BarwisePerry1983,Barwise1986,BarwiseTSIL1988,BarwiseEtchemendy1990,Devlin1991,BarwiseSeligman1997,Seligman2014}. 
The focus of situation theory is on making an applicable theory of how pieces of information are constructed and interpreted, including how they combine, how one piece of information may lead to another piece of information, and how they flow between situations. 

% \notemc{I think this is not necessary.}
%situation theory was developed in order to give a formal analysis of Barwise and Perry's situation semantics for natural language \cite{BarwisePerry1983}, itself motivated by a desire to give meaning to direct perception reports such as might be given by agent that observes some event in some situation \cite{Barwise1981scenes,BarwisePerry1981,BarwisePerry1983,kratzer2023}.
%\[
%    \note{EXAMPLE} 
%\]
%The notion of situation, and the associated notion of information, became central to the analysis. 

The techniques employed in the mathematical presentation of situation theory are given in the general style of model-theoretic semantics, using interpretation functions to give meaning to propositions. Before proceeding to develop our inferentialist account of information and information flow, grounded in proof-theoretic semantics, we give a brief sketch of the most basic concepts of situation theory. We do this in a simplified way, eliding much of the substantial complexity of the full set-up, an extensive  account of which is provided by Devlin \cite{Devlin1991}. In reality, there are several different situation theories, and we here try to give a flavour of some key concepts.

In situation theory, pieces of information are formalized using syntactic entities called \emph{infons}. The form of an (atomic) infon is, essentially, 
$
\llangle R , a_1  , \ldots a_n , \mathfrak{b} \rrangle
$, 
where $R$ is an $n$-ary multi-sorted relation %predicate 
symbol, and $a_1 , \ldots , a_n$ is a sequence of constants of the appropriate sorts to instantiate $R$, and $\mathfrak{b} \in \{0,1\}$ is a polarity. The idea is that this piece of syntax stands for `$a_1,\ldots , a_n$ do (do not) stand in relation $R$'. Compound infons are generated using connectives for conjunction, implication, quantifiers, and so on. 

Infons themselves are taken to be neither true nor false: 
the relevant question is whether any given situation \emph{supports} the infon (or not). To this end, a semantic judgement relation is used, of the form $s \vDash \sigma$, 
% \begin{equation*}
% %$
% s \vDash \sigma
% \end{equation*}
% %$
for $s$ a situation and $\sigma$ an infon. 
In applications, $s$ is often an \emph{abstract situation}, which is a set of infons. In this case, the support relation, $\vDash$, is just the `element-of' relation, $\in$.
%Sometimes employed is a notion of `real world' --- if $\vDash$ coincides for both the abstract and corresponding `real' situation, then the abstract situation is \emph{actual}. 
We remark that Mares \cite{Mares1996} and Restall \cite{Restall1996} consider how Routley--Meyer ternary-relation semantics for relevance logic can be used to give a model-theoretic logic of infons. 

The general set-up (e.g., \cite{Devlin1991}) uses a system of types. In particular, situations can be typed  using a typing judgement. This takes the form $s:T$, where $s$ is a situation of type $T$. In Barwise and Perry's early set-up \cite{BarwisePerry1983}, a type is (essentially) a set of infons. Situations then consist of a type together with a location. In Devlin's set-up \cite{Devlin1991}, one can abstract to form the types of situations supporting particular infons. For example, for a $0$-ary predicate (proposition) $\mathit{fire}$, a corresponding type of situations (in which there is fire) can be defined in which $\llangle \mathit{fire} , 1 \rrangle$. The notion of type employed in some parts of situation theory is rich and complicated and requires considerable mathematical machinery from classical model theory to realize set-theoretically. By contrast, the approach we develop in Section~\ref{sec:LogicalBasis} is grounded in proof theory.  
%For some additional complexity, return for a moment to the description of an infon $\llangle R , a_1  , \ldots a_n , \mathfrak{b} \rrangle$. What we referred to above as the `sort' of each $a_i$ is a types. Secondly, the individuals $a_i$ can themselves be situations and types. For each infon $\sigma$, there is a `situation type' $S=[\dot{s} \mid \dot{s} \vDash \sigma]$ such that $s:S$ if and only if $s \vDash \sigma$ for any situation $s$. 
%A major result in situation theory is that the types of situation theory can be realized by (non-well-founded) sets \cite{Aczel}. 

A \emph{constraint} is a binary relation on situation types, written $S \Rightarrow S'$, 
% \begin{equation*}
%         S \Rightarrow S'
% \end{equation*}
where $S$ and $S'$ are situation types. The idea of the constraint is, essentially, that any situation $s$ that has $s:S$ can be related to some  situation $s':S'$. In particular,  a constraint $S \Rightarrow S'$ gives us that if $S$ is the type of situations that support $\sigma$, and $S'$ is the type of situations that support $\sigma'$, then if there is $s \vDash \sigma$, then there is some $s' \vDash \sigma'$. For example, $S$ may be the type of all smokey situations and $S'$ may be the type of all firey situations, and a constraint might express `there is no smoke without fire'. More technical parts of the set-up allow closer control of how the resulting witness $s'$ is linked to the triggering $s$. Agents are said to be \emph{attuned} to constraints when they are aware of them, and (consciously or unconsciously) accept them.  

It's clear then that even in the most basic conception, as well as the commitments to ontology, situation theory also commits methodologically to model-theoretic semantics and, indeed, bivalent truth. Moreover, a constraint is a semantic 
%semantical 
construct designed to capture a  possible inference rule, while attunement means accepting the validity of that inference rule. 
%There are elements of realism underpinning the conceptions of situation, infon and constraint. Discussion in Devlin, Logic and Information, sections 2.6-2.8 pertinent. 

%\subsection{Classifications and Infomorphisms}
%\label{subsec:Classifications}

Situation theory requires the rich notion of type as described in order to support all the goals of situation semantics. Consequently, the realization of it in terms of sets (models) is somewhat complex. Later, simplified theories were produced to give a simpler treatment of information flow.
One such theory, involving sites and channels in information networks, and using graph structures and substructural logic, was given by Barwise, Gabbay, and Hartonas \cite{BarwiseGabbay1995}. 
We discuss a development inspired by this approach in Section~\ref{sec:flow}.

In a second simplified theory, Barwise and Seligman developed the theory of classifcations and infomorphisms, under the name `The Logic of Distributed Systems'. The standard introductory reference for this is the monograph \cite{BarwiseSeligman1997}. We return to the technical development of these ideas when we introduce our corresponding inferentialist notion in Section~\ref{sec:morphisms-flow}. 

% \note{Discussion of the embedded information flow principles from B\&S}.
Our development in the sequel is %will remain 
compatible with the four Principles of Information Flow \cite{BarwiseSeligman1997}. 
% (p8) `First principle of Information Flow: Information flow results fom regularities in a distributed system.'
Information flow requires some things to flow between, and this requires some conception of system that is distributed over various locations, and that has `regularities' describing the informational linkage of the distributed components. 
% (p27) `Second principle of infromation flow: Information flow crucially involves types and their particulars.'
Flow requires us to be able to identify the components of a system, say which are connected, and to reason about the (related, by regularity) informational properties of the various components. 
% (p35) `Third principle of Information Flow:  It is by virtue of regularities among connections that information about some components of a distributed system carries information about other components'
The level of detail of the model/analysis of the distributed system determines what are the regularities that can be identified. 
% (p43) `Fourth Principle of Information Flow: The regularities of a given distributed system are relative to its analysis in terms of information channels.'

\section{Logical Background: Proof-theoretic Semantics} 
\label{sec:P-tS} 

Recall that our purpose is to construct an inferentialist foundation for information, drawing upon situation theory, but grounded in proof-theoretic semantics rather than model-theoretic semantics. Accordingly, we now give a summary of the aspects of the theory of proof-theoretic semantics that we require. 

Proof-theoretic semantics is a counterpart for model-theoretic semantics in which the meanings of proofs and formul{\ae} are grounded in inference rather than truth (in the sense of 
Tarski). As we have remarked, can be seen as a logical realization of inferentialism \cite{Brandom1994, Brandom2000,Brandom2024}.  See Schroeder-Heister's \cite{Schroeder2007modelvsproof} for a brief explanation. 

We adopt closely the introduction to proof-theoretic semantics as  presented in \cite{PymRitterRobinson2024}. %add other similar? 
As in \cite{PymRitterRobinson2024}, our discussion, in Section~\ref{subsec:P-tV}, of proof-theoretic validity, which is concerned with the validity of proofs, is not strictly necessary for this paper. However, it provides valuable background for our discussion, in Section~\ref{subsec:B-eS}, of base-extension semantics, which is concerned with the validity of formul{\ae}. Also, the final paragraph of Section~\ref{subsec:P-tV} provides a brief discussion of the necessity of considering `base extensions'. 

Readers familiar with proof-theoretic validity (e.g., \cite{Prawitz1971ideas,Schroeder2007modelvsproof}), or who just wish to proceed directly to the development of this paper, may wish to skip Section~\ref{subsec:P-tV} and proceed directly to Section~\ref{subsec:B-eS} on B-eS. 

\subsection{Proof-theoretic Validity} \label{subsec:P-tV} 

Our technical starting point is the logical realization of an inferentialist theory of meaning known as \emph{proof-theoretic semantics} (P-tS). P-tS has its origins in the work of Prawitz \cite{Prawitz1971ideas,Prawitz1973towards,Prawitz2005,Prawitz2006meaning} and Dummett \cite{Dummett1975,Dummett1991}. Proof-theoretic semantics stands in contrast to model-theoretic semantics, in which the meaning of formul{\ae} and of proofs is given denotationally~\cite{Dummett1991}. 

We can usefully consider P-tS to have two main logical lines of development. The first, the earlier approach, is directly in the spirit of Prawitz and Dummett and is concerned with the validity of proofs. We refer to this as proof-theoretic validity (P-tV). 
The second line of development, \emph{base-extension semantics} (B-eS), is concerned with the validity of formul{\ae}. 

% Proof-theoretic Validity
Proof-theoretic validity is an attempt to answer the following seemingly naive question: how can the validity of a proof (or derivation) be judged? Before looking at the answer provided by P-tV, there is one answer already given by model-theoretic semantics. Following Schroeder-Heister \cite{Schroeder2007modelvsproof}, we remark that this view lies in the tradition of Tarskian semantics, which works through the transmission of truth, as follows: given a model (i.e., some semantic structure) $\mathfrak{M}$, $\Gamma \models_\mathfrak{M} \phi$ if, and only if, whenever $\models_\mathfrak{M} \psi$ holds for all $\psi \in \Gamma$, $\models_\mathfrak{M} \phi$ holds as well.  
    % \[
    % \begin{array}{rcl}
        % \Gamma \models_\mathfrak{M} \phi & \mbox{iff} & 
        %     \mbox{if, for all $\psi \in \Gamma$, if 
        %     $\models_\mathfrak{M} \psi$, then 
        %     $\models_\mathfrak{M} \phi$}   
    % \end{array}
    % \]
The validity of a consequence is then given as follows: 
    \[
    \begin{array}{rcl}
        \Gamma \models \phi & \mbox{iff} & 
            \mbox{for all models $\mathfrak{M}$, $\Gamma \models_\mathfrak{M}\phi$}   
    \end{array}    
    \]
    In this setting, proofs can be understood denotationally; that is, as operations on the interpretation of formul{\ae} in models, by which consequences are yielded; here, for example, an interpretation of a proof of $\phi$ from proofs of those $\psi$ in $\Gamma$. However, such a model-theoretic definition of validity is unsatisfactory for answering the specific question one asks here about justification of validity: M-tS provides little insight into the structure of the formul{\ae} involved in a valid sequence; it also presumes an epistemic priority of models over proofs which is itself ungrounded.

    In proof-theoretic semantics (P-tS), one seeks to give meaning to propositions and to proofs purely in terms of inference. More specifically, meaning is given in terms of inferences in systems of pre-logical inference rules between atoms, called \emph{atomic rules} \cite{Prawitz1971ideas}. A set of such rules is called a \emph{base}. An atomic rule is one in which an atomic proposition is inferred from atomic propositions and which makes no reference to logical constants (connectives or other operators). For example, we have that if, say, Amber is a vixen, we can infer that it  is a fox, and that it is female. In the standard format of proof rules, we have:
        \[
            \dfrac{\text{Amber is a vixen}}{\text{Amber is a fox}}
            \qquad \mbox{\text{and}} \qquad  
            \dfrac{\text{Amber is a vixen}}{\text{Amber is female}}
        \]
    Similarly, if we have that Amber is a fox and that Amber is female, then we can infer that Amber is a vixen:
        \[
            \dfrac{\mbox{Amber is female} \qquad 
            \mbox{Amber is a fox}}{\mbox{Amber is a vixen}}  
        \]
    In general, atomic rules have forms such as 
    \[
    \dfrac{\qquad}{p} \qquad 
    \dfrac{p_1 \ldots p_k}{p} \qquad \mbox{and} \qquad 
        \dfrac{\begin{array}{ccc}
               [P_1] &        & [P_k] \\ 
                p_1  & \ldots &  p_k \\ 
               \end{array}}{p} 
    \]
    and more (e.g., \cite{PiechaPS-H2016-AtomicSystems,SandqvistWLD2022,Sandqvist2025}). The latter allows each of the $p_i$s to be proved from dischargeable hypotheses $P_i$s \cite{Prawitz1971ideas,Sandqvist2015IPL,Piecha2019incompleteness}.\footnote{The choice of the form of atomic rules in bases has a profound effect on the strength of the semantics that is obtained \cite{PiechaPS-H2016-AtomicSystems,SandqvistWLD2022,Sandqvist2025}.} Again, we emphasize that such atomic rules are \emph{pre-logical} --- they contain only atoms and do not refer to any logical constant. The use of the discharging of hypotheses in atomic rules is discussed in \cite{Sandqvist2015IPL}.  

    The propositional bases we have described can be generalized to bases of rules for atomic predicates \cite{GP2025-P-tS-S-oL}. Such predicate bases remain pre-logical in that they make no reference to logical connectives. Predicate bases provide a key tool for our project of providing an inferential account of information using the logical tools of proof-theoretic semantics. 

    As we have seen, situation theory's account of information is grounded in the idea of infon of the form $\llangle R , a_1  , \ldots a_n , \mathfrak{b} \rrangle$ in which $R$ is relation, the intended interpretation of a corresponding predicate symbol. This set-up is realized in terms of model-theoretic semantics, which provides the foundation for the development of the logic of infons: infons and situations are interpreted truth-functionally. 

    In the manner of proof-theoretic semantics, predicate bases provide an inferential alternative to this model-theoretic interpretation: we develop, in Section~\ref{sec:LogicalBasis}, an inferential counterpart to the concept of the 
    infon, called `inferon', in which predicate bases give meaning to relational assertions.  

    The starting point for this development will be bases defined over atoms of the form $\langle p , \mathfrak{b} \rangle$, 
    where $p$ is a proposition or predicate and $\mathfrak{b}$ is a polarity. This will be explained in Section~\ref{sec:LogicalBasis}.

    P-tS in its modern form can be considered to begin with Appendix~A of \cite{Prawitz1971ideas}, `Validity of Derivations', in which Prawitz discusses the validity of derivations in terms of bases. We refer to  this topic as proof-theoretic validity (P-tV). In the remainder of this section, we give a very brief summary of P-tV in the sense of Prawitz \cite{Prawitz1971ideas,Schroeder2006validity}. While this topic is an important part of the background to the present paper, it may be omitted by readers who wish to progress quickly to our core concerns. Note, however, that our final paragraph stresses an important point.      

   While there are a few variations of P-tV, we consider here a slightly generalized expression of the original presentation by Prawitz  \cite{Prawitz1971ideas}, starting from  Schroeder-Heister's elegant discussion in \cite{Schroeder2006validity,Schroeder2007modelvsproof}. 
    %The validity of proofs is reduced to so-called $S$-validity, where $S$ 
    The validity of proofs is reduced to so-called $\mathcal{B}$-validity, where $\mathcal{B}$ 
    is a base (i.e., a set of atomic rules), using the following auxiliary concepts:
    \begin{itemize}[label = --]
        % \item[--] a set of \emph{canonical derivations}, ${\bf C}$, typically derivations whose last step uses an introduction rule,
        \item[--] a set of \emph{canonical derivations}, ${\bf C}$;  for example, in systems of natural deduction, typically derivations whose last step uses an introduction rule, though 
        the notion of canonical can be much more general (see \cite{GP-Tactical2023}), 
        % \item[--] a justification $\justification{J}$ that maps  proof-structures to a proof-structures.
        \item[--] a justification $\justification{J}$ that maps proof-structures --- that is, proof-like 
        structures that are not necessarily proofs --- to proof-structures (this notion of proof-like structure is explained carefully in \cite{Schroeder2007modelvsproof} and also, more technically, in \cite{GheorghiuPym-Reductive2025}).
    \end{itemize}

    Then, $\mathcal{B}$-validity can be defined as follows:
    \begin{itemize}[label=--]
        \item[--] All derivations formed using atomic rules from $\mathcal{B}$ are $\mathcal{B}$-valid; 
        \item[--] A closed canonical derivation --- that is, 
        a canonical derivation that has no open assumptions --- is $\mathcal{B}$-valid; 
        \item[--] A closed non-canonical derivation is $\mathcal{B}$-valid if $\justification{J}$ reduces it to a canonical one; 
        \item[--] An open derivation $\derivation{D} \colon \psi_1, \dots, \psi_k \seq \varphi$ (i.e., $\derivation{D}$ has open assumptions among $\psi_1, \dots, \psi_k$ and conclusion $\varphi$) is $\mathcal{B}$-valid if, for any $\mathcal{B}' \supseteq \mathcal{B}$ and $\mathcal{B}'$-valid closed derivations $\derivation{D}_i$s, when we close each open assumption $\psi_i$ with $\derivation{D}_i$, the resulting derivation is $\mathcal{B}'$-valid.
    \end{itemize}
    
    Validity relative to a base $\mathcal{B}$ and a justification  ${\bf J}$, $\langle {\bf J} , \mathcal{B} \rangle$-validity is now defined as follows: a proof-structure $\Phi$ is $\langle {\bf J} , \mathcal{B} \rangle$-valid --- that is, represents a proof --- if either 
    $\Phi \in {\bf C}$ 
    or if ${\bf J}$ can be applied to $\Phi$ to yield an element 
    of ${\bf C}$. This definition gives rise to a 
    consequence relation as follows \cite{Schroeder2007modelvsproof}: 
    \begin{quote}
        Let $\langle {\bf J} , \mathcal{B} \rangle \Vdash \phi$  denote that ${\bf J}$ generates a valid closed proof of $\phi$ relative to $\mathcal{B}$. Then we obtain the following proof-theoretic notion of logical consequence: 
        \[
        \begin{array}{rcl}
            \mbox{$\phi_1 , \ldots , \phi_k \Vdash \phi$} & \mbox{iff} & \mbox{there is a ${\bf J}$ s.t., for every $\mathcal{B}$ and all ${\bf J}_1$ , \ldots , ${\bf J}_k$,} \\ 
            & & \mbox{if $\langle {\bf J}_1 , \mathcal{B} \rangle \Vdash \phi_1$, \ldots ,  $\langle {\bf J}_k , \mathcal{B} \rangle \Vdash \phi_k$, then 
            $\langle {\bf J} , \mathcal{B} \rangle \Vdash  \phi$}    
        \end{array} 
        \]
     \end{quote}

    If ${\bf J}$ specifies a reduction system --- as with normalization in natural deduction or BHK-semantics, as described in \cite{Prawitz1971ideas,Schroeder2007modelvsproof} --- $\langle {\bf J} , \mathcal{B} \rangle$-valid proofs can be defined inductively on their component structure. 
    
    Other choices of justifcation are, of course, possible and P-tV remains an active area of research. 

    Finally, we consider the role of arbitrary extensions $\mathcal{B}'$ of the current atomic system $\mathcal{B}$ in the $\mathcal{B}$-validity condition for open derivations. This is to prevent open derivations from being $\mathcal{B}$-valid simply due to that some of its open assumptions be  \emph{un}derivable in $\mathcal{B}$. It is helpful to read a base $\mathcal{B}$ as some `epistemic stage' (cf. Kripke's semantics of implication): an open derivation is $\mathcal{B}$-valid if at any future point one can close the open assumption with closed derivations, then the resulting closed derivation is also valid. Such a potential completion at a future point justifies the need for extensions of the current base $\mathcal{B}$.\footnote{In fact, Prawitz \cite{Prawitz1971ideas} points out that the extensions considered can be restricted to those required by ${\bf J}$ for the construction of $\psi$.} 

\subsection{Base-extension Semantics} \label{subsec:B-eS}

While P-tV deals with the validity of proofs, it does so without considering directly the validity of formul{\ae} (though, of course, validity of formul{\ae} could be indirectly derived by saying $\Gamma \seq \varphi$ is valid precisely when there exists some valid derivation $\mathcal{D} : \Gamma \seq \phi$).
The second line of development, which we call \emph{base-extension semantics} (B-eS), is concerned with the validity of formul{\ae} (and sequents).     
Sandqvist \cite{Sandqvist2015IPL} has given an elegant B-eS for intuitionistic propositional logic (IPL). This analysis demonstrates very clearly the basic principles of B-eS. Strong relationships between P-tV and B-eS in the setting of IPL have been discussed in \cite{GheorghiuPym2025}.   

The basic idea of such inductive definition of validity of formul{\ae} can be articulated conveniently by comparison with the model-theoretic definition of the validity of formul{\ae}. In the Kripke (and related) semantics of intuitionistic propositional logic based on ordered worlds, meaning is defined inductively as follows: the truth of an atomic formula at a world is determined by its interpretation in a model; that is, the atom is true at the world just in case the world is an element of the valuation of the atom. The meaning of the remaining connectives is then defined inductively, with the meaning of implication formul{\ae} requiring, analogously to the requirement for base-extensions described above, judgements relative to worlds higher in the ordering. In fact, looking at the current world instead will incur the vacuous satisfaction problem similar to that in P-tV as mentioned at the end of Section~\ref{subsec:P-tV}: $\varphi \supset \psi$ could be true at a world $w$ simply because $\varphi$ is false at $w$ (but not necessarily false at all worlds greater than $w$).

By contrast, base-extension semantics gives the meaning ($\Vdash_\base{B}$) of atomic formul{\ae} relative to a base $\base{B}$ in terms of provability ($\vdash_\base{B}$) in $\base{B}$, where the underlying set of atoms is assumed to be denumerably infinite:
% \footnote{We move to the notation $\base{B}$, etc., for bases to stress the move from P-tV to B-eS.} 
    \[
        \Vdash_\base{B} p \quad\mbox{iff}\quad 
        \vdash_\base{B} p
    \]
The meanings of the connectives are then given inductively, as in Kripke's semantics.

Atomic derivability in a base $\baseB$ of rules of the form given in Figure~\ref{fig:base} is defined using the combinators (Ref) and (App), as follows: 
\[
\begin{array}{rl}
 \mbox{(Ref)} & \mbox{$P , p \vdash_\base{B} p$} \\ 
(\mbox{App}_\atRule{R}) & \mbox{if $\atRule{R} \in \base{B}$ and, for all $i \in [1,n]$, where $n \geq 0$, $Q , P_i \vdash_\base{B} q_i$, then $Q \vdash_\base{B} r$} 
\end{array}
\]
Note that the rule $\atRule{R}$ encompasses the axiom and 
hypothesis-free rules ($n=0$ and $P_i$s empty, respectively). For brevity, the rule $\atRule{R}$ in Figure~\ref{fig:base} is also written inline as
\[
((P_1 \Rightarrow q_1) , \ldots , (P_n \Rightarrow q_n)) \Rightarrow r)
\]

\begin{figure}[t]
\hrule 
\[
 \dfrac{\qquad}{p} \qquad 
    \dfrac{p_1 \ldots p_n}{p} \qquad \mbox{and} \qquad 
\frac{\begin{array}{ccc}
        [P_1] &        & [P_n] \\
        q_1   & \ldots & q_n  
      \end{array}
}{r} 
\]
\hrule 
\caption{Base rules $\mathcal{R}$: atomic, level 1, and level 2, 
            respectively} 
            \label{fig:base}
\end{figure}
\noindent There is a substitution (cut) operation on bases that maps derivations $P \vdash_\baseB p$ and $Q , p \vdash_\baseB q$
to a derivation $P , Q \vdash_\baseB q$ \cite{Sandqvist2015IPL}.

Before proceeding to the remainder of the discussion that we have sketched, we first recall and fix some notational conventions that we shall need throughout this section. These are compatible with Sandqvist \cite{Sandqvist2015IPL}, Gheorghiu \cite{GP2025-P-tS-S-oL}, and Gheorghiu and Pym \cite{GheorghiuPym2025} and are summarized in Figure~\ref{fig:notation}. 

\begin{figure}[ht]
\hrule 
\vspace{1mm}
\begin{itemize}
\item[--] $p$, $q$, $r$, ... denote propositional letters 
%\item[--] $P$, $Q$, $R$, ... denote sets of atomic formulae 
\item[--] $\mathfrak{b}$, $\mathfrak{c}$, 
    $\mathfrak{d}$, ... denote polarities 
\item[--] $x$, $y$, $z$, ... denote variables
\item[--] $a$, $b$, $c$, ... denote constants 
\item[--] $f$, $g$, $h$, ... denote function symbols 
\item[--] $t$, $s$, $r$, ... denote terms
\item[--] $P$, $Q$, $R$, ... denote either propositional letters or atomic predicates
\item[--] $\phi$, $\psi$, $\chi$, ... denote   
    formul{\ae}
\item[--] $\mathbb{P}$, $\mathbb{Q}$, $\mathbb{R}$, ... denote (possibly empty, possibly countably infinite) sets of atoms --- 
propositional letters and atomic predicates (note this is a minor though wholly compatible variation on Sandqvist's notational scheme, as used in Section~\ref{subsec:B-eS})
\item[--] $\Gamma$, $\Delta$, $\Theta$, ... denote (possibly empty, possibly infinite) sets of formul{\ae}
\item[--] $\base{B}$, $\base{C}$, $\base{D}$, ... denote atomic systems
\item[--] $\mathfrak{P}$, $\mathfrak{Q}$, $\mathfrak{R}$, ... denote (possible empty, possibly infinite) sets of inferonic atoms (see Definition \ref{def:inf-atom})
\end{itemize} 
The sets of all variables, constants, terms, propositional atoms, and formul{\ae} are $V$, $C$, $T$, $A$, and $F$, respectively; the subsets containing only closed elements (where appropriate) are denoted {\rm Cl}(-). We write $FV$ to denote the function that takes a term, atom, or formula and returns the set of its free variables. We write quantified formul{\ae} as $\forall{x}.\phi$ and $\exists{x}.\phi$, etc., and let $[t/x]$ to denote a substitution
function that replaces all free occurrences of $x$ by $t$. 
\vspace{1mm}
\hrule
\caption{Notational conventions}
\label{fig:notation}
\end{figure}

We use $\bot$, $\wedge$, $\vee$, $\supset$, $\forall$, $\exists$ as the logical signs. We do not regard $\bot$ 
as an atomic formula; intuitively, following Dummett \cite{Dummett1991}, it has logical structure that renders it logically/semantically complex even though it is indeed syntactically atomic. In this intuitionistic setting, we elide $\top$. 

We write `Atoms' for the (denumerably infinite) set of propositional and predicate atomic formul{\ae} in a specified context.  Then the propositional and first-order formul{\ae} are given by the usual grammar: 
\[
    \begin{array}{rcl}
    \phi & ::= & P \mid \bot \mid \phi \wedge \phi \mid \phi \vee \phi \mid \phi \supset \phi \mid \forall{x}.\phi \mid \exists{x}.\phi 
    \end{array}
\]
where atomic predicates defined over terms $f(t_1,\ldots,t_k)$ in the usual way. Terms and predicates are \emph{closed} if they have no free variables. 

The B-eS clause for disjunction can be given in the form that is familiar from Kripke semantics, 
\[
 \Vdash_\mathcal{B} \phi_1 \vee \phi_2 
        \mbox{ iff }  \mbox{$\Vdash_\mathcal{B} \phi_1$ or $\Vdash_\mathcal{B} \phi_2$}
\]
but such a choice leads to a failure of completeness \cite{Piecha2016completeness,Piecha2019incompleteness,Sandqvist2015IPL}.

\begin{figure}[ht]
\hrule
\[
    \begin{array}{lrcl} 
       (\mbox{At}) & \Vdash_\mathcal{B} P   & \mbox{iff} & 
       \mbox{$\vdash_{\mathcal{B}} P$ for closed $P$} \\ 
       (\wedge) & \Vdash_\mathcal{B} {\phi}_1 \wedge {\phi}_2 & \mbox{iff} & \mbox{$\Vdash_\mathcal{B} {\phi}_1$ and $\Vdash_\mathcal{B} {\phi}_2$} \\ 
       (\vee) & \Vdash_\mathcal{B} {\phi}_1 \vee {\phi}_2 
       & \mbox{iff} & \mbox{for every closed 
       $P$ and every 
       $\mathcal{C} \supseteq \mathcal{B}$,} \\   
       & & & \mbox{if $\phi_1  \Vdash_\mathcal{C} P$ and $\phi_2 \Vdash_\mathcal{C} P$, then $\Vdash_\mathcal{C} P$} \\
       (\supset) & \Vdash_\mathcal{B} {\phi}_1 \supset {\phi}_2 & \mbox{iff} & \mbox{${\phi}_1 \Vdash_\mathcal{B} {\phi}_2$} \\ 
       (\mbox{Inf}) & \mbox{for $\Theta \neq \emptyset$, 
        $\Theta \Vdash_\mathcal{B} {\phi} $} & \mbox{iff} & 
        \mbox{for every $\mathcal{C} \supseteq \mathcal{B}$, if $\Vdash_\mathcal{C} {\psi}$, for every ${\psi} \in \Theta$,} \\ 
        & & & \mbox{then $\Vdash_\mathcal{C} {\phi}$} \\ 
        (\bot) & \Vdash_\mathcal{B} \bot  
        & \mbox{iff} & \mbox{for all closed $P$, $\Vdash_\mathcal{B} P$} 
    \end{array}
\]

\dotfill 
\[
\begin{array}{lrcl}   
(\forall) &  \Vdash_\mathcal{B} \forall{x}.\phi & \mbox{iff} & 
     \mbox{$\Vdash_\mathcal{B} \phi[t/x]$ for all $t \in {\rm Cl}({T})$} \\
(\exists) &  \Vdash_\mathcal{B} \exists{x}.\phi & \mbox{iff} & 
  \mbox{for any $\mathcal{C} \supseteq \mathcal{B}$ and 
  any closed $P$, if} \\ 
   & & & \mbox{$\phi[t/x] \Vdash_\mathcal{C} P$, for any $t \in {\rm Cl}({T})$, then $\Vdash_\mathcal{C} P$} 
\end{array}
\] 
  \dotfill 
\[
\begin{array}{lrcl}
(\mbox{Validity}) & \Gamma \Vdash \phi & \mbox{iff} & 
    \mbox{for all $\baseB$, $\Gamma \Vdash_\baseB \phi$}
\end{array}
\]
\hrule 
\caption{Sandqvist's B-eS for intuitionistic propositional logic and Gheorghiu and Pym's extension of it  to first-order} 
\label{fig:support1}
\end{figure} 

While  Figure~\ref{fig:support1}, which gives the inductively defined semantics, looks similar to the Kripke semantics of IPL, there are some key points to note: 
\begin{itemize}
% [leftmargin=6mm,label=--]
% \setlength\itemsep{-0.5mm}
% [leftmargin=6mm,label=--]
\setlength\itemsep{-0.5mm}
\item[--] The use of base-extension in the (Inf) clause reflects Prawitz's argument for its necessity in giving meaning to implication \cite{Prawitz1971ideas}, the same reasons underlying the consideration of extensions in P-tV (see the last paragraph of Section~\ref{subsec:P-tV}).
\item[--] Base-extension is transmitted to the implication connective $\supset$ through its clause $(\supset)$. 
\item[--] The form of the clause for $\bot$ may seem a little odd at first sight, but it gives the usual intuitionistic introduction and elimination rules for
negation \cite{Prawitz2005}, defined as $\neg \phi = \phi \supset \bot$, as well as Ex Falso Quodlibet \cite{Dummett1991,Sandqvist2015IPL,Sandqvist2025}. Note that the set of atoms ($p$s) is assumed to be denumerably infinite. 
\item[--] The form of the clause for disjunction is critical: it is this form that allows the completeness theorem (q.v., below) --- see, for example,  \cite{Sandqvist2015IPL,PymRitterRobinson2024,GheorghiuPym2025} for explanations of this, which are beyond the scope of this brief introduction. In fact, the typical way of defining that a disjunction is true iff at least one of the disjuncts is true results in incompleteness~\cite{Piecha2019incompleteness}. 
\end{itemize} 
Given all this, \emph{validity} (denoted as $\Vdash$) is defined as support in all bases: $\Gamma \Vdash \phi$ iff, for all $\base{B}$, $\Gamma \Vdash_\base{B} \phi$. 
Note that we have transitivity for $\Vdash$ as given in Figure~\ref{fig:support1}: if $\Gamma \Vdash \phi$ and $\Gamma , \phi \Vdash \psi$, then $\Gamma \Vdash \psi$. 

We write $\Gamma \vdash \phi$ to denote that $\phi$ is provable from $\Gamma$ in the natural deduction systems NJ (cf. \cite{Sandqvist2015IPL}). For reference, NJ, in sequential form, is given in Figure~\ref{fig:nj}. 

\begin{figure}[ht]
\hrule 
\vspace{2mm}
\[
\begin{array}{c@{\quad\quad}c}
\infer[(Ax)]{\Gamma , \phi \vdash \phi}{} 
    & \infer[\bot{E}]{\Gamma \vdash \phi}{\Gamma \vdash \bot} \\ [4pt] 
\infer[\wedge{I}]{\Gamma \vdash \phi \wedge \psi}{\Gamma \vdash \phi & \Gamma \vdash \psi}
    &  \infer[\wedge{E}]{\Gamma \vdash \phi}{\Gamma \vdash \phi \wedge \psi} \quad 
            \infer[\wedge{E}]{\Gamma \vdash \psi}{\Gamma \vdash \phi \wedge \psi} \\ [4pt] 
\infer[\vee{I}]{\Gamma \vdash \phi \vee \psi}{\Gamma \vdash \phi} \quad 
            \infer[\vee{I}]{\Gamma \vdash \phi \vee \psi}{\Gamma \vdash \psi}     
    &  \infer[\vee{E}]{\Gamma \vdash \chi}
            {\Gamma \vdash \phi \vee \psi & \Gamma, \phi \vdash \chi &
              \Gamma, \psi \vdash \chi} \\ [4pt]
\infer[\supset{I}]{\Gamma \vdash \phi \supset \psi}{\Gamma , \phi \vdash \psi}
    & \infer[\supset{E}]{\Gamma \vdash \psi}{\Gamma \vdash \phi & \Gamma \vdash \phi \supset \psi} \\ [4pt]  
\end{array}
\]
\dotfill
\[ 
\begin{array}{c@{\quad\quad}c}
\dfrac{\Gamma \vdash \phi(x)}{\Gamma \vdash \forall{x}.\phi(x)}\quad\mbox{($x$ not free in $\Gamma$)}\quad{\forall{I}} & 
    \dfrac{\Gamma \vdash \forall{x}.\phi(x)}{\Gamma \vdash \phi[t/x]}\quad \mbox{($t$ free for $x$ in $\phi$)}\quad{\forall{E}} \\ 
    & \\ 
\dfrac{\Gamma \vdash \phi[t/x]}{\Gamma \vdash \exists{x}.\phi(x)}\quad\mbox{($t$ free for $x$ in $\phi$)}\quad{\exists{I}} & 
    \dfrac{\Gamma \vdash \exists{x}.\phi(x) 
    \quad \Gamma , \phi \vdash \psi}{\Gamma \vdash \psi}\quad\mbox{($x$ not free in $\Gamma$ or $\psi$})\quad{\exists{E}} \\ [4pt]   
\end{array}
\]

\vspace{1mm}
\hrule 
\caption{The calculus NJ in sequential form (eliding $\top$)   \cite{Prawitz2005,TroelstraSchwichtenberg2000,Girard1989}
% with the axiom (\emph{Inferon}) for the theory of inferons.
}
% \note{Prawitz's form is additive in the metatheory (constant context $\Gamma$ throughout. Could we use something multiplicative instead eg Dummett Elements of Intuitionism, p 123.?}
% \note{Was following Sandqvist, so that we have a fixed reference point for the metatheory ... .}
\label{fig:nj}
\end{figure}
Sandqvist \cite{Sandqvist2015IPL} establishes the soundness and completeness of the natural deduction calculus NJ \cite{Gentzen1934,Prawitz2005} with respect to this semantics. 
Gheorghiu and Pym \cite{GP2025-P-tS-S-oL} extend  soundness and completeness to first- and second-order intuitionistic logic. In the notation of Figure~\ref{fig:nj}, we obtain soundness and completeness for validity. 

\begin{theorem}[Sandqvist \cite{Sandqvist2015IPL} and Gheorghiu \cite{GP2025-P-tS-S-oL}] 
% For the pure predicate logic without \emph{Inferon}, we have 
$\Gamma \vdash \phi$ iff $\Gamma \Vdash \phi$. \fillBox
\end{theorem}

\begin{theorem}[Sandqvist \cite{Sandqvist2015IPL}] 
For the propositional intuitionistic logic, %without \emph{Inferon}, we have 
$\Gamma \vdash \phi$ iff $\Gamma \Vdash \phi$. \fillBox
\end{theorem}

We conjecture that this result extends to the intuitionistic first-order setting --- see, for example, \cite{Sandqvist2005inferentialist,Sandqvist2009CL,Nascimento2023,GP2025-P-tS-S-oL} for some extant approaches. This is work in progress. 

It is base-extension semantics that provides the logical basis for the use of bases to provide the inferential alternative to infons that we call inferons. In this setting, developed in Section~\ref{sec:LogicalBasis}, we need to define base rules of atoms of the form $\langle p , \mathfrak{b} \rangle$, where $p$ is a propositional or predicate atom and $\mathfrak{b}$ is a boolean polarity.\footnote{For now, we take polarities to be Boolean, but we conjecture that other choices are of substantive interest.}

\section{An Inferentialist Logical Basis for Information} \label{sec:LogicalBasis}

With the background ideas and techniques in place, we now employ base-extension semantics to establish an inferentialist logic of information. Our set-up follows the base-extension semantics for intuitionistic predicate logic as described in Section~\ref{subsec:B-eS}. However, the key distinction is that atomic predicates are replaced by \emph{inferons}, our inferentialist version of infons.  

It should be noted that our choice to ground our account of information on an intuitionistic logic of inferons, building on the B-eS for intuitionistic logic \cite{Sandqvist2015IPL,GP2025-P-tS-S-oL} as described above, is a wholly pragmatic one. It does not represent a commitment or suggestion that intuitionistic logic as presented. Rather, it is the simplest viable point of departure. 

Borrowing some notation from situation theory, we take \emph{inferons} to be of the form 
\[
    \inferon{R(t_1,...,t_n)}{B}{b}{} 
\]
This stands in contrast to form of \emph{infons} 
taken by Barwise, Devlin et al. (see \cite{Devlin1991}), namely 
\[
    \inferatom{R(t_1,...,t_n)}{b}{} 
\] 
The latter gains meaning through the assumption of the availability of a notion of truth for predicates $R(t_1,...,t_n)$, 
which can be obtained only from a presumed model within which predicates can be interpreted.\footnote{Note that while, following situation theory, we retain the presence of a polarity $\mathfrak{b}$ within the structure of an inferon, it is not clearly that its presence is strictly essential, at least theoretically. However, their presences seems to allow us to avoid, to begin with, some delicate logical issues around negation. Furthermore, as we shall see in Section~\ref{sec:flow}, for example, polarities are useful in the context of modelling systems inferonically. Moreover, both theoretically and for modelling purposes, the question of the possible value of richer types of polarity is intriguing.}   

%\note{for inferons, stress between from inference in local u global base}

Our position is that we seek to replace such a presumption with an account of meaning that is inferential. In the present
setting, we seek to achieve this through proof-theoretic semantics and, in particular, through base-extension semantics. In this view, as we have explained in Section~\ref{subsec:B-eS}, the meaning of propositions --- specifically, their validity --- is given by the idea of derivability in a base of atomic rules. So, an atomic predicate $R(t_1,...,t_n)$ is valid in a base $\mathcal{B}$ iff it
is derivable in $\mathcal{B}$: 
\[
    \Vdash_\baseB R(t_1,...,t_n)  \quad\mbox{iff}\quad
        \vdash_\baseB R(t_1,...,t_n)
\] 
So, the ambient, extensional notion of truth presumed in the infon \cite{Devlin1991} 
$\inferatom{R(t_1,...,t_n)}{b}{}$ is replaced by derivability in a base.  

At this point, we have two options:
\begin{itemize}
\item[--] we can presume a global base to be used to give meaning 
    to all atomic predicates, which is analogous to the use of an implicit model to give meaning to infons, or  
\item[--] we can associate a local base with an atomic predicate  
    and so obtain our notion of inferon.  
\end{itemize}
We choose the latter. Why? The former is of course possible, but, we would argue, simply not as useful or as convenient as an account of information. 
%(I'm taking the role of polarities as given for the purpose of this discussion.)
From the inferentialist  perspective, it is natural for an assertion to be accompanied by an account --- albeit a context-sensitive one --- of from where it is intended to derive its meaning. Indeed, it might be argued that the absence of such an account is a weakness of the set-up of
Barwise et al. \cite{Devlin1991}. While it is in principle possible to address this by associating with each assertion
a model from which it acquires its meaning, we would argue that that would represent a large, perhaps
unmanageable, overhead in modelling information-theoretic scenarios.
% Indeed, it might be argued that the absence of such an account is a weakness of the set-up of Barwise et al. \cite{Devlin1991} that cannot be addressed there without associating with each assertion a model from which it acquires its meaning. While this is possible in principle, we would argue that is represents a large, perhaps unmanageable, overhead in modelling information-theoretic scenarios. 
By contrast, the inferentialist set-up can fix this weakness simply by associating a base, which, for practical modelling purposes can be considered essentially finite.

Furthermore, with this set-up, we can now naturally ask how the meaning of a predicate in this setting is affected by its context, as represented by a base $\baseC$; that is,
\[
    \Vdash_\baseC \inferon{P}{B}{b}{}  \quad\mbox{iff}\quad \vdash_{\mathcal{C} \,\cup\, \mathcal{B}} \inferatom{P}{b}{} 
\]
in which the local and contextual bases are combined. Indeed, we would stress that inferons have no inherent meaning. In Dretske's terms, they are `objective commodities' and meaning is associated with them only through their use in inferential contexts.  The force of this is particularly clear if $\inferatom{P}{b}{}$
is not derivable in $\baseB$, but is derivable with the addition of $\baseC$. 

Moreover, this set-up is naturally aligned with a key concept in the metaphysics of information as we see it: that of agency. Agents --- the literature is vast, we make no attempt here to review it, though \cite{vanBenthemLIA2025} (see also \cite{vanBLDII2011}) provides a comprehensive, elegant, and insightful account of the state of the art  --- can naturally be taken to have inferential capacities. In our setting, we associate with an agent $a$ set of bases ${\bf Bases}(a)$ representing its inferential capacities in various contexts --- an agent may have different abilities in different contexts, such as different language environments 
--- and capture this in the associated inferential semantics. Here $\iota$ denotes an inferon. So,  
\[
\begin{array}{rcl}
\Vdash_\baseC [a] \iota & \mbox{iff} & 
    \mbox{for all $\baseA \in {\bf Bases}(a)$, $\Vdash_{\mathcal{C} \,\cup\, \mathcal{A}} \iota$} 
\end{array} 
\] 
and
\[
\begin{array}{rcl}
\Vdash_\baseC \langle a \rangle \iota & \mbox{iff} & 
    \mbox{for some $\baseA \in {\bf Bases}(a)$, $\Vdash_{\mathcal{C} \,\cup\, \mathcal{A}} \iota$} 
\end{array} 
\]
and so on.

The set-up so far notwithstanding, one might ask whether the inferential content of inferons --- that is, the base --- is essential. We can see that it is from two points of view. First, from the perspective of inferentalist metaphysics, we have argued that the meaning of information should be determined by how it is intended to be used, which is determined by its base. Second, from a technical point of view, one might ask whether all the local inferential capacities of a given collection of inferons should be added to the underlying base, as in the first option mentioned above. To do so would impair the expressivity of the set-up by removing the local--global distinction between specific units of information/agents and the environment with which they are understood. Very simple examples will illustrate this point. 

\begin{example}[Intelligence sharing~1] 
\label{ex:intel-share}
Consider a group of allies who share intelligence and suppose that they have a shared information repository.  

\begin{itemize}[label=--] 
\item $\baseB$ describes the repository of shared intelligence
\item $\baseC$ describes the informational capabilities of an allied agent $c$; that is, $\baseC \in {\bf Bases}(c)$ 
\item $\baseD$ describes the informational capabilities of a hostile agent $d$: that is, $\baseD \in {\bf Bases}(d)$. 
\end{itemize} 
Given these assumptions, we have that
\begin{itemize}[label=--] 
\item $\baseB \cup \baseC$ describes the information that can be obtained from the 
repository by the allied agent
\item $\baseB \cup \baseD$ describes the information that can be obtained from the 
repository by the hostile agent.
\end{itemize}
Thus $\baseB \cup \baseC$ and $\baseB \cup \baseD$ describe distinct situations with, typically, distinct properties that can be expressed relative to the distinct 
bases $\baseB \cup \baseC$ and $\baseB \cup \baseD$ through, for example, the support-relation clauses for 
$\langle a \rangle \iota$ and $[a] \iota$, as described above. 
\fillBox 
\end{example}

\begin{example}[Intelligence sharing~2] 
\label{ex:intel-share-2}
We can unpack Example~\ref{ex:intel-share} a little further. 
Suppose that agent $c$'s capability, $\baseC$, is added to 
the repository's base $\baseB$ so as to give access to the 
inferon $\iota = \inferon{P}{E}{b}{}$, which, in addition to its 
logical assertion $P$, carries the inferential capacity $\baseE$ 
that is associated with $P$ --- that is, its local inferential context. Thus the inferential capacity against which agent $c$'s 
ability to validate $P$, the propositional part of the inferon, 
includes not only that of the repository and of $c$, but also $E$. 
\fillBox 
\end{example}

So, while it is possible to reduce the set-up to employ fixed global bases, it seems conceptually wrong and technically inconvenient to do so, especially from a modelling perspective (to which we shall return in the sequel): the expressiveness  of the set-up --- in particular, its natural ability to capture context-sensitivity --- would be 
lost.  
%significantly impaired. 

% \note{Propositional language, predicate language: do we need to spell out the grammar? That would mean we could write 'Atoms($\mathcal{L}$)', and so on.}  
% \note{Need notational consistency with Gheorghiu 
% as well as later in this paper}

\begin{definition}[Inferonic atoms] \label{def:inf-atom}
An \emph{inferonic atom} is a pair $\inferatom{P}{b}{}$, where $P$ denotes a closed proposition (i.e., a propositional letter or a closed atomic predicate) and $\mathfrak{b}$ denotes a \emph{polarity}, taking the values $0$ or $1$. 
% An inferonic atom $\inferatom{P}{b}{}$ is closed if $P$ has no free variables. 
\fillBox
\end{definition}

The polarities associated with propositions correspond to the `assertions' and `denials', in the sense of Rumfitt \cite{Rumfitt2000}, as used by 
Gheorghiu and Buzoku \cite{GheorghiuBuzoku2025} 
in giving a Sandqvist-style base-extension 
semantics for classical propositional logic.  

% \note{Perhaps we should re-brand polarities as `assertion' and `denial', as in \cite{BuzokuGheorghiu} --- more natural in this inferential setting?}

In the base-extension semantics for intuitionistic 
propositional logic \cite{Sandqvist2015IPL},\footnote{Indeed, also in the base-extension semantics for many other, though by no means all, logics.} bases are defined over propositional letters, $p$. In establishing a logic 
of inferons, we need to define bases over inferonic atoms, $\inferatom{P}{b}{}$. We use $\mathfrak{P}$ to denote sets of inferonic atoms. Figure~\ref{fig:inferonic-base} gives the form of inferonic base rules that we consider. 

\begin{definition}[Inferonic base] \label{def:inf-base} An \emph{inferonic base} $\mathcal{B}$ is a base of rules over a denumerable set of \emph{inferonic atoms}, as given in Figure~\ref{fig:inferonic-base}. \fillBox
\end{definition}

In using the logic we develop below --- q.v. 
Figure~\ref{fig:support2} --- as a tooling for modelling an information-theoretic phenomenon, we can think of inferonic bases as the \emph{basic situation} describing the inherent inferential properties of the phenomenon. Then supported inferons $\inferon{P}{P}{b}{}$ provide additional inferential content. Later, we shall see how we can incorporate the local reasoning capabilities of agents through support relations for agent-modalities.

% \note{Rewrite as 'think of': In the context of setting up situations, we shall often refer to inferonic bases as \emph{basic situations}. In a specified context, we refer to the set `BasicSituations' of basic situations.} 

\begin{figure}[ht]
\hrule
\[
 (1) \quad \dfrac{\qquad}{\inferatom{P}{b}{}} \qquad
 (2) \quad \dfrac{\inferatom{P}{b}{1} \ldots \inferatom{P}{b}{k}}{\inferatom{P}{b}{}} \qquad (3) \quad
            \dfrac{
                \begin{array}{ccc} 
                [\mathfrak{P}_1] & & [\mathfrak{P}_k] \\  
                & \ldots & \\ 
                \inferatom{P}{b}{1} & & 
                \inferatom{P}{b}{k}
                \end{array}}
            {\inferatom{P}{b}{}} 
\]
\hrule 
\caption{Inferonic base rules, $\mathcal{R}$} 
\label{fig:inferonic-base}
\end{figure}

Other forms of bases rules are also possible (e.g., \cite{SandqvistWLD2022}). 
It could be suggested that the presence of polarities in inferonic atoms constitutes a degree of `semantic pollution' \cite{Read2015}. We would argue that our set-up lies within the scope of Avron's criterion \cite{Avron1996} for acceptability, that 
\begin{quote}
` \ldots the framework should be independent of any particular semantics. One should not be able to guess, just from the form of the structures which are used, the intended semantics of a given proof-system \ldots\, .'
\end{quote}

\begin{definition}[Inferonic basis] \label{def:inferonic-basis} 
An \emph{inferonic basis} $\bf{B}$ is a denumerable set of inferonic bases. \fillBox
\end{definition}

% A \emph{basic situation} $\mathcal{S}$ is a base over denumerable set of \emph{inferonic atoms}, $\langle p , \mathfrak{b} \rangle$, where $p$ denotes a proposition and $\mathfrak{b}$ denotes a \emph{polarity}, taking the values $0$ or $1$. 

% Let's call such a base \emph{inferonic}. An 
% inferonic basis $\bf{B}$ is a denumerable set of inferonic bases. 

Note that atoms may be propositional or predicates, following 
Gheorghiu and Pym's set-up \cite{GP2025-P-tS-S-oL}. %Note that we follow Sandqvist's notational conventions \cite{Sandqvist2015IPL}. 

%\note{refer to Figure XYX, above, and check for other instances} 

\begin{definition}[Inferon] \label{def:inferon}
An % a basic 
\emph{inferon} is a triple $\iota = \inferon{P}{P}{b}{}$, where $P$ is closed atomic predicate, $\mathcal{P}$ is an inferonic base, and $\inferatom{P}{b}{}$ is an inferonic atom. We let $\iota$ range over inferons. $\inferon{\phi}{P}{b}{}$ is closed if $P$ is closed. \fillBox
\end{definition}

% \note{We need consistent names to be able to distinguish these basic inferons (triples) from compound inferons made from the inferon connectives. I (MC) called them `basic' inferons in the Inferomorphism section. They are the `atomic' inferons in Figure 4 (support relation), so we have atomic inferons (the triples) and inferonic atoms (the pairs)?}

An inferonic base, $\mathcal{B}$, of rules of the form given in Figure~\ref{fig:inferonic-base}, generates a basic derivability relation $\vdash_\mathcal{B}$.  This relates a finite set of inferonic atoms, $\mathfrak{P}$ to a single inferonic atom:
\begin{equation*}
    \mathfrak{P} \vdash_\mathcal{B} \inferatom{P}{b}{}
\end{equation*}

\begin{definition}[A logic of inferons]
\label{def:def:logic-inferons}
For a given inferonic base $\mathcal{B}$, a \emph{logic of inferons} is given by a support relation for validity. Figure~\ref{fig:support2}  
gives the support relation for logics or inferons corresponding, when the polarities are all $1$ to intuitionistic propositional and predicate logics).  
\fillBox
\end{definition}

The logical connectives defined by the support relation in Figure~\ref{fig:support2} construct complex inferons --- henceforth 
called just inferons --- from basic inferons. 

\begin{definition}[Inferential situation]
\label{def:situation}
An \emph{inferential situation} is then a set of inferons and support for an  inferon by an inferential situation is given by the support relation given in Figure~\ref{fig:support2}. \fillBox
\end{definition}

In the sequel, where no confusion is likely, we shall frequently contract `inferential situation' to just `situation'. 

\begin{definition}[Derivations in a base]
For inferonic base rules of the form given in Figure~\ref{fig:inferonic-base}, here for convenience 
written as 
\[
\frac{\begin{array}{ccc}
        [\mathfrak{P_1}] &        & [\mathfrak{P}_k] \\
        & \ldots & \\
        \inferatom{Q}{b}{1}   
        & & \inferatom{Q}{b}{k} \rangle  
      \end{array}
}{\inferatom{R}{c}{}} \, \atRule{R}
\]
where $\mathfrak{P}$s denote sets of inferonic atoms, derivations are constructed using the following combinators: 
\[
\begin{array}{rl}
 \mbox{(Ref)} & \mbox{$\mathfrak{P} , \inferatom{P}{b}{} \vdash_\base{B} \inferatom{P}{b}{}$} \\ 
(\mbox{App}_\atRule{R}) & \mbox{if $\atRule{R} \in \base{B}$ and, for all $i \in [1,n]$, $\mathfrak{Q} , \mathfrak{P}_i \vdash_\base{B} \inferatom{Q}{b}{i}$, then $\mathfrak{Q} \vdash_\base{B} \inferatom{R}{c}{}$} 
\end{array}
\]
\fillBox
\end{definition}

\noindent As in the purely propositional setting, there is a substitution (cut) operation on base derivations. 

\medskip 

\begin{definition}[Consistent base]
\label{def:consistentInferonicBase}
A base $\mathcal{B}$ is \emph{consistent} if, for all $P$ in the given language, it is not the case that both $\vdash_\mathcal{B} \inferatom{P}{0}{}$ 
and $\vdash_\mathcal{B} \inferatom{P}{1}{}$. \fillBox
\end{definition}
%\note{What's the relationship between this definition of consistency and the definition of $\bot$? What should it be?}

Many different logics of inferons can be defined using support relations that build upon inferonic bases.  As we have remarked, inferonic bases over sets/algebras of polarities other than just $\mathfrak{0}$ and $\mathfrak{1}$ can of course also be considered. We begin with an intuitionistic logic of inferons 
--- similar to Sandqvist's treatment of IPL \cite{Sandqvist2015IPL}, with the support relation as given in Figure~\ref{fig:support1}.\footnote{The same support relation gives both intuitionistic and classical logics of inferons. Classical logic is obtained if bases are restricted to rules of types (1) and (2) and intuitionistic logic is obtained if rules of type (3) are permitted.} 

Later, we extend this to first-order following \cite{GP2025-P-tS-S-oL}.

\begin{figure}[ht]
\hrule
\[
    \begin{array}{lrcl} 
       (\mbox{At}) & \Vdash_\mathcal{B} \inferon{P}{P}{b}{}   & \mbox{iff} & 
       \mbox{$\vdash_{\mathcal{B} \,\cup\, \mathcal{P}} 
       \inferatom{P}{b}{}$, for closed atomic predicates $P$} \\ 
       (\wedge) & \Vdash_\mathcal{B} {\phi}_1 \wedge {\phi}_2 & \mbox{iff} & \mbox{$\Vdash_\mathcal{B} {\phi}_1$ and $\Vdash_\mathcal{B} {\phi}_2$} \\ 
       (\vee) & \Vdash_\mathcal{B} {\phi}_1 \vee {\phi}_2 
       & \mbox{iff} & \mbox{for every inferon 
       ${\iota}$ and every 
       $\mathcal{C} \supseteq \mathcal{B}$,} \\   
       & & & \mbox{if $\phi_1  \Vdash_\mathcal{C} \iota$ and $\phi_2 \Vdash_\mathcal{C} \iota$, then $\Vdash_\mathcal{C} \iota$} \\
       (\supset) & \Vdash_\mathcal{B} {\phi}_1 \supset {\phi}_2 & \mbox{iff} & \mbox{${\phi}_1 \Vdash_\mathcal{B} {\phi}_2$} \\ 
       (\mbox{Inf}) & \mbox{for $\Theta \neq \emptyset$, 
        $\Theta \Vdash_\mathcal{B} {\phi} $} & \mbox{iff} & 
        \mbox{for every $\mathcal{C} \supseteq \mathcal{B}$, if $\Vdash_\mathcal{C} {\psi}$, for every ${\psi} \in \Theta$,} \\ 
        & & & \mbox{then $\Vdash_\mathcal{C} {\phi}$} \\ 
        (\bot) & \Vdash_\mathcal{B} \bot  
        & \mbox{iff} & \mbox{for all closed ${\iota}$, $\Vdash_\mathcal{B} {\iota}$} 
    \end{array}
\]

% \dotfill 

% \[
% \begin{array}{lrcl}     
% (\forall) &  \Vdash_\mathcal{B} \forall{x}.\phi & \mbox{iff} & 
%      \mbox{$\Vdash_\mathcal{B} \phi[t/x]$ for all $t \in {\rm Cl}({\mathcal{T}})$} \\
% (\exists) & \Vdash_\mathcal{B} \exists{x}.\phi & \mbox{iff} & 
%   \mbox{for any $\mathcal{C} \supseteq \mathcal{B}$ and all closed $\iota$, if} \\ 
%   & & & \mbox{$\phi[t/x] \Vdash_\mathcal{C} \iota$, for all $t \in {\rm Cl(\mathcal{T})}$, then $\Vdash_\mathcal{C} \iota$} 
% \end{array}
% \]

  \dotfill 
\[
\begin{array}{lrcl}
(\mbox{Validity}) & \Gamma \Vdash \phi & \mbox{iff} & 
    \mbox{for all $\baseB$, $\Gamma \Vdash_\baseB \phi$}
\end{array}
\]

\hrule 
\caption{Support relation for inferons: propositional case} 
\label{fig:support2}
\end{figure}  

The base-extension semantics presented here very closely corresponds that for intuitionistic logic, the difference being the inferonic structure of propositions. 
For the soundness and completeness of this theory of inferons, we consider the derivability relation 
given by NJ with inferons as atomic formul{\ae} and with the ({Inferon}) 
axiom, as given in Figure~\ref{fig:inferon-ax}. Note that $\Vdash$ is transitive: that is, if 
$\Gamma \Vdash \phi$ and $\Gamma , \phi \Vdash \psi$, then $\Gamma \Vdash \psi$. 
% for all 
% $\phi$, $\psi$, $\chi$, if $\phi \Vdash \psi$ and $\psi \Vdash \chi$, then $\phi \Vdash \chi$. 

\begin{figure}[ht]
\hrule
\vspace{1mm} 
\[
% \begin{array}{lccl}
% \mbox{(Inferon)} & 
% \dfrac{}{\vdash \inferon{P}{\baseP}{b}{}}{} & \mbox{iff} & 
% \vdash_\baseP \inferatom{P}{b}{}
% \end{array}
\begin{array}{lccl}
\mbox{(Inferon)} & 
\dfrac{}{\Gamma \vdash \inferon{P}{\baseP}{b}{}}{} & \mbox{iff} & 
\mbox{for all $\baseP' \supseteq \baseP$, if $\Vdash_{\baseP'} \phi$, for every $\phi \in \Gamma$, then $\vdash_{\baseP'} \inferatom{P}{b}{}$} 
\end{array}
\] 
\vspace{1mm}
\hrule
\caption{The axiom ({Inferon}) for the theory of inferons}
\label{fig:inferon-ax}
\end{figure}

The (Inferon) axiom, as given in Figure~\ref{fig:inferon-ax}, is a generalization, in the atomic case of the principal formula, of the ($Ax$) rule of NJ and its purpose is to axiomatize the theory of inferons within NJ. As such, it must refer to the (inferentialist) semantic judgement $\Vdash_{\mathcal{P}'}$ in order to connect NJ to the inferonic base, which is essential to the definition of the theory. 
In model-theoretic semantics, the counterpart would be the interpretation of the atoms of a theory in a model. 
In our setting, this constraint moves from the interpretation to the proof system using the inferential semantics and the use of the support relation $\Vdash_{}$ (rather than $\vdash$) is essential because the elements of $\Gamma$ are inferons. 
However, the support relation for a given $\phi$ encodes 
an unpacking of the meaning of $\phi$ to base proofs ($\vdash$, as in the (At) rule) of its atomic subformulae. 
The base-extension condition $\baseP' \supseteq \baseP$ is essential to ensure coherence with the (Inf) clause of the semantics. In the case in which $\Gamma$ is empty, so that (Inf) will not be invoked, it reduces, as one would expect, to 
the purely syntactic form 
\[
 \begin{array}{lccl}
 \mbox{} & 
 \dfrac{}{\vdash \inferon{P}{\baseP}{b}{}}{} & \mbox{iff} & 
 \vdash_\baseP \inferatom{P}{b}{}
 \end{array}
\] 
which simply enforces the coherence of 
NJ and base proofs for inferons. 

The (Inferon) axiom also provides a 
clear illustration of how --- not only in B-eS but also 
in P-tV --- bases replace models in the giving of meaning to formul{\ae} and proofs. Specifically, an appeal to the truth of an atomic formula in a model is replaced by the requirement of provability in a base. We remark that it would be interesting to compare and contrast Mares's model-theoretic semantics, expressed in terms of `info models' \cite{Mares1996}, and Restall's analysis \cite{Restall1996} with our B-eS for our logic of inferons. 

A general account of theories in B-eS and how to represent them 
in support relations and proof systems is yet to be developed. 

%The need for the richer form of (Inferon) arises technically from the presence of local bases within inferons --- it handles the effect on the support relation that is induced via the axiom rule.  

\begin{remark}[Veridicality 1] \label{rem:verid-1}
The veridicality thesis asserts the information must encapsulate truth. In our inferentialist set-up, truth (implicitly, in a model) is replaced by validity (provability in a base). This is a very substantial difference. A base makes information assumptions explicit and the base-extension semantics, through the support relation, translates 
these assumptions to the inferons --- that is, the information --- that are supported. Thus validity of information is inferential, depending wholly on explicit assumptions, without reference to truth. 
\fillBox
\end{remark}

\begin{remark}[Veridicality 2] \label{rem:verid-2}
It might be asked whether our choice to follow situation theory in employing both the assertion and denial of inferons, through the parameter $\mathfrak{b}$, entails a violation of veridicality \cite{Floridi2011-PhilInf,FrescoMichael2016} --- that is, the 
principle that information must be true. We would argue that it does not. A meaningful version of veridicality for inferon logic has to be expressed in terms of validity relative to the assumptions represented in the underlying base. For denials (i.e., inferons with polarity $\mathfrak{b}=0$) that is still respected at the level of the support relation --- the support of a denial is determined to hold based on the underlying base. That said, we observe that it would seem to be valuable to consider in more detail the relationship between validity and assertion/denial \cite{Rumfitt2000}. \fillBox  
\end{remark}

%In inferon logic, veridicality, expressed as validity relative to the assumptions represented in the underlying base, is respected at the level of the support relation --- the support of a denial is determined to hold. 

\begin{remark}[Veridicality 3] \label{rem:verid-3}
An explicit comparison of truth and B-eS validity, and an explicit inferentialist counterpart to veridicality, would require an analysis, for inferons at least, of possible correspondences between validity through bases and interpretation in models. There are two questions to address for such a counterpart to veridicality. First, what does it mean for base rules to be correct (as a counterpart to the truth of atoms)? For example, does a base capture the scenario of interest as intended? Second, does the support relation preserve this correctness. This would be a substantial undertaking, beyond the scope of this paper. \fillBox
\end{remark}

For the propositional logic of inferons,
as given by the propositional support relation and the (Inferon) axiom with 
propositional atoms, we can adapt 
Sandqvist's proofs of soundness and completeness  \cite{Sandqvist2015IPL}. 

\begin{theorem}[Soundness for the propositional logic of inferons] \label{thm:prop-inferon-sound}
If $\Gamma \vdash \phi$, then $\Gamma \Vdash \phi$.
\end{theorem}
\begin{proof} 

We extend the proof in \cite{Sandqvist2015IPL} by the case for the (Inferon) rule. The proof goes by showing that a translation of the rules of $\vdash$ into the support relation still hold. For example, here are the translations of the rules for implication:
\[
\begin{array}{ll}
\mbox{($\supset I$)}' & \mbox{If $\Gamma,\phi \Vdash\psi$, then $\Gamma\Vdash\phi\supset\psi$}\\ 
\mbox{($\supset E$)}' & \mbox{If $\Gamma\Vdash\phi\supset\psi$ and $\Gamma\Vdash \phi$, then $\Gamma\Vdash\psi$}
\end{array} 
\] 

For $(\supset I)'$, assume that $\Gamma,\phi \Vdash\psi$. Take some $\baseB$ s.t. $\Vdash_\baseB \chi$ for every $\chi\in\Gamma$. For every $\baseC \supseteq\baseB$ s.t. $\Vdash_\baseC \phi$ we have $\Vdash_\baseC \psi$. So, we have $\phi\Vdash_\baseB \psi$ and $\Vdash_\baseB \phi\supset\psi$. From this we can conclude $\Gamma\Vdash \phi\supset\psi$.

For $(\supset E)'$, note, by (Inf), that if $\phi\Vdash_\baseB \psi$ and $\Vdash_\baseB \phi$, then $\Vdash_\baseB \psi$. 
By $(\supset)$ (in Figure~\ref{fig:support2}), $\phi\Vdash_\baseB \psi$ iff $\Vdash_\baseB\phi\supset\psi$ and so, if $\Vdash_\baseB\phi\supset\psi$ and $\Vdash_\baseB \phi$, then $\Vdash_\baseB \psi$. So, $\phi\supset\psi, \phi\Vdash \psi$. 
%That is, whenever $\Vdash_\baseB\phi\supset\psi$ and $\Vdash_\baseB \psi$. So, $\phi\supset\psi, \phi\Vdash \psi$. 
%We conclude by the transitivity of $\Vdash$. 
By transitivity, for every $\Gamma$ s.t. $\Gamma \Vdash \phi\supset\psi$ and $\Gamma\Vdash \phi$, we have $\Gamma\Vdash\psi$.

For our proof here, we have added only the ($\mbox{Inferon}$) rule and so it suffices to show that
\[
\begin{array}{lccl}
\mbox{(Inferon)}' & \Gamma \Vdash \inferon{P}{\baseP}{b}{} & \mbox{iff} & \mbox{for all $\baseP' \supseteq \baseP$, $\Vdash_{\baseP'} \phi$, for every $\phi \in \Gamma$, then $\vdash_{\baseP'} \inferatom{P}{b}{}$} 
\end{array} 
\] 
where $(\mbox{Inferon})'$ is, following Sandqvist's presentation \cite{Sandqvist2015IPL}, the translation of $(\mbox{Inferon})$ into the support relation. 

\[
\begin{array}{lcll}
\Gamma \Vdash \inferon{P}{\baseP}{b}{} & \mbox{iff} & \mbox{for all $\baseB$, $\Vdash_{\baseB}\inferon{P}{\baseP}{b}{}$} & \mbox{(Validity)}\\
& \mbox{iff} & \mbox{for all $\baseB'$, if $\Vdash_{\baseB'}\phi$, for every $\phi\in\Gamma$, then $\Vdash_{\baseB'}\inferon{P}{\baseP}{b}{}$} & \mbox{(Inf)}\\

& \mbox{iff} & \mbox{for all $\baseB'$, if $\Vdash_{\baseB'}\phi$, for every $\phi\in\Gamma$, then $\vdash_{\baseB'\cup\baseP}\inferatom{P}{b}{}$} & \mbox{(At)}\\

& \mbox{iff} & \mbox{for all $\baseP' \supseteq \baseP$, if $\Vdash_{\baseP'} \phi$, for every $\phi \in \Gamma$, then $\vdash_{\baseP'} \inferatom{P}{b}{}$} & \\
\end{array} 
\] 
For the last step, note that for every such $\baseB'$ we can take $\baseP'=\baseB'\cup \baseP$. For the upwards direction, note that, for all $\baseB'$, if $\Vdash_{\baseB'}\phi$, for every $\phi\in\Gamma$, then $\vdash_{\baseB'\cup\baseP}\inferatom{P}{b}{}$ only fails for $\baseB'$ s.t. $\Vdash_{\baseB'}\phi$, for every $\phi\in\Gamma$ and  $\nvdash_{\baseB'\cup\baseP}\inferatom{P}{b}{}$. However, if $\Vdash_{\baseB'}\phi$, then $\Vdash_{\baseB'\cup\baseP}\phi$ and so, from the last line, $\vdash_{\baseB'\cup\baseP}\inferatom{P}{b}{}$.
% By validity and (Inf), $\Gamma \Vdash \inferon{P}{\baseP}{b}{}$ iff for all $\base B$ and $\baseB'\supseteq\baseB$ s.t., for all $\phi \in \Gamma$, if $\Vdash_{\baseB'} \phi$, $\Vdash_\baseB \inferon{P}{\baseP}{b}{}$. By (At), $\Vdash_\baseB \inferon{P}{\baseP}{b}{}$ iff $\Vdash_{\baseB\cup\baseP} \inferatom{P}{b}{}$. It suffices to point out that for every $\baseP'$, there is a $\baseB$ s.t. $\baseB\cup\baseP = \baseP'$ and vice-versa.
\end{proof} 

%\note{set up analogy with sandqvist, etc.}

Before we move to completeness we require the Lemmas \ref{lem:monotonicitysupport} and \ref{lem:INFsupport}.

\begin{lemma}
    \label{lem:monotonicitysupport}
If $\mathfrak{P}\Vdash_\baseB \inferatom{P}{b}{}$, then $\mathfrak{U} \cup \mathfrak{P}\Vdash_\baseB \inferatom{P}{b}{}$. 
    
\end{lemma}

\begin{lemma}
    \label{lem:INFsupport}

$\mathfrak{P}\vdash_{\baseB} \inferatom{P}{b}{}$ iff for every $\baseC \supseteq \baseB$, if $\vdash_{\baseC} \phi$ for every $\phi \in \mathfrak{P}$, then $\vdash_{\baseC} \inferatom{P}{b}{}$.
    
\end{lemma}

\begin{proof}
    The proofs remain unchanged from the proofs of equivalent results for intuitionistic propositional logic (namely, Lemmas 2.1 and 2.2 in \cite{Sandqvist2015IPL}) except with inferonic atoms in place of propositional letters.
\end{proof}

To show completeness we adapt the proof of \cite{Sandqvist2015IPL}, with some important differences. 
The basic strategy there is to create a base $\mathcal{N}$ such that $\vdash_\mathcal{N}$ behaves like $\vdash$. This is done by encoding the rules of $\vdash$ into base rules in $\mathcal{N}$ with the help of a flattening function that assigns an atomic formul{\ae} to each formula. We then show two properties of $\mathcal{N}$ ($(\dagger)$ and $(\ddagger)$ below) that, together with the property $\mathbb{T} \vdash_\baseB p$ iff $\mathbb{T}\Vdash_\baseB p$ where, following Figure \ref{fig:notation}, $\mathbb{T}$ is a set of atoms (Theorem 3.1 in \cite{Sandqvist2015IPL}), suffice to proof the completeness result.

Our proof here follows that strategy, but differs in some key aspects. First, we extend the %base-
support relation to inferonic atoms 
% formul{\ae} of the form 
$\inferatom{P}{b}{}_\baseP$ such that $\inferatom{P}{b}{}_\baseP$ is supported at a base $\baseB$ if and only if $\inferatom{P}{b}{}$ is supported at $\baseB\cup\baseP$. 
We require two flattening functions: $\flat$ assigns to every formula a basic inferon $\inferon{P}{\baseP}{b}{}$ and $\flat-$ that assigns to them the corresponding formula $\inferatom{P}{b}{}_\baseP$, with the same closed atomic predicate $P$, polarity $\mathfrak{b}$, and base $\baseP$. 
We then prove $(\ref{comp3})$ below which takes the role of Theorem 3.1 in the proof. Note that a direct translation of Sandqvist's Theorem 3.1 cannot hold for this semantics as $\vdash_\baseB$ deals with inferonic atoms and $\Vdash_\baseB$ with inferons. Given that, we show that $(\dagger)$ and $(\ddagger)$ follow as before and we obtain the wanted result. 

The details of the proof follow.

\begin{theorem}[Completeness for the propositional logic of inferons] \label{thm:prop-inferon-complete}

If $\Gamma \Vdash \phi$, then $\Gamma \vdash \phi$.
\end{theorem}
\begin{proof} 

Let $\Gamma'$ be the set containing all members of $\Gamma\cup\phi$ and their subsentences. We define a flattening function $\phi^\flat$ so that for for any $\phi\in \Gamma$ we associate a $\phi^\flat \notin \Gamma$ s.t. for every $\phi,\psi\in \Gamma$ $\phi^\flat \neq \psi^\flat$ whenever $\phi\neq\psi$ and so that $\inferon{P}{\baseP}{b}{}^\flat = \inferon{P}{\baseP}{b}{}$, for any $\inferon{P}{\baseP}{b}{}$. For a set of formul{\ae} $\Gamma$, let $\Gamma^\flat$ be the sequence obtained by replacing any $\phi \in \Gamma$ with $\phi^\flat$.

% OLD Let $\Gamma'$ be the set containing all members of $\Gamma\cup\phi$ and their subsentences. We define a flattening function $\phi^\flat$ so that for for any $\phi\in \Gamma$ we associate a $\phi^\flat \notin \Gamma$ s.t. for every $\phi,\psi\in \Gamma$ $\phi^\flat \neq \psi^\flat$ whenever $\phi\neq\psi$ with $\inferon{P}{\baseP}{b}{}$ taking the role of the atom and so, $\inferon{P}{\baseP}{b}{}^\flat = \inferon{P}{\baseP}{b}{}$. Our one change to $\phi^\flat$ is that we take $\bot^\flat = \bot$. As before for a set of formul{\ae} $\Gamma$, let $\Gamma^\flat$ be the sequence obtained by replacing any $\phi \in \Gamma$ with $\phi^\flat$.

% \note{A GAP HERE? Need to get from $<P,B,b>$ to $<P,b>_B$?}  
% \note{That is done below through the second flattening function $\phi^{b-}$}

For technical convenience, we extend the base-support relation to a new type of inferonic atom of the form $\inferatom{P}{b} 
{}_\baseP$, as follows: 
\begin{equation*}
    \mathfrak{P} \vdash_\baseB \inferatom{P}{b}{}_\baseP     \quad\mbox{iff}\quad  \mathfrak{P} \vdash_{\baseB\cup\baseP} \inferatom{P}{b}{}
\end{equation*}

%\[
%\begin{array}{lcl}
 %\mathfrak{P} \vdash_\baseB \inferatom{P}{b}{}_\baseP    & \mbox{iff} & \mathfrak{P} \vdash_{\baseB\cup\baseP} \inferatom{P}{b}{}   \\
%\end{array}
%\]
With this device we can define a further flattening function $\phi^{b-}$ just like $\phi^\flat$ except to inferonic atoms so that $\inferon{P}{\baseP}{b}{}^{\flat -} = \inferatom{P}{b}{}_\baseP$. Again, let $\Gamma^{\flat -}$ be $\Gamma$ with all $\phi$ replaced by $\phi^{\flat -}$. We also define the functions $\phi^{\natural}$ and $\phi^{\natural -}$ so that for any $\phi\in \Gamma'$, $(\phi^\flat)^\natural = \phi$ and $(\phi^{\flat-})^{\natural-} = \phi$.

We prove the following property of $\inferatom{P}{b}{}_\baseP$ formul{\ae} as it will be required in our proof. 
% It is analogous to Lemma~2.2 in \cite{Sandqvist2015IPL}: $\mathbb{T}\vdash_\baseB p$ iff for every $\baseB'\supseteq \baseB$, if $\vdash_{\baseB'}p$ for every $p \in \mathbb{T}$, then $\vdash_{\baseB'} p$.
%(Lemma 2.2 for $<<P,b>>_\baseB$)
\begin{equation}
\label{comp1}
\mathfrak{P}\vdash_\baseB \inferatom{P}{b}{}_\baseP     \quad\mbox{iff}\quad \mbox{for every $\baseB'\supseteq \baseB$, if $\vdash_{\baseB'}\phi$ for every $\phi \in \mathfrak{P}$, then $\vdash_{\baseB'} \inferatom{P}{b}{}_\baseP$}
\end{equation}

%\[
%\begin{array}{llcl}
% (1) & \mathfrak{P}\vdash_\baseB \inferatom{P}{b}{}_\baseP    & \mbox{iff} & \mbox{for every $\baseB'\supseteq \baseB$, if $\vdash_{\baseB'}\phi$ for every $\phi \in \mathfrak{P}$, then $\vdash_{\baseB'} \inferatom{P}{b}{}_\baseP$}  \\
%\end{array}
%\]
Note, by definition, $\vdash_{\baseB'} \inferatom{P}{b}{}_\baseP$ iff $\vdash_{\baseB'\cup\baseP} \inferatom{P}{b}{}$ and $\mathfrak{P}\vdash_\baseB \inferatom{P}{b}{}_\baseP$ iff $\mathfrak{P}\vdash_{\baseB\cup\baseP} \inferatom{P}{b}{}$. 
By Lemma \ref{lem:INFsupport}, $\mathfrak{P}\vdash_{\baseB\cup\baseP} \inferatom{P}{b}{}$ iff for every $\baseC \supseteq \baseB\cup\baseP$, if $\vdash_{\baseC} \phi$ for every $\phi \in \mathfrak{P}$, then $\vdash_{\baseC} \inferatom{P}{b}{}$. 

So it suffices to show that for every $\baseB'\supseteq \baseB$, if $\vdash_{\baseB'} \phi$ for every $\phi\in \mathfrak{P}$, then $\vdash_{\baseB'\cup\baseP} \inferatom{P}{b}{}$ iff for every $\baseC \supseteq \baseB\cup\baseP$, if $\vdash_{\baseC} \phi$ for every $\phi \in \mathfrak{P}$, then $\vdash_{\baseC} \inferatom{P}{b}{}$. For the left-to-right direction, it suffices to point out that $\baseC\supseteq\baseB'$ and that $\baseC\cup\baseP = \baseC$. Right-to-left follows because if $\vdash_{\baseB'} \phi$, then also $\vdash_{\baseB'\cup\baseP} \phi$ and $\baseB'\cup\baseP \supseteq \baseB\cup\baseP$.

%\note{Should we state the definition of $\mathcal{N}$?}

%\note{Yes! At least in general terms: idea + a couple of examples?}

%In \cite{Sandqvist2015IPL}, this completeness proof is reliant on the fact that the following holds for atomic $p$ in their semantics: $\mathbb{T}\vdash_\baseB p$ iff $\mathbb{T}\Vdash_\baseB p$. Note that such a result cannot hold for this semantics as $\vdash_\baseB$ deals with inferonic atoms and $\Vdash\baseB$ with inferons.

% Our extension of the %base-
% support relation and the new flattening function $\phi^{\flat -}$ allow us to prove an analogous result sufficient to give completeness:

Our extension of the support relation and the new flattening function allow us to prove the following:
\begin{equation}
\label{comp2}
    \Gamma^{\flat -}\vdash_\baseB \inferatom{P}{b}{}_\baseP \quad\mbox{iff }\quad \Gamma^\flat\Vdash_\baseB \inferon{P}{\baseP}{b}{}
\end{equation}

%\[
%\begin{array}{llcl}
%(2) & \Gamma^{\flat -}\vdash_\baseB \inferatom{P}{b}{}_\baseP    & \mbox{iff} & \Gamma^\flat\Vdash_\baseB \inferon{P}{\baseP}{b}{}  \\
%\end{array}
%\]

For empty $\Gamma$, this follows directly from the definition of $\inferatom{P}{b}{}_\baseP$ and (At). Otherwise this follows from those together with (Inf) and (\ref{comp1}). Since for any $\phi^\flat = \inferon{P}{\baseP}{b}{}$ we have $\phi^{\flat -} = \inferatom{P}{b}{}_\baseP$,  we also get the following: 
\begin{equation}
\label{comp3}
   \Gamma^{\flat -}\vdash_\baseB \phi^{\flat -} \quad\mbox{iff}\quad \Gamma^\flat\Vdash_\baseB \phi^{\flat} 
\end{equation}

%\[
%\begin{array}{llcl}
%(3) & \Gamma^{\flat -}\vdash_\baseB \phi^{\flat -}    & \mbox{iff} & \Gamma^\flat\Vdash_\baseB \phi^{\flat}  \\
%\end{array}
%\]

% The rest of the proof remains unchanged except uses of (\ref{comp3}) are substituted for any uses of $\mathfrak{P} \vdash_\baseB p$ iff $\mathfrak{P}\Vdash_\baseB p$ and $\phi^{\flat -}$ replaces $\phi^b$ on the level of base-support. 

% The strategy involves defining a base $\mathcal{N}$ that mimics the rules of natural deduction. 
We can now define $\mathcal{N}$ that mimics the rules of natural deduction as the base containing all and only the rules of the following form:
\[
\begin{array}{ll}
1. & (\phi^{\flat -} \Rightarrow \psi^{\flat -})\Rightarrow (\phi\supset \psi)^{\flat -}\\
2. & (\Rightarrow(\phi\supset\psi)^{\flat -}),(\Rightarrow\phi^{\flat -})\Rightarrow\psi^{\flat -}\\
3. & (\Rightarrow\phi^{\flat -}), (\Rightarrow\psi^{\flat -})\Rightarrow(\phi\wedge\psi)^{\flat -}\\
4. & (\Rightarrow(\phi\wedge\psi)^{\flat -})\Rightarrow \phi^{b-}\\
5. & (\Rightarrow(\phi\wedge\psi)^{\flat -})\Rightarrow \psi^{\flat -}\\
6. & (\Rightarrow \phi^{\flat -})\Rightarrow(\phi\vee\psi)^{\flat -}\\
7. & (\Rightarrow \psi^{\flat -})\Rightarrow(\phi\vee\psi)^{\flat -}\\
8. & (\Rightarrow (\phi\vee\psi)^{\flat -}),(\phi^{\flat -}\Rightarrow \inferon{P}{\baseP}{b}{}^{\flat -}),(\psi^{\flat -}\Rightarrow \inferon{P}{\baseP}{b}{}^{\flat -})\Rightarrow\inferon{P}{\baseP}{b}{}^{\flat -}\\
9. & (\Rightarrow\bot^{\flat -})\Rightarrow\inferatom{P}{b}{}\\
\end{array}
\]

With these changes, we can now prove $(\dagger)$ and $(\ddagger)$: 
\[
\begin{array}{ll}
(\dagger) & \mbox{for every $\phi\in \Gamma'$ and every $\baseB \supseteq \mathcal{N}$, $\Vdash_\baseB \phi^\flat$ iff $\Vdash_\baseB \phi$} \\
(\ddagger) & \mbox{for any $\mathfrak{P}$ and $\inferatom{P}{b}{}_\baseP$ if $\mathfrak{P}\vdash_\mathcal{N} \inferatom{P}{b}{}_\baseP$ then $\mathfrak{P}^\natural \vdash (\inferatom{P}{b}{}_\baseP)^\natural$}\\
\end{array}
\]

For $(\dagger)$, we start by showing for any $\baseB\supseteq\mathcal{N}$ the following hold for any $\phi\supset\psi, \phi\wedge\psi,$ and $\phi\vee\psi \in \Gamma'$: 
\[
\begin{array}{lll}
(a) & \vdash_\baseB (\phi\supset\psi)^{\flat-} & \mbox{ iff } \phi^{\flat-} \vdash_\baseB \psi^{\flat-}\\
(b) & \vdash_\baseB (\phi\wedge\psi)^{\flat-} & \mbox{ iff } \vdash_\baseB\phi^{\flat-}\mbox{ and } \vdash_\baseB \psi^{\flat-}\\
(c) & \vdash_\baseB (\phi\vee\psi)^{\flat-} & \mbox{ iff for every inferon $\iota$ and every } \baseC\supseteq\baseB, \mbox{ if }  \phi^{\flat-} \vdash_\baseC\iota^{\flat-} \mbox{ and } \psi^{\flat-} \vdash_\baseC\iota^{\flat-}, \mbox{then } \vdash_\baseC \iota^{\flat-} \\
(d) & \vdash_\baseB \bot^{\flat-} & \mbox{ iff } \vdash_\baseB \inferatom{P}{b}{} \mbox{ for every $\inferatom{P}{b}{}$}
\end{array}
\]

\emph{$(a)$}. Left-to-right: By Lemma \ref{lem:monotonicitysupport}, if $\vdash_\baseB (\phi\supset\psi)^{\flat-}$, then $\phi^{\flat-}\vdash_\baseB (\phi\supset\psi)^{\flat-}$. It follows that $\phi^{\flat-}\vdash_\baseB \psi^{\flat-}$ from $\phi^{\flat-}\vdash_\baseB\phi^{\flat-}$ and an application of the relevant rule of the form 2 of $\mathcal{N}$. Right-to-left follows immediately from the relevant rule of the form 1 of $\mathcal{N}$.

\emph{$(b)$}. Follows from the relevant rules of forms 3, 4, and 5.

\emph{$(c)$}. Left-to-right follows from rule form 8. Right-to-left: Assuming the right side then it holds that if $\phi^{\flat-}\vdash_\baseB(\phi\vee\psi)^{\flat-}$ and $\psi^{\flat-}\vdash_\baseB(\phi\vee\psi)^{\flat-}$, then $\vdash_\baseB(\phi\vee\psi)^{\flat-}$. From the relevant rule of form 6 and $\phi^{\flat-}\vdash_\baseB\phi^{\flat-}$ it follows that $\phi^{\flat-}\vdash_\baseB(\phi\vee\psi)^{\flat-}$ and likewise $\psi^{\flat-}\vdash_\baseB(\phi\vee\psi)^{\flat-}$ from $\psi^{\flat-}\vdash_\baseB\psi^{\flat-}$ and rule form 7. So, $\vdash_\baseB (\phi\vee\psi)^{\flat-}$.

\emph{$(d)$}. Left-to-right follows from rule form 9. Right-to-left follows from $\bot^{\flat-}$ being an inferonic atom.

The following now hold for any $\baseB\supseteq\mathcal{N}$ and formula in $\Gamma'$ so that $\flat$ is defined, by (3) applied to (a)--(d):

\[
\begin{array}{lll}
(a') & \Vdash_\baseB (\phi\supset\psi)^{\flat} & \mbox{ iff } \phi^{\flat} \Vdash_\baseB \psi^{\flat}\\
(b') & \Vdash_\baseB (\phi\wedge\psi)^{\flat} & \mbox{ iff } \Vdash_\baseB\phi^{\flat}\mbox{ and } \Vdash_\baseB \psi^{\flat}\\
(c') & \Vdash_\baseB (\phi\vee\psi)^{\flat} & \mbox{ iff for every inferon $\iota$ and every } \baseC\supseteq\baseB, \mbox{ if }  \phi^{\flat} \Vdash_\baseC\iota \mbox{ and } \psi^{\flat} \Vdash_\baseC\iota, \mbox{then } \Vdash_\baseC \iota \\
(d') & \Vdash_\baseB \bot^{\flat} & \mbox{ iff } \vdash_\baseB \iota \mbox{ for all closed $\iota$}
\end{array}
\]

For $(c')$ note that $\iota^\flat = \iota$ and for $(d')$ note that for every $\iota$ there is an inferonic atom $\iota^{\flat-}$ s.t. $\vdash_\baseB \iota^{\flat-}$.

We now show $(\dagger)$ by induction on the complexity of $\phi$.

For basic $\phi$, $(\dagger)$ holds from $\phi^\flat = \phi$. For $\phi=\bot$ it follows from $(d')$.

The induction hypothesis is that for every $\baseB\supseteq\mathcal{N}$, $\Vdash_\baseB\psi$ iff $\Vdash_\baseB\psi$ and $\Vdash_\baseB\chi$ iff $\Vdash_\baseB\chi$. We show that $(\dagger)$ holds for $\psi\supset\chi$, $\psi\wedge\chi$, and $\psi\vee\chi$.

From the induction hypothesis it follows that, for any formula $\pi$,

\[
\psi^\flat\Vdash_\baseB \pi \mbox{ iff } \psi\Vdash_\baseB \pi,
\]
\[
\psi\Vdash_\baseB \pi \mbox{ iff } \psi^\flat\Vdash_\baseB \pi
\]

and likewise for $\chi$.

From this and $(a')$, we get that $\Vdash_\baseB (\psi\supset\chi)^\flat$ iff $\psi^\flat\Vdash_\baseB\chi^\flat$ iff $\psi\Vdash_\baseB\chi$ iff $\Vdash_\baseB (\psi\supset\chi)$.

The cases for $\psi\wedge\chi$ and $\psi\vee\chi$ follow in the same way. This concludes the proof of $(\dagger)$.

We proceed to show $(\ddagger)$. Since the rules in $\mathcal{N}$ correspond to the standard introduction and elimination rules for the intuitionistic operators, any derivation in $\mathcal{N}$ can be translated into a derivation of intuitionisitc natural deduction. So we show $(\ddagger)$ by proving that 

\[\Gamma^{\sharp-}, \inferatom{P}{b}{}^{\sharp-}\vdash \inferatom{P}{b}{}^{\sharp-}\]

and 

\[\text{if }\atRule{R} \in \mathcal{N} \text{ and, for all }i \in [1,n], \mathfrak{Q}^{\sharp-} , \mathfrak{P}_i^{\sharp-} \vdash \inferatom{Q}{b}{i}^{\sharp-},\text{ then }\mathfrak{Q}^{\sharp-} \vdash \inferatom{R}{c}{}^{\sharp-}.
\]

$\Gamma^{\sharp-}, \inferatom{P}{b}{}^{\sharp-}\vdash \inferatom{P}{b}{}^{\sharp-}$ is just an instance of $(Ax)$. The other is equivalent to the conjunction of the following claims:

\[
\begin{array}{l}
  \mbox{If } \Gamma^{\sharp-},(\phi^{\flat-})^{\sharp-}\vdash(\psi^{\flat-})^{\sharp-},\mbox{then } \Gamma^{\sharp-}\vdash((\phi\supset\psi)^{\flat-})^{\sharp-}.     \\
  \mbox{If } \Gamma^{\sharp-}\vdash((\phi\supset\psi)^{\flat-})^{\sharp-} \mbox{ and }\Gamma^{\sharp-}\vdash(\phi^{\flat-})^{\sharp-}, \mbox{then } \Gamma^{\sharp-}\vdash(\psi^{\flat-})^{\sharp-}.\\
  \mbox{If } \Gamma^{\sharp-}\vdash(\phi^{\flat-})^{\sharp-} \mbox{ and } \Gamma^{\sharp-}\vdash(\psi^{\flat-})^{\sharp-}, \mbox{then }\Gamma^{\sharp-}\vdash((\phi\wedge\psi)^{\flat-})^{\sharp-}.\\
  \mbox{If }\Gamma^{\sharp-}\vdash((\phi\wedge\psi)^{\flat-})^{\sharp-}, \mbox{then }\Gamma^{\sharp-}\vdash(\phi^{\flat-})^{\sharp-}.\\
   \mbox{If }\Gamma^{\sharp-}\vdash((\phi\wedge\psi)^{\flat-})^{\sharp-}, \mbox{then }\Gamma^{\sharp-}\vdash(\psi^{\flat-})^{\sharp-}.\\
  \mbox{If }\Gamma^{\sharp-}\vdash(\phi^{\flat-})^{\sharp-}, \mbox{then }\Gamma^{\sharp-}\vdash((\phi\vee\psi)^{\flat-})^{\sharp-}.\\
    \mbox{If }\Gamma^{\sharp-}\vdash(\psi^{\flat-})^{\sharp-}, \mbox{then }\Gamma^{\sharp-}\vdash((\phi\vee\psi)^{\flat-})^{\sharp-}.\\
   \mbox{If }\Gamma^{\sharp-}\vdash((\phi\vee\psi)^{\flat-})^{\sharp-} \mbox{ and }\Gamma^{\sharp-},(\phi^{\flat-})^{\sharp-}\vdash(\inferon{P}{\baseP}{b}{}^{\flat-})^{\sharp-} \mbox{ and }\Gamma^{\sharp-},(\psi^{\flat-})^{\sharp-}\vdash(\inferon{P}{\baseP}{b}{}^{\flat-})^{\sharp-}, \\
   \mbox{then } \Gamma^{\sharp-},\vdash(\inferon{P}{\baseP}{b}{}^{\flat-})^{\sharp-}.\\
   \mbox{If }\Gamma^{\sharp-}\vdash(\bot^{\flat-})^{\sharp-}, \mbox{then } \Gamma^{\sharp-},\vdash(\inferon{P}{\baseP}{b}{}^{\flat-})^{\sharp-}.
\end{array}
\]

Since $(\phi^{\flat-})^{\sharp-} = \phi$ whenever $\phi^{\flat-}$ is defined these claims follow from the rules of NJ (Figure \ref{fig:nj}). That concludes the proof of $(\ddagger)$.

The final result follows from $(\dagger)$, $(\ddagger)$, and (\ref{comp3}): Assume $\Gamma\Vdash \phi$. For every $\baseB\supseteq\mathcal{N}$ such that $\Vdash_\baseB \psi^\flat$ for every $\psi^\flat\in\Gamma^\flat$, it follows by $(\dagger)$ that $\Vdash_\baseB \psi$ for every $\psi\in\Gamma$. So, since $\Gamma\Vdash\phi$, $\Vdash_\baseB \phi$ and, again by $(\dagger)$, $\Vdash_\baseB \phi^\flat$. It follows that $\Gamma^\flat \Vdash \phi^\flat$. By $(3)$, we get $\Gamma^{\flat -}\vdash_{\mathcal{N}} \phi^{\flat -}$ and so, by $(\ddagger)$, we can conclude $\Gamma\vdash\phi$.
\end{proof}

For the first-order logic of inferons, 
we extend the support relation as in 
Figure~\ref{fig:first-order-support}
(building on \cite{GP2025-P-tS-S-oL}), and 
of course consider the $(\mbox{Inferon})$ axiom with closed atomic predicates. 

\begin{figure}[ht]
\hrule
\[
\begin{array}{lrcl}     
(\mbox{At}) & \Vdash_\mathcal{B} \inferon{P}{P}{b}{}   & \mbox{iff} & 
       \mbox{$\vdash_{\mathcal{B} \,\cup\, \mathcal{P}} 
       \inferatom{P}{b}{}$ for closed  $P$} \\ 
\end{array}
\]
\dotfill \vspace{1mm}
\[
\begin{array}{lrcl}     
(\forall) &  \Vdash_\mathcal{B} \forall{x}.\phi & \mbox{iff} & 
     \mbox{$\Vdash_\mathcal{B} \phi[t/x]$ for all $t \in {\rm Cl}({{T}})$} \\
(\exists) & \Vdash_\mathcal{B} \exists{x}.\phi & \mbox{iff} & 
  \mbox{for any $\mathcal{C} \supseteq \mathcal{B}$ and all closed $\iota$, if} \\ 
  & & & \mbox{$\phi[t/x] \Vdash_\mathcal{C} \iota$, for all $t \in {\rm Cl}({T})$, then $\Vdash_\mathcal{C} \iota$} 
\end{array}
\]
\hrule
\caption{Support relation for inferons: first-order case}
\label{fig:first-order-support}
\end{figure}

We conjecture that the soundness and completeness result extends to the first-order setting, building on 
\cite{GP2025-P-tS-S-oL}. This is work in progress.

\begin{remark}[Modalities] \label{remark:modalities}
Note that although we have discussed how agent modalities --- of the form $[a]$ and $\langle a \rangle$ --- can be understood in the setting of our logic of inferons and its base-extension semantics, we have not extended the 
soundness and completeness theorems to include them. Such an extension can be expected to draw upon the account of base-extension semantics for intuitionistic modal logics given in \cite{BuzokuPym-Modal2026}. \fillBox 
\end{remark}

Our basic notion of inferon --- what is intended when we refer to inferons --- is atomic in that the propositional component is taken to be a closed atom. As presented, the support relation supports logical combinations of  \emph{inferons}. For example, 
\[
\inferon{P}{P}{b}{} \;\wedge\; \inferon{Q}{Q}{c}{}
\]
However, we can allow connectives to occur within inferons provided the support relation preserves the associated base and polarity, For example, we can meaningfully write 
\[
\begin{array}{rcl}
\Vdash_\mathcal{B} \inferon{P_1 \wedge P_2}{P}{b}{} & \mbox{iff} & \mbox{$\Vdash_\mathcal{B} \inferon{P_1}{P}{b}{}$ and $\Vdash_\mathcal{B} \inferon{P_2}{P}{b}{}$} 
\end{array}
\]
and so construct \emph{compound inferons}. Thus the support relation then allows us to extend the inferon notation to incorporate the connectives, as in Figure~\ref{fig:support-internal}.   

\begin{figure}[ht]
\hrule 
\vspace{1mm}
 \[
    \begin{array}{r@{\quad}c@{\quad}l}
    \Vdash_\mathcal{B} \inferon{\phi_1 \wedge \phi_2}{P}{1}{} & \mbox{iff} & \Vdash_\mathcal{B} \inferon{\phi_1}{P}{1}{} \wedge \inferon{\phi_2}{P}{1}{} \\
    \Vdash_\mathcal{B} \inferon{\phi_1 \wedge \phi_2}{P}{0}{} & \mbox{iff} & \Vdash_\mathcal{B} \inferon{\phi_1}{P}{0}{} \vee \inferon{\phi_2}{P}{0}{} \\
    \Vdash_\mathcal{B} \inferon{\phi_1 \vee \phi_2}{P}{1}{} & \mbox{iff} & \Vdash_\mathcal{B} \inferon{\phi_1}{P}{1}{} \vee \inferon{\phi_2}{P}{1}{} \\
    \Vdash_\mathcal{B} \inferon{\phi_1 \vee \phi_2}{P}{0}{} & \mbox{iff} & \Vdash_\mathcal{B} \inferon{\phi_1}{P}{0}{} \wedge \inferon{\phi_2}{P}{0}{} \\
    \Vdash_\mathcal{B} \inferon{\phi_1 \supset \phi_2}{P}{1}{} & \mbox{iff} & \Vdash_\mathcal{B} \inferon{\phi_1}{P}{1}{} \supset \inferon{\phi_2}{P}{1}{} \\
    \Vdash_\mathcal{B} \inferon{\phi_1 \supset \phi_2}{P}{0}{} & \mbox{iff} & \Vdash_\mathcal{B} \inferon{\phi_1}{P}{1}{} \wedge \inferon{\phi_2}{P}{0}{} \\
    \Vdash_\mathcal{B} \inferon{\forall x.\phi}{P}{1}{} & \mbox{iff} & \Vdash_\mathcal{B} \forall x.\inferon{\phi}{P}{1}{}\\
    \Vdash_\mathcal{B} \inferon{\forall x.\phi}{P}{0}{} & \mbox{iff} & \Vdash_\mathcal{B} \exists x.\inferon{\phi}{P}{0}{}\\
    \Vdash_\mathcal{B} \inferon{\exists x.\phi}{P}{1}{} & \mbox{iff} & \Vdash_\mathcal{B} \exists x.\inferon{\phi}{P}{1}{}\\
    \Vdash_\mathcal{B} \inferon{\exists x.\phi}{P}{0}{} & \mbox{iff} & \Vdash_\mathcal{B} \forall x.\inferon{\phi}{P}{0}{}\\
    %... & \mbox{iff} & ... \\ 
    %... & \mbox{iff} & ... \\ 
    \end{array}
 \]   
\hrule
\caption{Compound inferons, with internal logical connectives} 
\label{fig:support-internal}
\end{figure}

\begin{lemma}[]
    For any compound inferon $\inferon{\phi}{P}{b}{}$, there exists a non-compound inferon $\psi$ (i.e. an inferon without internal connectives) such that, for any $\mathcal{B}$,
    \[
    \Vdash_\mathcal{B} \inferon{\phi}{P}{b}{} \quad\text{iff}\quad  \Vdash_\mathcal{B}\psi 
    \]
\end{lemma}

\begin{proof}
    This follows from a simple induction on the complexity of $\phi$ using the extension of the support relation in Figure~\ref{fig:support-internal}. 
\end{proof}

So far, we have considered just Boolean polarities in inferons, just as in classical situation theory. 
This is closely related to Rumfitt's conception of assertion and denial \cite{Rumfitt2000}, which is also employed 
by Gheorghiu and Buzoku \cite{GheorghiuBuzoku2025} in their treatment of base-extension semantics for classical propositional logic. Here, $\Vdash_\mathcal{B} \inferon{\phi}{P}{1}{}$ expresses that the property denoted by $\phi$ is, given its associated $\baseB$, supported by $\baseP$, whereas 
$\Vdash_\mathcal{B} \inferon{\phi}{P}{o}{}$ expresses that, given $\baseB$, the property denoted by $\phi$ 
is not supported by $\baseP$. 
It seems clear that richer notions of polarity will be of interest, especially in the context of richer logics. 

Finally, we remark that presentation of the given base-extension semantics for intuitionistic first-order readily generalizes to intuitionistic second-order logic \cite{GP2025-P-tS-S-oL}, offering the possibility of a more expressive logic of inferons. 

% \note{Discuss parametric choice of logics, and so 
% of proof systems.} 

\medskip 

% \subsection{Bases and validity II} 

% \[
%  \dfrac{\qquad}{\langle p , \mathcal{P} , \mathfrak{b} \rangle} \quad
%  \dfrac{\langle p_1 , \mathcal{P}_1 ,\mathfrak{b}_1 \rangle \ldots 

%             \langle p_k , \mathcal{P}_k ,\mathfrak{b}_k \rangle}{\langle p , \mathcal{P} , \mathfrak{b} \rangle} \quad 
%             \dfrac{
%             \begin{array}{ccc} 
%             [\Gamma_1] & & [\Gamma_k] \\  
%             & \ldots & \\ 
%             \langle p_1 , \mathcal{P}_1 ,\mathfrak{b}_1 \rangle & & 
%             \langle p_k , \mathcal{P} ,_k \mathfrak{b}_k \rangle 
%             \end{array}}
%             {\langle p , \mathcal{P} , \mathfrak{b} \rangle} \quad \ldots 
% \]

% \medskip 

% \[
%     \begin{array}{lrcl} 
%        (\mbox{At}) & \Vdash_\mathcal{B} \langle p , \mathcal{P} , \mathfrak{b} \rangle   & \mbox{iff} & 
%        \vdash_{\mathcal{B}} 
%        \langle p , \mathcal{P} , \mathfrak{b} \rangle \\ 
%        (\wedge) & \Vdash_\mathcal{B} {\iota}_1 \wedge {\iota}_2  & \mbox{iff} & \mbox{$\Vdash_\mathcal{B} {\iota}_1$ and $\Vdash_\mathcal{B} {\iota}_2$} \\ 
%        (\vee) & & \mbox{iff} & \\ 
%     \end{array} 
% \]

\section{First Examples} \label{sec:first-examples}

Below are some examples that indicate the use of inferons and situations in modelling informational scenarios. 

% \note{We can do a better intro than this ... :-)}
Example~\ref{example:NoSmoke} shows how attunement to a constraint (aligning with the meaning of those words from situation theory) can, in some cases, be formalized as an inference rule linking inferonic atoms. When immersed in an inferential context supplying the premise of the rule, the rule can be used to infer the conclusion of the rule, as the antecedent of the constraint. This takes (sets of) inferonic atoms to characterize situation types. Moreover, the example shows the usefulness of the formulation of the At rule of Figure~\ref{fig:support2} in decomposing the inferential bases into a contextual part and a part associated to a particular inferon. Example~\ref{example:NoSmokeModality} shows how attunement can then captured by the $\agentd{a}$ modality.

Example~\ref{example:WiseMen} shows how agents can resolve epistemic puzzles using fixed inferential capability. Specifically, they do not need unbounded capability to consider possible states of the world, given public announcements. Justification of the base rules used by agents does require knowledge of the bases used by other agents. Alternative, modal-relational  approaches to epistemic logic are discussed elsewhere by Eckhardt and Pym \cite{EckhardtPym2025S5,EckhardtPym2024PAL}.

\begin{example}[{No smoke without fire}]
\label{example:NoSmoke}
We adapt a classic example discussed in situation theory \cite{Devlin1991}. Smoke is observed arising from a mountaintop, and an observer some distance away concludes that there is fire on the mountaintop. The observer cannot observe the fire directly, but is `attuned' to a constraint $\mathrm{Smoke} \Rightarrow \mathrm{Fire}$ that says that smokey-type situations entail firey-type situations.

In our set-up, we shall be concerned 
with --- in addition to an ambient inferonic base $\baseB$ containing the background description of the scenario, including the constraint --- 
two other inferonic bases. The first base, $\baseP_1$, describes immediate perceptions about the local frame of the observer. 
The second base, $\baseP_2$, contains facts about the remote frame (the mountaintop). %\note{What's $\baseB$? Need to mention the inferon a bit earlier for this to make sense? }
We consider atomic sentence $P_1$ to be an abbreviation for `smoke is observed in the local frame' and $P_2$ to mean `fire is present on the mountaintop'. Here, 
the base $\baseB$ contains a single rule 
\begin{equation}
\label{rule:SmokeFire}
    \dfrac{\inferatom{P_1}{1}{}} 
    {\inferatom{P_2}{1}{}}
    . 
\end{equation} 
Suppose, moreover, that 
\begin{equation*}
\dfrac{}{\inferatom{P_1}{1}{}} 
\end{equation*}
is the only rule contained in base $\baseP_1$. 

Suppose that the observer uses base 
$\baseB \,\cup\, \baseP_1$, 
because they are attuned to the constraint and have observed smoke.
%We then have that 
%\begin{equation*}
%\dfrac{}{\inferatom{p_1}{1}{}}
%\qquad \mbox{and} \qquad 
% \dfrac{\inferatom{p_1}{1}{}} 
%    {\inferatom{p_2}{1}{}}
%\end{equation*}
%are both in $\mathcal{B} \cup \mathcal{P}_1$.  
It is then the case that we can derive 
$ %\begin{equation*}
    \inferatom{P_2}{1}{} 
$ %\end{equation*}
in the base $\baseB \cup \baseP_1$, and therefore 
$\Vdash_{\baseB} \inferonPlain{P_2}{\baseP_1}{1}{}$.
%\Vdash _\mathcal{B} \inferonSub{p}{2}{P}{1}{1}{}$.
%\begin{equation*} 
%\inferonSub{p}{2}{P}{1}{1}{} 
%\end{equation*} 
%is inferrable over $\mathcal{B}$, since according to the definition of support relation we have 
%$\Vdash _\mathcal{B} \inferonSub{p}{2}{P}{1}{1}{}$.

The key point here is that a reasoner with access to only $\baseP _1$ and $\baseB$ has  
$
\Vdash_{\baseB}
\inferonPlain{P_2}{\baseP_1}{1}{}
$, 
even without direct access to the reasoning resources, including basic facts (coded as atoms), of $\baseP _2$. 
%In one of the classic examples of situation theory, $\mathcal{P}_2$ regards events on a mountaintop, while $\mathcal{P}_1$ is about inferences that can be made by an observer far away. If smoke is seen coming from the mountain top ($p_1$), then an inference is made that there is fire on the mountaintop ($p_2$). The fire itself cannot be seen by the observer. The constraint formalizes the justification of the inference, in this case ruling out other possible causes of smoke.
The $\mathrm{At}$ rule allows for the decomposition between rules specific to the context $\baseP_1$, 
and rules formalizing context-independent facts, such as the rule for the constraint  (\ref{rule:SmokeFire}), within $\baseB$. 
%In this example, we have used this to separate the introduction of the fact $p_1$ from $\mathcal{P}_1$, from the rule (\ref{rule:SmokeFire}) of $\mathcal{B}$ such wherein the observer accepts this conditional inference. For example, if $a$ was in a different situation say $\mathcal{P}_1 '$, where smoke was not observed, then 
%\begin{equation*}
%\inferonSub{p}{2}{P'}{1}{1}{} 
%\end{equation*}
%would not be supported over $\mathcal{B}$. 
\fillBox
\end{example}

\begin{example}[{Inferential capability via modality}]
\label{example:NoSmokeModality}
In Example~\ref{example:NoSmoke}, the availability of the rule (\ref{rule:SmokeFire}), the constraint linking inferonic bases, was central to the observer drawing the conclusion of interest. Different agents may have different  capabilities, say, because they are attuned to different constraints. The agent modalities, $\agentd{a}$ and $\agentb{a}$, introduced in Section~\ref{sec:LogicalBasis} are supported using decomposition of a base into the part used by a specific, named, agent, and an ambient part. This allows for logical expressions that capture different agent capabilities.

We continue with the set-up from Example~\ref{example:NoSmoke}. 
Suppose that we have an agent $a$ that can reason with the inferonic base $\baseB$, so that  $\mathcal{B} \in {\bf Bases}(a)$. 
In that case,  
%\begin{equation*}
$
    \Vdash _{\emptyset} \agentd{a} (\inferonSub{P}{2}{\baseP}{1}{1}{}) 
$
%\end{equation*}
is an instance of the support relation. 

Suppose another agent $a'$ with, say, ${\bf Bases}(a') = \{ \mathcal{B}' \}$, 
where $\mathcal{B}'$ includes no rules with a conclusion $\inferatom{P_2}{1}{}$.
We do not have 
%\begin{equation*}
$
    \agentd{a'}(\inferonSub{P}{2}{P}{1}{1}{})
$
%\end{equation*} 
supported over the empty base. \fillBox
\end{example}

%\begin{itemize}
%\item[--] Can we get at Example 1 along lines something like:

%\[
%    \dfrac{\langle p_1 , 1 \rangle}{\langle p_2 , 1 \rangle} \in \mathcal{B}
%\]
%\item[--] And perhaps it's natural in this setting to go with using agency? 
%
%\[
%    \Vdash_\mathcal{B} \poss{a} \langle p , \mathcal{P} , \mathfrak{b} \rangle
%\]
%Idea being that the agent reasons about the remote location by adding additional capability. 
%\item[--] Hunch: the classic ST example is a bit like the 
%classic MC example --- gets away with a lot because of 
%the way it's set up. 
%\end{itemize}

\begin{example}[The wise men]
\label{example:WiseMen}
The following example is a variant of a well-known puzzle:
\begin{quote}
`A certain king wishes to test his three wise men $a_1$, $a_2$, $a_3$. Wise man $a_1$ can see both $a_2$ and $a_3$. Man $a_2$ can see only $a_3$. Man $a_3$ is blindfolded. They all know this to be the case. 
The king tells them that he will put
a white or black spot on each of their foreheads but that at least one spot will
be white. In fact, all three spots are white. He then repeatedly asks them, in turn (first man, second man, third man), ``Do
you know the colour of your spot?'' What do they answer?'
\end{quote} 

With agents $a_1$, $a_2$, $a_3$ who make sufficient use of information, the sequence should go `no', `no', `yes'. Critical to this are the inferences that are drawn from announcements that are themselves tied to inferences that cannot be made by agents earlier in the sequence. It assumes that the second man can draw an inference from an inference not made by the first ($a_2$ must be able to infer reasons why $a_1$ was unable to infer his own colour). It assumes that the second man can combine the result of this inference with an observation to draw an inference that he may not infer his own colour  ($a_2$ can see that $a_3$'s spot is white, and so $a_2$ cannot infer his own colour). $a_3$ draws the same inference as $a_2$ from $a_1$'s announcement. Man $a_3$ infers from $a_2$'s announcement that his own spot must be white, for $a_2$ did not infer his own colour. 

Our focus is on 
how information, in the form of inferons, is produced for each agent in turn, using  announcements of agents earlier in the chain. Below we consider the case in which the agents are `wise' in the sense that they directly incorporate rules in their bases that enable them to make full and correct use of the information available.  

We suppose a basic situation $\mathcal{P}_1$, $\mathcal{P}_2$, $\mathcal{P}_3$ for each wise man, relating to inferences that he can make using only the observations visually made of colour (for example, `if the other two have black spots then I have a white spot' and `he has a black spot'). Specifically, each $\mathcal{P}_i$ does not capture the inferences from inferences-or-not flowing from the announcements. The first man has access only to $\mathcal{P}_1$. The second man has access to $\mathcal{P}_2$. The third man has access only to $\mathcal{P}_3$. 

Let $W(a_n)$ be the proposition `Man $n$ has a white dot' for each $n=1,2,3$. 
An inferonic atom $\inferatom{W(a_n)}{1}{}$ is an assertion of this fact, and an inferonic atom $\inferatom{W(a_n)}{0}{}$ is a denial of it.
Let $P_n$ be the proposition `Man $n$ knows the colour of his own  dot'. 
An inferonic atom $\inferatom{P_n}{1}{}$ is therefore an affirmation of the content of this announcement (that $a_n$ knows his own colour), and an inferonic atom $\inferatom{P_n}{0}{}$ is a denial of it. 

\medskip 
\noindent{(The first man).} 
The first man, $a_1$, has his visual observations, which we formalize as inferences 
\begin{equation*}
\dfrac{}{\inferatom{W(a_2)}{1}{}} 
\qquad 
\dfrac{}{\inferatom{W(a_3)}{1}{}}
\end{equation*} 
in a base $\mathcal{P}_1^O$. 
The first man also has `static' rules
\begin{equation}
\label{rule:man1used}
\dfrac{\inferatom{W(a_2)}{1}{}}{\inferatom{P_1}{0}{}}
\qquad 
\dfrac{\inferatom{W(a_3)}{1}{}}{\inferatom{P_1}{0}{}}
\end{equation}
and 
\begin{equation}
\label{rule:man1unused}
    \dfrac{\inferatom{W(a_2)}{0}{} 
    \qquad
    \inferatom{W(a_3)}{0}{}}
    {\inferatom{P_1}{1}{}}
\end{equation}
in a base $\baseP_1 ^S$. Define 
$\baseP_1 = \baseP_1^O \cup \baseP_1^S$. 
We therefore have a derivation of $\inferatom{P_1}{0}{}$ over $\baseP_1$. We then have
%\begin{equation*}
$
\Vdash_{\baseP _1 ^O} \inferonSubSup{P}{1}{P}{1}{S}{0}{} 
$, 
%\end{equation*}
while $\inferonSubSup{P}{1}{P}{1}{S}{1}{}$ is not supported over $\baseP _1^O$. The rule (\ref{rule:man1unused}) which is not, in fact, applied when all agents are dotted white, is justified directly by consideration of the King's announcement.

%If we take $\mathcal{P}_1^S \in {\bf Bases}(a_1)$, then 
%\begin{equation*}
%    \Vdash _{\emptyset} \AgPoss{a_1} (\inferonSubSup{p}{1}{P}{1}{0}{0}{}) 
%\end{equation*}
%abstracts away from naming the particular basis used by $a_1$ to witness support of $\langle p_1 , \mathcal{P}_1 ^ O , 0 \rangle$

%\newcommand{\inferonSubSup}[7]{\llangle {#1}_{#2} , \mathcal{#3}_{#4}^{#5} , \mathfrak{#6}_{#7} \rrangle}
%\inferonSubSup{p}{1}{P}{1}{0}{0}{}

\medskip 
\noindent{(The second man).}  
The second man has rules 
\begin{equation*}
\dfrac{}{\inferatom{W(a_3)}{1}{}}
\qquad 
\dfrac{}{\inferatom{P_1}{0}{}}
\end{equation*}
relating to observations and announcements.  
Let these be elements of base $\mathcal{P}_2^O$.
%\begin{equation*}
%\dfrac{}{\inferatom{p_1}{0}{}}
%\end{equation*}
The second man has static rules 
\begin{equation}
\label{rule:TheRuleUsed}
\dfrac{\inferatom{P_1}{0}{} \qquad \inferatom{W(a_3)}{1}{}}{\inferatom{P_2}{0}{}} 
\end{equation}
and  
\begin{equation}
\label{rule:Man2RuleUnused}
\dfrac{\inferatom{P_1}{0}{} \qquad \inferatom{W(a_3)}{0}{}}{\inferatom{P_2}{1}{}} 
\end{equation}
and
\begin{equation}
\label{rule:Man2RuleUnused2}
    \dfrac{\inferatom{P_1}{1}{}}
    {\inferatom{P_2}{1}{}} .
\end{equation}
Let these static rules be 
in a base $\baseP _2^S$. 
Take $\baseP_2 = \baseP_2^O \cup \baseP_2^S$.
There is a derivation of $\inferatom{P_2}{0}{}$ 
over the base $\baseP_2$ using the first of the two rules from $\baseP _2^S$. 
There is the following instance of the support relation 
%\begin{equation*}
$
    \Vdash _{\mathcal{P}_2^O} 
    \inferonSubSup{P}{2}{P}{2}{S}{0}{}
$, 
%\end{equation*}
%so that if $a_2$ is immersed in inferential context $\mathcal{P}_2^O$, the denial of $P_2$ will be supported, because of the static rules from $\mathcal{P}_2^S$. However
but $\inferonSubSup{P}{2}{P}{2}{S}{1}{}$ 
is not supported over $\mathcal{P}_2^O$. 

Notice that support for a disjunction 
$\inferonPlain{W(a_2)}{\baseP _1 ^S}{1}{} \vee \inferonPlain{W(a_3)}{\baseP _1 ^S}{1}{}$ 
over $\baseP _2 ^O$ is not required. 

The 
rule (\ref{rule:TheRuleUsed}) can be justified by considering $\baseP _1$ with $\baseP _1 ^S$ fixed as above, but an unknown 
$\baseP _1 ^O$ that is a subset of 
$
\{ 
\Rightarrow \inferatom{W(a_2)}{0}{} , 
\Rightarrow \inferatom{W(a_2)}{1}{} , 
\Rightarrow \inferatom{W(a_3)}{0}{} , 
\Rightarrow \inferatom{W(a_3)}{1}{} 
\}
$.
Suppose 
$\Vdash _{\baseP _1 ^O} \inferonPlain{P_1}{\baseP _1 ^S}{0}{}$ and 
$\Vdash _{\baseP _1 ^O} \inferonPlain{W(a_3)}{\baseP _1 ^S}{1}{}$. 
%It must then be that $\vdash _{\baseP_1} \inferatom{W(a_2)}{1}{}$ or $\vdash _{\baseP_1} \inferatom{W(a_3)}{1}{}$ ...
However, the rules (\ref{rule:man1unused}) allow these support relation instances to hold for both 
$
\baseP _1 ^O = 
\{ 
\Rightarrow \inferatom{W(a_2)}{0}{} ,  
\Rightarrow \inferatom{W(a_3)}{1}{} 
\}
$
and 
$
\baseP _1 ^O = \{ 
\Rightarrow \inferatom{W(a_2)}{1}{}  , 
\Rightarrow \inferatom{W(a_3)}{1}{} 
\}
$. That is, one cannot infer either of  
$\inferatom{W(a_2)}{1}{}$ or $\inferatom{W(a_2)}{0}{}$.
%, but the latter is impossible assuming consistency (Definition~\ref{def:consistentInferonicBase}).
By contrast, for rule (\ref{rule:Man2RuleUnused}), if we have both   
$\Vdash _{\baseP _1 ^O} \inferonPlain{P_1}{\baseP _1 ^S}{0}{}$ and 
$\Vdash _{\baseP _1 ^O} \inferonPlain{W(a_3)}{\baseP _1 ^S}{0}{}$, then it must be the case that $\vdash _{\baseP _1} \inferatom{W(a_2)}{1}{}$. 
For rule (\ref{rule:Man2RuleUnused2}), if 
$\Vdash _{\baseP_1 ^O} \inferonPlain{P_1}{\baseP _1 ^S}{1}{}$ for unknown $\baseP _1 ^O$, then it must be the case that 
$\vdash _{\baseP _1} \inferatom{W(a_2)}{1}{}$. 
%Again, taking $\mathcal{P}_2^S \in {\bf Bases}(m_2)$, we can abstract away from mentioning the particular witnessing `static' part of the base 
%\begin{equation*}
%    \Vdash _{\emptyset} \agentd{a_2} 
%    (\inferonSubSup{p}{2}{P}{2}{0}{0}{})
%\end{equation*}

\medskip 
\noindent{(The third man).} 
The third man has observations from announcements 
\begin{equation*}
\dfrac{}{\inferatom{P_1}{0}{}}
\qquad
\dfrac{}{\inferatom{P_2}{0}{}}
\end{equation*}
in a base $\mathcal{P}_2^O$. He has static rules
\begin{equation}
\label{rule:Man3RuleUsed}
    \dfrac{\inferatom{P_1}{0}{} \qquad \inferatom{P_2}{0}{}}{
    \inferatom{P_3}{1}{}
    }
\end{equation}
and 
\begin{equation}
\label{rule:Man3RuleUnused1}
    \dfrac{\inferatom{P_1}{1}{}}{\inferatom{P_3}{0}{}}
\end{equation}
and
\begin{equation}
\label{rule:Man3RuleUnused2}
    \dfrac{\inferatom{P_2}{1}{}}
    {\inferatom{P_3}{0}{}} 
\end{equation}
in base $\mathcal{P}_3^S$. 
Taking $\baseP _3 = \baseP_3^O \cup \baseP_3^S$, there is a derivation of $\inferatom{P_3}{1}{}$ over $\mathcal{P}_3$. We have that
\begin{equation*}
    \Vdash _{\mathcal{P}_3^O} 
    \inferonSubSup{P}{3}{P}{3}{S}{1}{}
\end{equation*}
but $\inferonSubSup{P}{3}{P}{3}{S}{0}{}$ 
is not supported over $\mathcal{P}_3^O$. 

In particular, rule (\ref{rule:Man3RuleUsed}) can be justified by reasoning about the second man's available inferences. If 
$\Vdash _{\baseP _2 ^O} \inferonPlain{P_2}{\baseP_2 ^S}{0}{}$ where $\baseP_2 ^S$ is as above, but $\baseP _2 ^O$ is an unknown subset of 
$
\{ 
\Rightarrow \inferatom{P_1}{1}{} , 
\Rightarrow \inferatom{P_1}{0}{} , 
\Rightarrow \inferatom{W(a_3)}{1}{} , 
\Rightarrow \inferatom{W(a_3)}{0}{}
\}
$, 
then it must be that 
$\vdash _{\baseP_2} \inferatom{W(a_3)}{1}{}$. 
The rule (\ref{rule:Man3RuleUnused2}) is also justified by considering $\Vdash _{\baseP _2}$ for unknown $\baseP _2 ^O$, while the rule (\ref{rule:Man3RuleUnused1}) can be justified by considering $\Vdash_{\baseP _1}$ for unknown $\baseP _1 ^O$ but known $\baseP _1 ^S$.

We have chosen this particular puzzle in order to give a simple demonstration of how elements of epistemic reasoning can be formalized within our present framework and to show how inferonic information cascades from one agent to the next. More complex variants of this puzzle involve higher iterations of agent knowledge about the knowledge of other agents. Our approach can be contrasted with state-based reasoning as in many classical modal accounts of these types of problems epistemic logic. More complex modal epistemic reasoning in base-extension semantics is explored in a separate work \cite{EckhardtPym2024PAL}. 

What we have not done, is to give a complete formalization of all relevant elements of `wise' reasoning. That is, while we have described how rules such as (\ref{rule:TheRuleUsed}) and (\ref{rule:Man3RuleUsed}) can be justified by meta-theoretic reasoning (to deduce the content of other agents' bases), we have formalized only input-output relationships on `obervable' atoms as rules, not the reasoning itself. Such formalization would require a much more intricate set-up, presumably involving bases used by agents that directly describe possible inferences by other agents, or the importation of semantic concepts. Again, related modal directions are explored in \cite{EckhardtPym2024PAL} without the concept of inferon.
\fillBox 
\end{example}

%The account given in Example~\ref{example:WiseMenOversimplified} is rather shallow: it simply encodes in the bases everything we needed to arrive at the correct sequence of `announcements'.   What is more interesting is why these are appropriate bases. We note the following:
%\begin{itemize}
%    \item[--] Rule~(\ref{rule:man1used}) in $\mathcal{P}_1^S$ are only correct because the first man has \emph{only} observations from the visual inspection of the other two men
%    \item[--] The second man has only observation of the third man, plus what he can infer from announcement by the first man, and this is what makes rule (\ref{rule:TheRuleUsed}) correct
%    \item[--] If second man understood the base used by the first man,  then he could justify (\ref{rule:TheRuleUsed})
%    \item[--] In this situation, the second man does not use the rule (\ref{rule:Man2RuleUnused})
%    \item[--] If the third man understood the base used by the second man, specifically including (\ref{rule:Man2RuleUnused}), and also the base used by the first man, then he could justify his rule (\ref{rule:Man3RuleUsed}). 
%\end{itemize}
%This suggests that each the second and third men need to reason about the bases, or the support relation, of other men. 

We conjecture that comparing this representation of the `wise men' problem with its representations in, for example, model-theoretic epistemic logics (see, for example, \cite{Fagin1995}) might expose a rich analysis. However, such a project is beyond our present scope.  

\section{Inferomorphisms and Information Flow} \label{sec:morphisms-flow}

As we have previously mentioned, in Section~\ref{sec:situations}, in order to produce a simpler treatment of information flow, Barwise, Seligman and others developed the theory of classifications and infomorphisms, under the name `The Logic of Distributed Systems' which we abbreviate as LDS. The standard introductory reference for this is the monograph \cite{BarwiseSeligman1997}. Before proceeding to develop our inferentialist counterpart to these ideas, we summarize some background on LDS from \cite{BarwiseSeligman1997}. 

 A \emph{classification} is a structure $\mathbf{A} = \langle \Sigma _\mathbf{A} , A , \vDash _{\mathbf{A}} \rangle$, where $A$ is a set (of `tokens'), $\Sigma _\mathbf{A}$ is a set of `types', and 
 $\vDash \ \subseteq A \times \Sigma_{\mathbf{A}}$ 
 is a binary relation. We think of a relational instance $a \vDash _\mathbf{A} \alpha$ as saying that the token $a$ is classified as being of type $\alpha$. Classification abstracts from the support relation of situation theory, where there are  situations classified $s \vDash T$ by types $T$, including the type $T$ of situations that satisfy a particular infon: $T = [ \dot{s} \mid \dot{s} \vDash \sigma ]$. 

An \emph{infomorphism} $f:\mathbf{A} \longrightarrow \mathbf{C}$ from classifcation $\mathbf{A}$ to classification $\mathbf{C}$ consists of a pair of functions 
$f^\wedge : \Sigma _\mathbf{A} \longrightarrow \Sigma _\mathbf{C}$, 
$f^\vee  : C \longrightarrow A$
%\begin{align*}
%f^\wedge &: \Sigma _\mathbf{A} \longrightarrow \Sigma _\mathbf{C}\\
%f^\vee   &: C \longrightarrow A  
%\end{align*}
such that the Chu condition (see \cite{BarwiseSeligman1997} for the background) holds:
\begin{equation}
\label{eq:Chu}
%\tag{Chu}
    f^\vee(c) \Vdash_\mathbf{A} \alpha \mbox{ iff } c \Vdash_\mathbf{C} f^\wedge(\alpha)
\end{equation}
for every $c \in C$ and $\alpha \in \Sigma_\mathbf{A}$. 

 Just as in situation theory, constraints are still viewed as the primary concept to analyse information flow, but general situated propositions are replaced by assertions about tokens being assigned to classes, while (logical) constraints themselves are understood as arising from certain configurations of the concept of infomorphism known as `channels'. We remark that Mares \cite{Mares1996} and Restall \cite{Restall1996} discuss  how constraints can be captured model-theoretically in the Routley--Meyer ternary-relation semantics for relevance logic. %Related ideas are explored by Restall \cite{Restall1996-IFRL}.  

A \emph{channel} is an indexed family 
$( f_i : A_i \longrightarrow C \mid i \in I)$ 
of infomorphisms, where each member of the family has the same target, $C$, called the \emph{core} of the channel. Channels provide a means to identify distributed systems. A system, in a particular configuration/state, is identified by a single $c \in C$. The distributed components are the $a_i = f_i^\vee (c) \in A_i$. Components $a_i$ and $a_j$ are \emph{connected} if there is some $c$ such that $a_i = f^\vee (c)$ and $a_j = f^\vee (c)$.

Later, in Section~\ref{sec:flow}, we shall apply the tools we develop herein to a quite generic notion of distributed system that, while consistent with the 
ideas of \cite{BarwiseSeligman1997}, derives from quite practical modelling concerns (see, for example, \cite{CIP2022,ICP2025} and the references therein).  

For an example of our present set-up, we can imagine a flashlight as a system consisting of a bulb and a switch (eliding other relevant components). We can imagine a domain $S$ of switches and a domain $B$ of bulbs. Taking $F$ to be the domain of flashlights, each flashlight connects a unique switch and a unique bulb. This provides mappings $f_S ^\vee : F \longrightarrow S$ and 
$f_B ^\vee :F \longrightarrow B$. An important type for bulbs is $\mathrm{LIT}$, those that are showing light, and we can imagine this to be recognizable also for flashlights, $f_B ^\wedge (\mathrm{LIT})$. Similarly, for switches, we have a type $\mathrm{ON}$ and for flashlights $f_S ^\vee (\mathrm{ON})$. 

For any classification $A$, a \emph{sequent} of $A$ is a pair $\langle \Gamma , \Delta \rangle$ of sets of types of $A$. A token $a \in A$ \emph{satisfies} a sequent $\langle \Gamma , \Delta \rangle$ if $a \vDash _A \alpha$ for every $\alpha \in \Gamma$ implies that $a \vDash _A \alpha '$ for some $\alpha ' \in \Delta$. The theory $\mathrm{Th}(A)$ is the set of all sequents that are satisfied by all $a \in A$.

The central proposal for the notion of information flow now goes quite naturally. Suppose a channel that is made from infomorphisms including $f:A \longrightarrow C$ and $g:B \longrightarrow C$ over the core $C$. Suppose that $a$ is a token of type $A$, $\alpha$ is a type of $A$, $b$ is a token of type $B$, and $\beta$ is a type of $B$. Then, relative to $C$, we say $a \vDash _A \alpha$ \emph{carries the information} that $b\vDash _B \beta$ if $a$ and $b$ are connected and $\langle f^\wedge( \alpha) , g^\wedge (\beta) \rangle$ is in $\mathrm{Th}(C)$.

\subsection{Inferomorphisms} 
\label{subsec:pred-inferomorphisms}

We now develop our inferentialist analysis, taking inspiration from the methodology of LDS.
%Recall that in the channel approach, a way of identifying which components were attached to which, by which connections, was required, and the solution involved the tokens of the channel core. A similar issue arises when using inferomorphisms, and so we move to a logical setting in which we can discuss individuals. Having made this move, it becomes natural to allow use of variables ranging over individuals and quantifiers, in the language of inferons, over such variables. 
Our target remains to produce an analysis of  information flow in distributed systems. To that end, we need to be able to identify and reason about definite components of the system. We use a first-order language, and instead of taking individual components to be elements from domains that are sets, individuals must simply be identified as equivalence classes of closed terms under the equality relation. 

%\notemc{(MC) I think we might get further if we have many-sorted logic. We want individuals of one sort represented by (equivalence classes) of terms of that sort. It gets messy with partial functions otherwise, and examples become cumbersome. Has many-sorted logic been done in BES?}

%This allows us to do is to replace the unanalysed atoms $p$ contained in inferons with relational instances from the term language of a first-order language. This allows us to consider a notion of inferomorphism that allows also for analysis with respect to  transfer of closed terms (as representations of `modelled' individuals) alongside, but in the opposite direction to, transfer of relation symbols. This is similar to the infomorphism definition. 

We use the notational conventions from Figure~\ref{fig:notation}. In particular, terms $t$ are generated from: individual constants $a,b,c,$ \ldots; individual variables $x,y,z$ \ldots;
fixed arity relation/predicate symbols $P$, $Q$, $R$ \ldots for various arities. 
Recall that $\setofterms$ is the set of terms, and $\cl{\setofterms}$ is the set of closed terms of this language. 
For each natural number $n$, let $\mathrm{Pred}_n$ be the set of $n$-ary relation symbols of the language.

% \notemc{There is a notational inconsistency between Figures 2 and 3 on the name of the set of terms: $T$ or $\mathcal{T}$.}

Below we write expressions for terms that are potentially undefined. We use the Kleene-equality symbol below: for two such expressions $t_1$ and $t_2$, we write $t_1 \simeq t_2$ if $t_1 = t_2$ whenever $t_1$ and $t_2$ are both defined, and $t_1$ is defined if and only if $t_2$ is defined.

\begin{definition}
A \emph{pre-inferomorphism $f: \mathcal{P} \longrightarrow \mathcal{P'}$ over $\mathcal{B}$} consists of a partial function
\begin{equation*}
    f^\vee: \cl{\setofterms} \longrightarrow \cl{\setofterms}
\end{equation*}
and a dependent function
\begin{equation*}
    f^\wedge : \prod _{n \in \mathbb{N}}:( \mathrm{Pred}_n \longrightarrow  \mathrm{Pred}_n) .
\end{equation*}
The partial function $f^\vee$ is required to have a substitution property: if $t_1 = t_2$, then $f^\vee (t_1) \simeq f^\vee (t_2)$. The pair of functions are required to satisfy the condition below:
    \begin{equation}
    \label{eq:multiChu}
        \begin{array}{rcl}
        \vdash_{\mathcal{C} \cup \mathcal{P}'}
        \inferatom{R(f^{\vee}(t_1) , \ldots , f^{\vee}(t_n))}{b}{}  
        & \mbox{iff} & 
        \vdash_{\mathcal{C} \cup \mathcal{P}} 
        \inferatom{f^{\wedge}(n)(R)(t_1 , \ldots , t_n)}{b}{} 
        \end{array}
    \end{equation}
    for every $\mathcal{C} \supseteq \mathcal{B}$,  every $n\geq 0$, every relation symbol $R$ of arity $n$, and every sequence of closed terms $t_1 , \ldots t_n$ of length $n$.
\fillBox
\end{definition}

Given a pre-inferomorphism $f:\mathcal{P} \longrightarrow \mathcal{P}'$ over $\baseB$, we define:
\[
\begin{array}{rcll}
     \inferon{R(t_1 , \ldots , t_n)}{P}{b}{}^{f , \vee} 
     & \simeq &  
    \inferon{R(f^\vee (t_1) , \ldots , f^\vee (t_n) )}{P}{b}{}
    & \\  
%2nd
     \inferon{R(t_1 , \ldots , t_n)}{P''}{b}{}^{f , \vee} 
     & = &  
    \inferon{R(t_1 , \ldots , t_n)}{P''}{b}{} 
     & 
    \mbox{for $\baseP'' \notequiv \baseP$}   \\
%3rd
 \inferon{R(t_1 , \ldots , t_n)}{P}{b}{}^{f , \wedge} & = & 
    \inferon{(f^\wedge (R))(t_1 , \ldots , t_n)} {P'}{b}{} \\
% 4th
 \inferon{R(t_1 , \ldots , t_n)}{P''}{b}{}^{f , \wedge} & = & \inferon{R(t_1 , \ldots , t_n)}{P''}{b}{} 
 & \mbox{for $\baseP'' \notequiv \baseP$}  
\end{array}
\]
and then for $\dagger \in \{ \wedge , \vee \}$, indicating the components of the morphism, we take 
\[
\begin{array}{rcll}    
    (\phi_1 \circ \phi_2) ^{f,\dagger} &= &
    (\phi_1 ^{f,\dagger}) \circ (\phi_2 ^{f,\dagger}) &
    \mbox{ for } \circ \in \{ \wedge , \vee , \supset \} \\
    \bot ^{f , \dagger} & = & \bot & \\
    \Theta ^{f, \dagger } & = & \{ \psi ^{f,\dagger} \mid \psi \in \Theta \} & \\
    (\Theta  \Vdash_\baseB \phi) ^{f,\dagger} 
    & = & (\Theta ^{f,\dagger}) \Vdash_\baseB (\phi^{f, \dagger}) .
\end{array}
\]
In particular, this defines a function $\iota \mapsto \iota^{f, \wedge}$ and a partial function $\iota \mapsto \iota^{f, \vee}$ on inferons. 

\begin{definition}
    A pre-inferomorphism is a \emph{quasi-inferomorphism} if $\iota \mapsto \iota^{f, \vee}$ and $\iota \mapsto \iota^{f, \wedge}$ are surjective. \fillBox
\end{definition}

%For intuition and comparison with the infomorphism concept based on the unary-relation case, think of $R$ as a relation interpreted on some `model' consisting of some set $A$ of  terms (closed under equality) and $t_1$ \ldots $t_n$ as denoting individuals in another such`model' $B$; we then have each of individuals $f^{\vee}(t_i)$ in $A$ and relation $f^{\wedge}(n)(R)$ on $B$. Under this convention,  $f^\wedge$ travels in the same direction as $f$, while $f^\vee$ travels in the opposite direction. At the same time, there is a substitution of the base.
%
%Condition~(\ref{eq:multiChu}) is an adaptation of the Chu condition (\ref{eq:Chu}) for informorphisms. It takes this indexed form because we are no longer only interested in `types' that are, essentially, unary predicates. 
 
%For every $\mathcal{B}$, there is a category with objects the elements of the basis $\baseMill$  and morphisms the inferomorphisms over $\mathcal{B}$. 
Note that $(fg)^\vee = g^\vee \circ f^\vee$ and $(fg)^\wedge = f^\wedge \circ g^\wedge$ for $f: \mathcal{P} \longrightarrow \mathcal{P}'$ and $g:\mathcal{P}' \longrightarrow \mathcal{P''}$.
The following substitution result can then be shown.

%\notemc{It also works with $\AgPoss{a}$.}
\begin{proposition}
\label{result:InferonChuEquivalence}
Suppose the first-order language of inferons and base-extension semantics as set out in Figure~\ref{fig:support2}. Let $f: \mathcal{P} \longrightarrow \mathcal{P}'$ be a quasi-inferomorphism over $\mathcal{B}$. Let $\phi$ be a formula and $\Theta$ be a finite set of formul{\ae}. Then 
    \begin{equation}
    \label{equation:InferonChuEquivalence}
        ( \Theta \Vdash _\mathcal{B} \phi)^{f , \vee} \mbox{ iff }
        ( \Theta  \Vdash _\mathcal{B} \phi )^{f, \wedge}
    \end{equation}
\end{proposition}

\begin{proof}
We show the stronger result that     
$( \Theta \Vdash _\baseC \phi)^{f , \vee}$ if, and only if 
$( \Theta  \Vdash _\baseC \phi )^{f, \wedge}$,
for all $C \supseteq \mathcal{B}$. 

We use a function $\gamma$ that defines complexity $\gamma(\Theta , \phi)$ of pairs $(\Theta , \phi )$. First, define (overloading the symbol $\gamma$) complexity of a formula: 
\[
   \begin{array}{rcl}
    \gamma ( \iota  ) & = & 1\\ 
     \gamma(\bot) & = & 1 \\
     \gamma (\phi \circ \phi') & = & 1+ \gamma (\phi) + \gamma (\phi') \mbox{ for } \circ \in \{ \wedge, \supset , \vee \} \\
     %\gamma (\phi \supset \phi') & = & \gamma (\phi) + \gamma (\phi') + 1 \\
     %\gamma (\phi\vee \phi') & = & \gamma (\phi) + \gamma (\phi') + 1 \\
 %    \gamma(\AgentPoss{a} \phi) & = & 1 + \gamma ( \phi ) \\
     \gamma (\forall x . \phi) & = & 1 + \gamma ( \phi) \\
     \gamma (\exists x . \phi) & = & 2 + \gamma ( \phi ) . % Need `2+...' for the induction
    \end{array}
\]
From there, extend this to complexity of pairs $(\Theta , \iota)$:
\begin{equation*}
        \gamma ( \Theta , \iota ) = 1+ \max ( \{ \gamma( \iota ) \} \cup \{ \gamma( \kappa ) \mid \kappa \in \Theta \} ) .
 \end{equation*}
%The map $F$ preserves complexity: $\gamma(F(\iota)) = \gamma (\iota )$ for all $\iota$.

  The proof works by induction on complexity of pairs $(\Theta , \phi)$ that appear in support relation instances for some base, with $\Theta$ on the left of $\Vdash$ and $\phi$ on the right. The induction hypothesis is that for all $(\Theta ' , \phi ')$ less complex than $(\Theta , \phi)$, and all bases 
 $\mathcal{C} \supseteq \mathcal{B}$, that 
    \begin{equation*}
    \label{eq:InferonEquivalenceIH}
        (\Theta' \Vdash_\mathcal{C} \phi')^{f,\vee} 
        \quad\mbox{iff}\quad 
        (\Theta' \Vdash_{\mathcal{C}} \phi ' ) ^{f,\wedge}
    \end{equation*}
    Given the induction hypothesis, we show that (\ref{equation:InferonChuEquivalence}) holds. 

\medskip
 ($\bot$) Immediate, since $\bot^{f,\vee} = \bot = \bot^{f, \wedge}$ and $\Theta = \emptyset$

\medskip
 (At) 
 We have $\Theta = \emptyset$. We take $\phi  ' = \iota = \langle R(t_1 , \ldots , t_n) , \baseP'', \mathfrak{b} \rangle$. In the case where $\baseP'' \notequiv \baseP$ the result is trivial because the two operations, $()^\vee$ and $()^\wedge$ make no change to the inferon.

In the case where $\baseP'' = \baseP$, we have: 
\[
\begin{array}{rcl}
\Vdash _{\mathcal{C}} \inferon{R(t_1 , \ldots , t_n)}{P}{b}{}^{f, \vee} 
    & \mbox{iff} & 
\Vdash _{\mathcal{C}} \inferon{R(f^\vee( t_1) , \ldots ,f^{\vee} (t_n))}{P}{b}{} \\
    & \mbox{iff} & \vdash_{\mathcal{C} \cup \mathcal{P}} 
\inferatom{R(f^\vee(t_1) , \ldots , f^{\vee} (t_n))}{b}{} \\
& \mbox{iff} & \vdash _{\mathcal{C} \cup \mathcal{P}' } \inferatom{(f^{\wedge}(n)(R))( t_1 , \ldots t_n)}{b}{} \\ & \mbox{iff} & 
\Vdash_{\mathcal{C}} \inferon{(f^{\wedge}(n)(R))(t_1 , \ldots t_n )}{P'}{b}{} \\
& \mbox{iff} &  \Vdash _{\mathcal{C}} \inferon{R(t_1 , \ldots , t_n)}{P}{b}{}^{f , \wedge}
\end{array}
\]

\medskip 
($\wedge$, $\supset$, $\forall x$) Straightforward use of the induction hypothesis, noting that the `only if' and `if' directions of the proof rely, respectively, only on the same direction (`only if', `if') in the induction hypothesis. For $\supset$, the fact that $\gamma (\phi_1, \phi_2) < \gamma (\phi_1 \supset \phi _2)$ is used.

\medskip
(Inf and $\vee$) Straightforward application of the induction hypothesis, noting that both the `if' and `only if' directions in the hypothesis are required to prove each of the `if' and `only if' directions. For the case of ($\vee$), surjectivity of $\iota \mapsto \iota ^{f,\wedge}$ is used in showing the `only if' direction, while surjectivity of $\iota \mapsto \iota ^{f,\vee}$ is used in showing the `if' direction. 

\medskip
($\exists x)$ The `only if' direction uses the surjectivity of $\iota \mapsto \iota ^{f,\wedge}$, and the `if' direction uses the surjectivity of $\iota \mapsto \iota^{f, \vee}$. Both directions use both of the `if' and `only if' directions of the induction hypothesis.
\end{proof}

%\notemc{Compare this result to f-Intro and f-Elim rules from B and S... section 2.3}

A \emph{stock} $(M, \baseP)$ consists of a base $\baseP$ and a set of closed terms $M$ that is closed under equality. 
%A stock $(M,\baseP)$ features the set $M$ containing possible components of a larger system, and a base $\baseP$ that may be used to support facts about elements of $M$.
Recall that the \emph{domain} of a partial function is the set of elements where it is defined.
\begin{definition}
Let $(M, \baseP')$ and $(N , \baseP)$ be stocks. 
An \emph{inferomorphism 
$f:(N,\baseP) \longrightarrow (M, \baseP')$ over $\baseB$} is a quasi-inferomorphism $f:\baseP  \longrightarrow \baseP'$ over $\baseB$ such that 
\begin{itemize}
    \item[--] the domain of the partial function $f^\vee$ is equal to $M$
    \item[--] $f^\vee$ defines a function from $M$ to $N$.
\end{itemize} \vspace{-5mm}  \fillBox
\end{definition}
Stocks and inferomorphisms over $\baseB$ form a category. 

\begin{definition}
Let $(L,\mathcal{P})$ be a stock. 
A \emph{channel over $\baseB$ with core $(L,\mathcal{P})$} is a family  
\[
(
f_i: (M_i,\base P_i) \longrightarrow (L,\mathcal{P})
\mid i \in I )
\]
of inferomorphisms over $\baseB$.  
Two terms $t_i \in M_i$ and $t_j \in M_j$ are \emph{connected} in the channel if there is $t \in L$ such that $t_i = f^\vee (t)$ and $t_j = g^\vee (t)$. \fillBox
\end{definition}

For a channel as above, we can think of an element $t \in L$ as identifying a particular system that is made up of connected components from the family $( f _i ^\vee (t) \in M_i \mid i \in I)$. The base $\baseP _i$ provides supports some facts about elements of $M_i$. The base $\baseP$ supports facts about the entire system. In particular, it will allows us to compare and combine sentences of the forms $f_i ^\wedge ( \phi _i )$ and $f_j ^\wedge ( \phi _j )$ where $\phi _i$ is a sentence supported by $\baseP _i$ and $\phi _j$ is a sentence supported by $\baseP _j$.

\begin{definition}
\label{def:InferonCarriesInferon}
Suppose a channel with core $(L,\baseP'')$ and with inferomorphisms $f:(M,\baseP) \longrightarrow (L,\baseP'')$ and $g : (N,\baseP') \longrightarrow (L,\baseP'')$ over 
$\baseB$. Let $R$ and $S$ be unary predicates. Let $t \in M$ and $t' \in N$ be terms.  Relative to the channel, we say $\Vdash_\base{B} \inferon{R(t)}{P}{b}{}$ \emph{carries the information that}  
$\Vdash_\baseB \inferon{S(t')}{P'}{b}{}$ if 
%\begin{itemize}
%\item 
$t$ and $t'$ are connected by the channel and
%\item 
for all terms $t'' \in L$, 
\begin{equation*}
\inferon{(f^\wedge (R)(t'')}{P''}{b}{} 
\Vdash_\baseB 
\inferon{(g^\wedge (S))(t'')}{\baseP''}{b}{} \vspace{-5mm}
\end{equation*} \fillBox
%\end{itemize}
\end{definition}

Suppose that we have a channel as in Definition~\ref{def:InferonCarriesInferon} and that $t = f^\vee(t'') \in M$ and $t' = g^\vee (t'') \in N$ for some $t'' \in L$. 
Considering Proposition~\ref{result:InferonChuEquivalence}, if 
$\Vdash_\base{B} \inferon{R(t)}{P}{b}{}$, then 
$\Vdash _\base{B} \inferon{(f^\wedge (R)(t'')}{P''}{b}{}$. If $\Vdash_\base{B} \inferon{R(t)}{P}{b}{}$ carries the information that  
$\Vdash_\baseB \inferon{S(t')}{P'}{b}{}$, then by Inf-clause in Figure~\ref{fig:support1} of the support relation it is the case that $\Vdash_\baseB 
\inferon{(g^\wedge (S))(t'')}{\baseP''}{b}{}$. Applying Proposition~\ref{result:InferonChuEquivalence} again, we find $\Vdash_\baseB \inferon{S(t')}{P'}{b}{}$. \fillBox

\begin{example}
 This example is adapted from an example regarding infomorphisms and classifications   \cite{BarwiseSeligman1997}. We can think of a flashlight as a system consisting of a switch and a bulb. Regard stock $(L,\mathcal{P''})$, modelling flashlights (by providing a set of terms identifying  flashlights and a base of facts about them), as a way of describing the interface between $(M,\mathcal{P})$, modelling switches, and $(N,\mathcal{P}')$, modelling bulbs. In particular, $\mathcal{P}''$ will support constraints that link the state of the bulb to the state of the switch. A closed term $t'' \in L$ determines both $t= f^\vee (t'') \in M$, a particular switch, and $t' = g^\vee (t'') \in N$, a particular bulb. The term $t''$, as the whole flashlight, represents the particular connection, so which switch is connected to which bulb and how they are connected.

We consider the relationship between inferons 
%\begin{equation*}
$
\inferon{\predON(x)}{P}{1}{}   
$
%\end{equation*}
saying that switch $x$ is on, and and inferons
%\begin{equation*}
$
\inferon{\predLIT{(y)}}{P'}{1}{} 
$
%\end{equation*}
saying that bulb $y$ is lit.
We can think of $\predON{}$ as an abbreviation for the predicate `is a switch and is on', and $\predLIT{}$ as meaning `is a bulb and is lit'. 
In particular, then, we must have
$
 \{ t \in \mathrm{Cl}(\mathrm{Terms}) \mid \ \vdash _\baseB \inferatom{\predON{(t)}}{1}{} \} \subseteq M$
 and 
 $\{ t \in \mathrm{Cl}(\mathrm{Terms}) \mid \ \vdash _\baseB \inferatom{\predLIT{(t)}}{1}{} \} \subseteq N $.
%\end{align*}
%\begin{align*}
% \{ t \in \mathrm{Cl}(\mathrm{Terms}) \mid \ \vdash _\baseB \inferatom{\predON{(t)}}{1}{} \} & \subseteq M \\
% \{ t \in \mathrm{Cl}(\mathrm{Terms}) \mid \ \vdash _\baseB \inferatom{\predLIT{(t)}}{1}{} \} & \subseteq N .
%\end{align*}

We are supposing a particular flashlight named by $t'' \in L$, and that its component switch, $t =f^\vee (t'') \in M$, is in the on state, as derivable in $\mathcal{P}$, so that we have both 
%\begin{equation*}
$
\vdash _\baseP 
\inferatom{\predON{(t)}}{1}{}
$
%\qquad \mbox{ and } \qquad
and 
\[
\Vdash _\emptyset 
\inferon{\predON{(t)}}{\baseP}{1}{} 
\]
for all bases $\baseP$.

Suppose also predicates $\predONflash{}$ and $\predLITflash{}$ such that
$
 \{ t \in \mathrm{Cl}(\mathrm{Terms}) \mid \ \vdash _\baseB \inferatom{\predONflash{(t)}}{1}{} \}  \subseteq L 
$ and 
$
 \{ t \in \mathrm{Cl}(\mathrm{Terms}) \mid \ \vdash _\baseB \inferatom{\predLITflash{(t)}}{1}{} \} \subseteq L 
$
%\begin{align*}
% \{ t \in \mathrm{Cl}(\mathrm{Terms}) \mid \ \vdash _\baseB \inferatom{\predONflash{(t)}}{1}{} \} & \subseteq L \\
% \{ t \in \mathrm{Cl}(\mathrm{Terms}) \mid \ \vdash _\baseB \inferatom{\predLITflash{(t)}}{1}{} \} & \subseteq L 
%\end{align*}
for all bases $\baseB$.
These predicates say `is a flashlight with a switch that is on', and `is a flashlight with a bulb that is lit', respectively.
For the purposes of this example, abbreviate $f^\wedge (1)$ as $f^\wedge$, as we are only interested in unary predicates. Define 
$
f^\wedge (\predON{}) =  \predONflash{}
$
and 
$
    g^\wedge (\predLIT{})  =  \predLITflash{} 
$.
%\begin{align*}
%    f^\wedge (\predON{})  & =  \predONflash{} \\ 
%    g^\wedge (\predLIT{}) & =  \predLITflash{} .
%\end{align*}
We are supposing the condition (\ref{eq:multiChu}), which guarantees that
$ % \begin{equation*}
\vdash_{\mathcal{P''}} 
\inferatom{(f^\wedge (\predON)){(t'')}}{1}{}  
$ %\end{equation*} 
since $\vdash _\baseP \inferatom{\predON{(f^\vee(t''))}}{1}{}$. 

Suppose stock $(L,\mathcal{P}'')$ represents only `normal' (properly functioning \cite{BarwiseSeligman1997}), flashlights, $x$: that is, suppose that 
\begin{equation*}
    \Vdash _{\mathcal{P''}} 
    \forall x . (
    \inferon{\predONflash{(x)}}{P''}{1}{} 
    \supset 
    \inferon{\predLITflash{(x)}}{P''}{1}{})
\end{equation*}
Then the constraint obtains that if the switch is on, then the bulb is lit holds at the flashlight $t''$: 
\begin{equation*}
\inferon{(f^\wedge (\predON{}))(t'')}{P''}{1}{}   
    \Vdash_{\mathcal{P''}} 
        \inferon{(g^\wedge (\predLIT{}))(t'')}{P''}{1}{} 
\end{equation*}
In other words,  
$\Vdash _{\baseP''} \inferon{\predON{(t)}}{P}{1}{}$ carries the information that 
$\Vdash_{\mathcal{P''}} \inferon{\predLIT{(t')}}{P'}{1}{}$.
From this it follows by the {Inf} clause of the support relation that 
$ %\begin{equation*}
 \Vdash_{\mathcal{P''}} 
 \inferon{(g^\wedge (\predLIT{}))(t'')}{P''}{1}{}
$ %\end{equation*}
and therefore 
$ %\begin{equation*}
 \vdash _{\mathcal{P''}} 
 \inferatom{(g^\wedge (\predLIT{}))(t'')}{1}{} 
$. %\end{equation*}
Applying condition (\ref{eq:multiChu}) again we get 
$ % \begin{equation*}
 \vdash_{\mathcal{P'}} 
 \inferatom{\predLIT{(g^\vee (t''))}}{1}{}  
$ % \end{equation*}
and so 
\begin{equation*}
 \Vdash_{\emptyset} 
 \inferon{\predLIT{(g^\vee (t''))}}{P'}{1}{} 
\end{equation*}
saying that the component bulb $t= g^\vee (t'')$ is lit according to base $\mathcal{P'}$. \fillBox 
\end{example}

\section{Information Flow and System Structure} 
\label{sec:flow}

\subsection{Information Flow from Inferomorphisms and Inferonic Bases} 
\label{subsec:flow-morph-base}

Recall the intuition behind the notion of `constraint' from situation theory in which logical properties' holding (that is, infons' being `supported') in one situation necessarily entail logical properties in another situation. That is, constraints describe a form of information flow
in the information-as-correlation sense \cite{Benthem2006,BenthemMartinez2008stories}. The existence of a constraint is often presumed to be tied to some structural factor relating situations, known as a `channel'. In the `Logic of Distributed Systems' 
approach \cite{BarwiseSeligman1997} using classifications,  %(Section~\ref{subsec:Classifications}), 
infomorphisms, when combined into channels, give a mathematically precise notion of correlation of logical properties (instances of classification relations) across linked system components.

In the approach of Section~\ref{sec:LogicalBasis}, the logical language and the base-extension support relation semantics are designed to allow for reasoning, relative to an inferonic base, about properties of other inferonic bases, via the inferon construct. 
Flow of information can then be handled, as in Example~\ref{example:NoSmoke}, by ensuring constraints connecting inferons regarding inferonic bases are supported in the inferonic base. An example of this is the rule (\ref{rule:SmokeFire}), where support is ensured by requiring the constraint as a rule within a certain inferonic base. 

In Section~\ref{sec:morphisms-flow}, an inferomorphism, because it leads to Proposition~\ref{result:InferonChuEquivalence}, provides an alternative view, in which the (now, bidirectional) constraint is not formalized within the inferonic base, but rather holds because of the properties of the inferomorphism mapping. 

In this section, we study another notion of situation, namely `site' that comes along with a natural notion of channel. Our sites are built from  inferonic atoms (Definition~\ref{def:inf-atom}). The logical languages and base-extension semantics will be generalizations of that introduced in Section~\ref{sec:LogicalBasis}. This approach combines the view of information developed above with the structural approach of Barwise, Gabbay, and Hartonas  \cite{BarwiseGabbay1996}.

\subsection{A Logic of Inferons with Sites} 
\label{subsec:LogicSites} 

Building on ideas given by Barwise and Gabbay in \cite{BarwiseGabbay1996}, we define a \emph{site}, $\Pfrak$ to be a finite set of inferonic atoms and
recall that sites form the contexts used in the basic derivability relation for any base. %(Definition~\ref{def})

Given this notion of site, we can extend the logic of inferons by giving a \emph{contextual support relation} $\Vdash^\Pfrak$ defined as in Figure~\ref{fig:supportSites}. The soundness and completeness theorems of Section~\ref{sec:LogicalBasis} can be extended to this logic. 

\begin{figure}[ht]
\hrule
\[
    \begin{array}{lrcl} 
        (\mbox{At}) & \Vdash ^\Pfrak _\baseB \inferonPlain{P}{\baseC}{b}{}
        & \mbox {iff} & 
        \mbox{$\Pfrak \vdash_{\baseB \cup \baseC} \inferatom{P}{b}{}$
        for closed $P$}\\
       %(\mbox{At}) & \Vdash_\mathcal{B} \inferon{P}{P}{b}{}   & \mbox{iff} & 
       %\mbox{$\vdash_{\mathcal{B} \,\cup\, \mathcal{P}} 
       %\inferatom{P}{b}{}$ for closed $P$} \\ 
       (\wedge) & \Vdash^{\Pfrak}_\mathcal{B} {\phi}_1 \wedge {\phi}_2 & \mbox{iff} & \mbox{$\Vdash^\Pfrak_\mathcal{B} {\phi}_1$ and $\Vdash^\Pfrak _\mathcal{B} {\phi}_2$} \\ 
       (\vee) & \Vdash^\Pfrak_\mathcal{B} {\phi}_1 \vee {\phi}_2 
       & \mbox{iff} & \mbox{for every closed  
       ${\iota}$ and every 
       $\mathcal{C} \supseteq \mathcal{B}$,} \\   
       & & & \mbox{if $\phi_1  \Vdash ^\Pfrak_\mathcal{C} \iota$ and $\phi_2 \Vdash ^\Pfrak _\mathcal{C} \iota$, then $\Vdash^\Pfrak_\mathcal{C} \iota$} \\
       (\supset) & \Vdash ^\Pfrak_\mathcal{B} {\phi}_1 \supset {\phi}_2 & \mbox{iff} & \mbox{${\phi}_1 \Vdash^\Pfrak_\mathcal{B} {\phi}_2$} \\ 
       (\mbox{Inf}) & \mbox{for $\Theta \neq \emptyset$, 
        $\Theta \Vdash ^\Pfrak_\mathcal{B} {\phi} $} & \mbox{iff} & 
        \mbox{for every $\mathcal{C} \supseteq \mathcal{B}$, and every $\Pfrak '$, if $\Vdash ^{\Pfrak'}_\mathcal{C} {\psi}$, for every ${\psi} \in \Theta$,} \\ 
        & & & \mbox{then $\Vdash^{\Pfrak \cup \Pfrak'}_\mathcal{C} {\phi}$} \\ 
        (\bot) & \Vdash^\Pfrak_\mathcal{B} \bot  
        & \mbox{iff} & \mbox{for all closed $\iota$, $\Vdash^\Pfrak_\mathcal{B} \iota$} 
    \end{array}
\]

\dotfill 
\[
\begin{array}{lrcl}     
(\forall) &  \Vdash^\Pfrak_\mathcal{B} \forall{x}.\phi & \mbox{iff} & 
     \mbox{$\Vdash ^\Pfrak_\mathcal{B} \phi[t/x]$ for all $t \in {\rm Cl}({\mathcal{T}})$} \\
(\exists) &  \Vdash^\Pfrak_\mathcal{B} \exists{x}.\phi & \mbox{iff} & 
  \mbox{for all $\mathcal{C} \supseteq \mathcal{B}$ and any closed $\iota$, if} \\ 
  & & & \mbox{$\phi[t/x] \Vdash ^\Pfrak _\mathcal{C} \iota$, for any $t \in {\rm Cl(\mathcal{T})}$, then $\Vdash^\Pfrak_\mathcal{C} \iota$} 
\end{array}
\]

 \dotfill 
\[
\begin{array}{lrcl}
(\mbox{Validity}) & \Gamma \Vdash \phi & \mbox{iff} & 
    \mbox{for all $\baseB$, $\Gamma \Vdash^\emptyset_\baseB \phi$}
\end{array}
\]

\hrule 
\caption{Contextual support relation} 
\label{fig:supportSites}
\end{figure}

For any site $\Pfrak$, let $\baseop{(\Pfrak)}$ be the base consisting of the rule set $\{ \Rightarrow \iota \mid \iota \in \Pfrak \}$. The following is easily shown: 
\begin{equation*}
    \Pfrak \vdash_\baseB \iota 
    \mbox{ iff } 
    \vdash _{\baseB \cup \baseop{(\Pfrak)}} \iota  
\end{equation*}
for all $\iota$, $\Pfrak$ and $\baseB$. From this it follows that
\begin{equation*}
    \Theta \Vdash ^\Pfrak _ \baseB \phi 
    \mbox{ iff }
    \Theta \Vdash _{\baseB \cup \baseop{(\Pfrak)}} \phi 
\end{equation*}
for all $\Theta$, $\phi$, $\Pfrak$ and $\baseB$.

When site $\Pfrak$ is included in $\Pfrak '$, we say that the inclusion is a \emph{channel}. Such a channels is labelled by the relative complement $\lambda:=\Pfrak ' \setminus \Pfrak$. We write the channel in the notation $\Pfrak \stackrel{\lambda}{\rightsquigarrow} \Pfrak '$. Channels can be composed by function composition, with labels combining by union. As a result, site, channels and composition give a simple example of an \emph{information network} in the sense of Barwise, Gabbay and Hartonas \cite{BarwiseGabbay1995}. 

The support relation $\Vdash ^\Pfrak _\baseB$ provides us with instances of the form 
$
    \Vdash ^\Pfrak _\baseB \phi
$
at sites $\Pfrak$. 
We can take this base-extension support relation instance to be a rendering of the situation-theoretic `$\Pfrak$ supports $\phi$' relation instance, relativized to $\baseB$.  

The base-extension support relation instances of the form 
$
  \Gamma \Vdash ^{\Pfrak} _\baseB \phi  
$ 
for $\Gamma$ non-empty can be viewed as formul{\ae} that are situated at channels, relative to bases. 
%Both of these may be compared to the `of-type' relation for types of sites, and types of channels  \cite{BarwiseGabbay1996} (specifically, Definitions 2 and 3).
For a fixed base $\mathcal{B}$, examining again the Inf-rule
%\begin{equation*}
%\begin{array}{rcl}
%    \Gamma \Vdash _\mathcal{B}^{R( \cdot )} \psi 
%    & \mbox{iff} & \mbox{for all $\mathcal{X} \supseteq \mathcal{B}$ and all $U \in \mathbb{B}(A)$, if $\Vdash _{\mathcal{X}}^U \Gamma$, then  $\Vdash_\mathcal{X}^{R(U)} \psi$}
%\end{array}
%\end{equation*}
of Figure~\ref{fig:supportSites}, 
we see that this concerns the flow of information along the channel (inclusion) from $\Pfrak'$ to $\Pfrak \cup \Pfrak'$. The channel underpins the logical constraints according to Inf. 

The base-extension support relation semantics ties together, in a systematic way, support at sites and support at channels. For example, this is evident in  the support relation clauses
\[
\begin{array}{rcl}
\Vdash ^\Pfrak _\baseB \phi \supset \psi & \mbox{iff} &
\phi \Vdash^{\Pfrak}_\baseB \psi \\
    & \mbox{iff} & \mbox{for all $\baseC \supseteq \baseB$, and all ${\Pfrak'}$, if $\Vdash^{\Pfrak'}_\baseC \phi$, then $\Vdash^{\Pfrak \cup {\Pfrak'} }_\baseC \psi$} 
\end{array}
\]
This can be seen as a way of formalizing constraints, via flow along channels, and also of internalizing this in the logic using implication.

\begin{example}[Access control] 
\label{example:access-control}
    We consider an access-control system that consumes a password and a token, and then grants access. 
    The system consumes an inferon 
    $\inferonPlain{p}{\baseB _1}{1}{}$, 
    representing (information carried by) supply of the password. It becomes a sub-process that consumes a token, carrying 
    $\inferonPlain{q}{\baseB _2}{1}{}$, 
    and makes a decision to grant access, corresponding to atom 
    $\inferonPlain{r}{\baseB _3}{1}{}$. We want a constraint, informally expressed as `if $p$ is presented, then $q$ is presented, then access, $r$, is granted.'
    We specify the system, without the token $p$, by a formula 
    $
    \inferonPlain{p}{\baseB _1}{1}{} \supset (\inferonPlain{q}{\baseB _2}{1}{}  \supset \inferonPlain{r}{\baseB _3}{1}{}) 
    $.
    In this example, we do not want anything other than the context to lead to support for the three inferons, and so we take $\baseB _1 = \baseB_2 = \baseB_3 = \emptyset$.
    
    A model of the system consists of a base and a site, such that this formula is supported. 
    We take the site to be empty, since we consider it in isolation. 
    The complete system, without the token, is modelled by some $\baseB$ with 
    %\begin{equation}
    %\label{equation:SitesChannelsFromInf}
    $
    \Vdash ^\emptyset _\baseB
    \inferonPlain{p}{\emptyset}{1}{} \supset (\inferonPlain{q}{\emptyset}{1}{}  \supset \inferonPlain{r}{\emptyset}{1}{}) 
    $
    .
    %\end{equation}
    The significance of the use of $\emptyset$ will be seen in capturing the desired constraints: 
    \begin{equation}
    \label{equation:SitesChannelsFromInfConstraint}
    \mbox{if $\Vdash^{\Pfrak}_\baseB \inferonPlain{p}{\emptyset}{1}{}$ and 
    $\Vdash^{\Pfrak'}_\baseB \inferonPlain{q}{\emptyset}{1}{}$, 
    then $\Vdash^{\Pfrak \cup \Pfrak'}_\baseB \inferonPlain{r}{\emptyset}{1}{}$}
    \end{equation}
    parametrized on sites $\Pfrak, \Pfrak'$.

    Assume that we do have a base $\baseB$ providing a model.%(\ref{equation:SitesChannelsFromInf}). 
    This yields 
    $\inferonPlain{p}{\emptyset}{1}{}
    \Vdash^{\emptyset}_\baseB 
    \inferonPlain{q}{\emptyset}{1}{}  \supset \inferonPlain{r}{\emptyset}{1}{}  
    $ 
    using the support relation clause for $\supset$. For any site $\Pfrak$, there is a channel from site $\Pfrak$ to $\Pfrak \cup \emptyset = \Pfrak$. Now suppose $\Vdash_\baseB^\Pfrak \inferonPlain{p}{\emptyset}{1}{}$ as in the premiss of the desired constraint (\ref{equation:SitesChannelsFromInfConstraint}). Then Inf transfers this along the channel to give 
    $\Vdash_\baseB^\Pfrak 
    \inferonPlain{q}{\emptyset}{1}{} \supset \inferonPlain{r}{\emptyset}{1}{}$. 
    The support relation clause at $\supset$ then tells us that 
    $\inferonPlain{q}{\emptyset}{1}{} \Vdash_\baseB ^{\Pfrak} \inferonPlain{r}{\emptyset}{1}{}$. 
   Assuming 
   $\Vdash^{\Pfrak'}_\baseB \inferonPlain{q}{\emptyset}{1}{}$, the channel from $\Pfrak$ to $\Pfrak \cup \Pfrak'$ along provides the constraint instance of Inf ensuring that 
    $\Vdash _\baseB ^{\Pfrak \cup \Pfrak'} 
    \inferonPlain{r}{\emptyset}{1}{}$. 
    %
    % COMMENT FROM BUNCHED ACCOUNT
    %Compared to the informal, hypothetical statement of the desired constraint on the access control system, this makes it clear how the site changes when the system is presented with a password and a token. 
    \fillBox
\end{example}

\begin{example}[Airport security]
\label{example:airport} 
A model of the departure-security process at an airport is depicted in Figure~\ref{fig:airport} (taken from \cite{GGP2024-MFPS}). In 
\cite{GGP2024-MFPS}, the example is formulated using the substructural logic BI \cite{GP-BI2023,GGP2025-BI}. Our analysis is based 
on just intuitionistic propositional logic, but makes use of the theory of inferons. This use inferons here  allows local inferential capacities to capture the distinctions (which are, essentially, about location) that are enforced in \cite{GGP2024-MFPS} using the substructural connectives. Thus 
our deployment of inferons, with their local inferential capacities, allows the structure of the system that is being modelled to be captured in an intuitionistic setting rather than a much more complex substructural one (cf. \cite{IshtiaqOHearn2001,YangOHearn2002}. 

Passengers enter the airport and proceed to the check-in desk. Their baggage enters the hold-baggage security process and they enter the cabin and passenger security process, and so on. Critically, the two processes are synchronized at the aircraft, where the baggage in the hold is reconciled with the passengers in the cabin to whom it belongs. Just in Examples \ref{ex:intel-share} and \ref{ex:intel-share-2}, the use inferons provides a clear distinction here between local and global reasoning capacities. While, as mentioned in the discussion preceding Examples \ref{ex:intel-share} and \ref{ex:intel-share-2}, one could of course work relative to a single underlying base,  
% --- indeed, one could encode everything in classical set theory --- 
to do so would obscure the structure of the model and 
potentially admit incorrect reasoning about its security properties by allowing local reasoning from one location to become available at another. For example, items permitted only in checked baggage might then be understood to be allowed in cabin baggage. 

A particular passenger, $o$, passes through security controls prior to a flight. The passenger is travelling  with hold-baggage $h$, and is using passport $p$. We describe how information is processed throughout a simplified and idealized version of process. The focus in our description here is solely on securing the flight itself. The architecture of the system is indicated in Figure~\ref{fig:airport}.

\begin{figure}[ht]
    \hrule
    \includegraphics{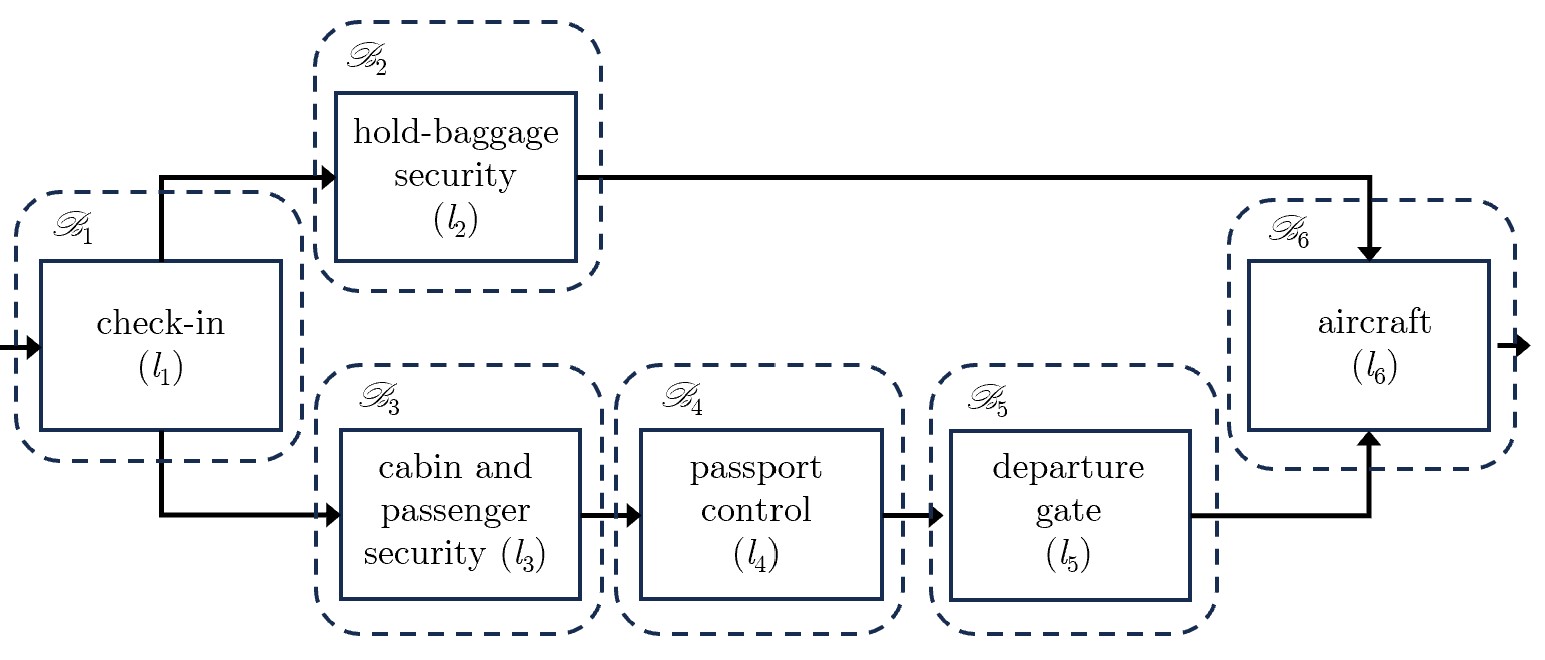}
    \hrule
    \caption{The airport security system (taken from \cite{GGP2024-MFPS})}
    \label{fig:airport}
\end{figure}

%The propositions, to be the content of information carried:
%\begin{itemize}
%    \item $I^W(p)$: the identity contained in the passport, together with all assurances checkable by a checck-in desk operative
%    \item $I^S(p)$: the identity contained in the passport, together with all assurances checkable by a passport control operative
%    \item $P^W (p)$: Passport $p$ has been (weakly) checked at check-in for bag-drop
%    \item $A(h,t,p)$: There is an association between hold baggage $h$, ticket $t$, and passport $p$
%    \item $H(h)$: baggage $h$ contains harmless contents
%    \item $L(h)$; the baggage is (to be) loaded
%    \item $T(t)$: Content carried by ticket $t$ 
%    \item $O(o)$ the passenger $o$ carries no harmful materials materials % `o' for Otto
%    \item $N(a)$: Airside contains no harmful materials
%    \item $P^S(p)$: Passport $p$ has been strongly checked at passport control
%    \item $B(p,t)$: the bearer of passport $p$ and ticket $t$ may board
    % THE FOLLOWING NOT uSED
    %\item $M(a)$: All individuals airside were ticket cecked at security % I think we can do without this
    %\item $F(p)$: that the passenger is not on a no-fly list
    %\item $D$: the security process is done (complete), within the aircraft
%\end{itemize}

For each $j=1 , \ldots ,6$, at each (physical) location $l_j$ we define a site $S_j$ and a base $\baseB _j$. For $j = 1 , \ldots , 5$, the base is the base underpinning the (relevant for process component) background reasoning capacity of the agent at that location, prior to execution of the process. The site component is the atomic inferonic information that becomes available for use with the base at the location as the process executes.
A slightly different set-up applies at location $l_6$.

\medskip
\textbf{$l_1$ (check-in): } 
 The passport $p$ is checked for validity. This involves using some of the atomic inferonic information $\inferatom{I^{W} (p)}{1}{}$ carried by the passport. 
A ticket $t$ is issued. This functions as an  authentication token that is used in subsequent stages. The ticket carries this authentication information as an inferonic atom $\inferatom{T(t)}{1}{}$. 
% - The ticket is a short-term authentication token
% - In practice, there is a check that the identity is expected for a flight in the next, say, 5 hours, but we ignore this. 
Not all of the information carried by the passport can be used by the check-in agent, but will be used at a later stage: we represent this by the atom $\inferatom{I^S(p)}{1}{}$. The site 
$S_1 = \{\inferatom{I^W(p)}{1}{} , \inferatom{I^S(p)}{1}{}\}$ represents the passport information carried by the passport.
%\notemc{We are identifying Otto with his passport in the present bit of the treatment, not merely binding them, which isn't quite right.}

Base $\baseB_1$ is used by the agent at $l_1$. This includes rules
\begin{equation*}
    \dfrac{
            \inferatom{I^W(p)}{1}{}
        }
        {
            \inferatom{P^W(p)}{1}{}
        }
    \qquad
    \dfrac{
            \inferatom{I^W(p)}{1}{}
        }
        {
            \inferatom{A(h,t,p)}{1}{}
        }
    \qquad
    \dfrac{
            \inferatom{I^W(p)}{1}{} 
        }
        {
            \inferatom{T(t)}{1}{}
        }
        .
\end{equation*}
%\begin{equation*}
%    \dfrac{
%    }
%    {
%        \inferatom{I^W(p)}{1}{} \vdash \inferatom{P^W(p)}{1}{}
%    }
%    \qquad
%    \dfrac{}{\inferatom{I^W(p)}{1}{} \vdash \inferatom{A(h,t,p)}{1}{}}
%    \qquad
%    \dfrac{}{\inferatom{I^W(p)}{1}{} \vdash \inferatom{T(t)}{1}{}}
%\end{equation*}
The first rule is used to generate an inferonic atom $\inferatom{P^W(p)}{1}{}$ meaning that the passenger with passport $p$ has passed the (weak) check of the passport conducted by the check-in desk agent. 
The second rule is used to generate the inferonic atom $\inferatom{A(h,t,p)}{1}{}$, used to support information carried by the luggage label, and associating o's baggage, passport, and  ticket.  
The third rule is used to generate an inferonic atom $\inferatom{T(t)}{1}{}$ used to support information carried by the ticket. 

We now have $S_1 \vdash _{\baseB_1} \inferatom{P^W(p)}{1}{}$ and $S_1 \vdash_{\baseB_1} \inferatom{A(h,t,p)}{1}{}$ 
and 
$S_1 \vdash_{\baseB_1} \inferatom{T(t)}{1}{}$, 
and so 
\begin{equation*}
    \Vdash_{\emptyset}^{S_1} \inferonPlain{P^W(p)}{{\baseB}_1}{1}{}
    \qquad \mbox{and} \qquad 
    \Vdash_{\emptyset}^{S_1} \inferonPlain{A(h,t,p)}{{\baseB}_1}{1}{}
    \qquad \mbox{and} \qquad 
    \Vdash_{\emptyset}^{S_1} \inferonPlain{T(t)}{{\baseB}_1}{1}{}
\end{equation*}
representing the information that is generated by the check-in agent.

\medskip \textbf{$l_2$ (hold-baggage security): } 
The tag on the baggage carries with it information to this location. The 3-way association of $h$, $t$, and $p$, flows from the site at $l_1$ to the site at $l_2$ is explicit in the tag. Implicitly, the information that $p$ has been checked at check-in is inferrable, and thus also carried along this channel. The baggage scanner at $l_2$ provides an inferonic atom $\inferatom{H(h)}{1}{}$, meaning that the bag is clear, to the baggage-handler agent at this location. 
The agent at $l_2$ makes use of site 
$
S_2 = S_1 \cup \{ 
    \inferatom{P^W (p)}{1}{} , 
    \inferatom{A(h,t,p)}{1}{}, 
    \inferatom{H(h)}{1}{}
\}, 
$and we have a channel 
$S_1 \stackrel{}{\rightsquigarrow} S_2$. 
We have that $\inferatom{T(t)}{1}{}$ is not in $S_2$, representing the fact that this information does not flow to $l_2$.
%labelled by
%$
%\lambda_2 :=\{ 
%\inferatom{P^W (p)}{1}{} , 
%\inferatom{A(h,t,p)}{1}{} , 
%\inferatom{H(h)}{1}{}
%\}
%$.
%The information is then used in processing by an agent at $l_2$ using $\baseB _2$. The agent scans the baggage to obtain information that the contents of the bag are harmless. 

The agent uses the three atoms to generate information that the bag should be, and is, loaded. This is supported by an inferonic atom $\inferatom{L(h)}{1}{}$. Base $\baseB_2$, used by the agent, contains the rule
\begin{equation*}
    \dfrac{
        \inferatom{P^W (p)}{1}{} 
        \qquad 
        \inferatom{A(h,t,p)}{1}{} 
        \qquad
        \inferatom{H(h)}{1}{}
    }
    {
        \inferatom{L(h)}{1}{}
    }
\end{equation*}
%\begin{equation*}
%    \dfrac{
%        S \vdash \inferatom{P^W (p)}{1}{} \qquad 
%        S \vdash \inferatom{A(h,t,p)}{1}{} \qquad
%        S \vdash \inferatom{H(h)}{1}{}
%    }{S \vdash \inferatom{L(h)}{1}{}}
%\end{equation*}
%parametrized on $S$. 
Then
$S_2 \vdash_{\baseB_2} \inferatom{L(h)}{1}{}$
and $\Vdash^{S_2}_\emptyset \inferonPlain{L(h)}{\baseB_2}{1}{}$. 

\medskip
\textbf{$l_3$ (cabin and passenger security):} 
At this location, the ticket $t$ is presented and the information carried by it is used. An inferonic atom 
$\inferatom{O(o)}{1}{}$, capturing the successful scanning of the passenger is also used.
Information $N(a)$ that no harmful materials are airside is generated (in a more detailed model, an invariant would be maintained).
% ABout $M(a)$
%Information that all passengers presenting at the gate were previously ticketed is generated. This helps prevent failures (and insider attacks) at the subsequent stages, in particular removing single-point-of-failure with respect to ticket use at the gate. It also helps maintain security properties of airside. However, this information is not normally used at later stages, and we make no further mention of it.

Site $S_3 = S_1  \cup \{\inferatom{T(t)}{1}{} , 
\inferatom{O(o)}{1}{}\}$ contains both the information obtained from $l_1$ and the information obtained from the security scanning process. There's a channel $S_1 \stackrel{}{\rightsquigarrow} S_3$. 
%labelled by 
%$\lambda_3 := \{ 
%\inferatom{T(t)}{1}{}  ,
%\inferatom{O(o)}{1}{} 
%\}
%$.  
The atom $\inferatom{A(h,t,p)}{1}{}$ is not in $S_3$, representing the fact that this information does not flow to, and so is not used at, $l_3$. Base $\baseB_3$ for the security agent at $l_3$ contains 
\begin{equation*}
    \dfrac{ 
        \inferatom{T(t)}{1}{}
        \qquad
        \inferatom{O(o)}{1}{}
    }
    { 
        \inferatom{N(a)}{1}{}
    }
\end{equation*}
%\begin{equation*}
%\dfrac{ }
%{\inferatom{T(t)}{1}{},  \inferatom{O(o)}{1}{} \vdash \inferatom{N(a)}{1}{}
%}
%\end{equation*}
so that we have $S_3 \vdash _{\baseB_3} \inferatom{N(a)}{1}{}$ and $\Vdash _\emptyset ^{S_3} \inferonPlain{N(a)}{\baseB_3}{1}{}$. 

\medskip
\textbf{$l_4$ (passport control): } The only relevant information here is that carried by the passport (and that carried by the passenger, in his biometric form, which we do not here model separately). In information-processing terms, for assuring the aircarft is secure, this process could be carried out in parallel with those at $l_1$ and $l_3$. It does have to be sequenced before the processes at $l_5$ and $l_6$. We ignore other forms of information required to successfully cross passport control in practice, such as absencce of various `flags'. We take site 
$S_4 = S_3 \cup \{ \inferatom{N(a)}{1}{} \}$
and note that there is a channel $S_3 \stackrel{\lambda_4}{\rightsquigarrow} S_4$. 
%with $\lambda _4 = \{ \inferatom{N(a)}{1}{} \}$.

The agent at this location makes use of all the information in the passport, and so base $\baseB_4$ has a rule
\begin{equation*}
    \dfrac{
        \inferatom{I^W(p)}{1}{}
        \qquad 
        \inferatom{I^S(p)}{1}{}
    }{
        \inferatom{P^S(p)}{1}{}
    }
\end{equation*}
%\begin{equation*}
%    \dfrac{
%    }{
%        \inferatom{I^S(p)}{1}{}  
%        \vdash 
%        \inferatom{P^S(p)}{1}{}
%    }
%\end{equation*}
generating an inferonic atom $\inferatom{P^S(p)}{1}{}$ to the effect that the passport passes these strong checks. This yields 
$S_4 \vdash_{\baseB_4} \inferatom{P^S(p)}{1}{}$
and therefore $\Vdash^{S_4}_\emptyset 
    \inferonPlain{P^S(p)}{\baseB_4}{1}{}$. 

\medskip
\textbf{$l_5$ (departure gate):} At the departure gate, an agent checks the ticket again. This, together with information that airside is secure and the passenger has passed through passport control, is taken to be sufficient to allow the passenger to board. (In practice, another passport check is often conducted.) 

We take site $S_5 = S_4 \cup \{ \inferatom{N(a)}{1}{} \}$ and $\baseB_5$ to include rule representing the ability of the agent at $l_5$ to make a decision for the passenger with $p$ and $t$ to board: 
\begin{equation*}
    \dfrac{
        \inferatom{P^S(p)}{1}{} 
        \qquad 
        \inferatom{T(t)}{1}{} 
        \qquad 
        \inferatom{N(a)}{1}{}
    }{
        \inferatom{B(p,t)}{1}{}
    }
\end{equation*}
%\begin{equation*}
%    \dfrac{
%    }{
%        \inferatom{P^S(p)}{1}{} ,
%        \inferatom{T(t)}{1}{} ,
%        \inferatom{N(a)}{1}{}
%        \vdash
%        \inferatom{B(p,t)}{1}{}
%    }
%    .
%\end{equation*}
We have $\vdash^{S_5}_{\baseB_5} \inferatom{B(p,t)}{1}{}$ and 
$\Vdash^{S_5}_\emptyset \inferonPlain{B(p,t)}{\baseB_5}{1}{}$. 

\medskip
\textbf{$l_6$ (aircraft): } 
A final proccess occurs in the cabin, where it is checked that the passengers boarded and baggage loaded match. This requires input of information from a lead member (purser) of the cabin crew and a lead staff member of the baggage-handling team. 
We describe a simplified version of such a process, in which the purser conducts the check using boarding information provided by the gate and baggage information from the baggage-handler. Rather than modelling the  taking of the decision as something that consumes some information and produces new information, we rather focus on the fact that there is sufficient information in place to support such a decision.

We are not interested presently in the information produced by the purser, so for this model have a free choice for what to take as their base, $\baseB _6$.  We are interested in the information that they have available. We take site $S_6 := S_2 \cup S_5 
\cup \{ \inferatom{B(p,t)}{1}{} \}$. 
This does not include the inferons
$\inferatom{L(h)}{1}{}$ and $\inferatom{A(h,t,p)}{1}{}$ generated by the baggage handler.
%
%with 
%\begin{align*}
%    \lambda_{2,6}: = & S_5 \\%\cup \{ \inferatom{L(h)}{1}{} ,\inferatom{B(p,t)}{1}{} \} \\
%    \lambda_{5,6}: = & S_2 \\ %\cup \{ \inferatom{L(h)}{1}{} ,\inferatom{B(p,t)}{1}{} \} 
%\end{align*}
%giving channels
%\begin{equation*}
%    S_2 \stackrel{\lambda_{2,6}}{\rightsquigarrow} S_6
%    \qquad
%    \mbox{ and }
%    \qquad
%    S_5 \stackrel{\lambda_{5,6}}{\rightsquigarrow} S_6
%\end{equation*} 

We have, by monotonicity, that  
$S_6 \vdash_{\baseB_2 \cup \baseB_6}  \inferatom{B(p,t)}{1}{}$, that $S_6 \vdash_{\baseB_2 \cup \baseB_6} \inferatom{L(h)}{1}{}
$, and that $S_6 \vdash_{\baseB_2 \cup \baseB_6}  \inferatom{A(p,t,h)}{1}{}$.
%\end{align*}
%\begin{align*}
%    S_6 \vdash _{\baseB_2 \cup \baseB_6} & \inferatom{B(p,t)}{1}{} \\
%    S_6 \vdash _{\baseB_2 \cup \baseB_6} & \inferatom{L(h)}{1}{} \\
%    S_6 \vdash _{\baseB_2 \cup \baseB_6} & \inferatom{A(p,t,h)}{1}{}
%\end{align*}
We therefore have 
$\Vdash_{\baseB_2}^{S_6} 
    \inferonPlain{B(p,t)}{\baseB_6}{1}{}
    \wedge 
    \inferonPlain{L(h)}{\baseB_6}{1}{}
    \wedge 
    \inferonPlain{A(p,t,h)}{\baseB_6}{1}{}$. 
Equivalently,  
$\Vdash_{\baseB_2}^{S_6} 
    \inferonPlain{
        B(p,t)\wedge 
        L(h)  \wedge 
        A(p,t,h)
    }
    {\baseB_6}{1}{}$. 
The conjunction of three inferons is supported at the site $S_6$ using shared inferential capacity from the baggage handler (who uses base $\baseB_2$). Each of the three inferons involves the inferential capacity of the purser ($\baseB_6$) who accepts information passed along from the departure gate. The inferential capacity of the baggage-handler appears as the base $\baseB_2$ that provides context in the support relation. 
%
%Note:
%The exchange law
%\begin{equation*}
%    \Vdash _{\baseB_2} ^{S_6} 
%    \inferonPlain{P}{\baseB_5}{1}{}
%    \mbox{ iff }
%    \Vdash _{\baseB_5} ^{S_6}
%    \inferonPlain{P}{\baseB_2}{1}{}
%\end{equation*}
%was used with $P=L(h)$ and $P=A(p,t,h)$ 
\fillBox
\end{example}

%\notemc{Mention bunched lambda calculi to connect back to the Lambek calculi of \cite{BarwiseGabbay1998}.}

% NEXT steps would be to develop the sequent calculi.

In the view above, flows of information are not between distributed parts of an information system; rather they are between a part and another part containing it. However, we can use a technique similar to the methods of modelling systems using inf(er)omorphisms with a common codomain. Suppose we have a pair of channels 
$\Pfrak_1 \stackrel{}{\rightsquigarrow} \Pfrak_3$
and $\Pfrak_2 \stackrel{}{\rightsquigarrow} \Pfrak_3$. 
That is, $\Pfrak_3 = \Pfrak_1 \cup \Pfrak_1' = 
\Pfrak_2 \cup \Pfrak_2'$, for some $\Pfrak_1'$ and $\Pfrak_2'$. Suppose $\phi$ and $\psi$ are contradictory inferons: in particular 
$\Vdash_\baseB ^{\Pfrak_3} (\phi \wedge \psi) \supset \bot$.
If $\Vdash_\baseB^{\Pfrak_1} \phi$, then it cannot be the case that $\Vdash_\baseB ^{\Pfrak_2} \psi$. %A form of information about site $\Pfrak _1$ is transferred to site $\Pfrak_2$.

Example~\ref{example:airport} is modelled in a framework for understanding distributed systems, depicted more generally than for the airport security examples in Figure~\ref{fig:generic}. 
It should be clear --- at least in principle --- that the approach adopted for modelling airport security suggests the availability of a more generic approach. 

%\note{Move to Discussion section?}

These are systems with physically distributed, interacting components, embedded in some operating environment (see, for example, \cite{CIP2022,ICP2025}). Interaction happens via flows of resources of some kind, where the flow of resources may be constrained by an interface. In the present paper, we are interested in the case where `resource' is a form of information, flowing via `channels' and with logical `constraints' specifying the interface. 

\begin{figure}[ht]
    \hrule
    \vspace{3mm}
%    \includegraphics{airport.jpg}
%    \hrule
%    \caption{\begin{figure}[ht]
%    \hrule
%    \vspace{3mm}
    \includegraphics[scale=0.7]{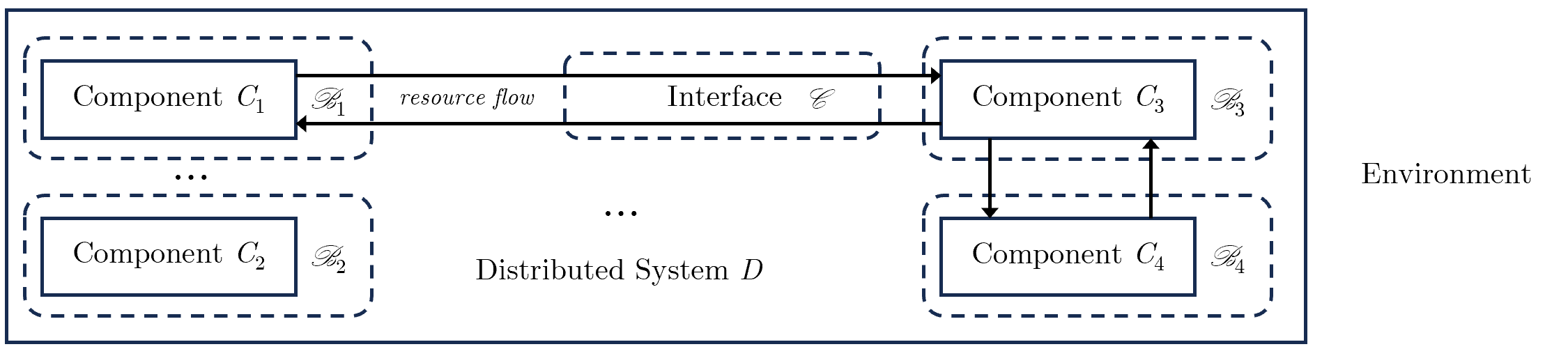}
    \vspace{3mm}
    \hrule
    \caption{A generic distributed system (taken from \cite{GGP2024-MFPS}, cf. \cite{ICP2025})}
    \label{fig:generic}
%    \label{fig:airport}
%\end{figure}}
\end{figure}

Certain substructural logics are well-known to be useful in producing richer models than can be expressed purely in intuitionistic logic. For example, multiplicative linear logics have a number-of-uses reading, while bunched logics have a sharing interpretation. It is known how to give a base-extension semantics to both the linear logic IMALL and the bunched logic BI \cite{GGP2024-MFPS,GGP2025-BI}.\footnote{Note that BI has a model-theoretic semantics in terms of ternary relations (e.g., \cite{GP-BI2023} and references therein), perhaps offering an interesting comparison with \cite{Mares1996,Restall1996}.} In both cases, the support relation takes the form $\Gamma \Vdash ^S \phi$ in which $S$ has structure (and is a multiset or bunch) containing atoms, rather than simply a set as we have it for intuitionistic logic. The set-up with sites and channels generalizes to these support relations. Sites consist of these structured collections of atoms, channels correspond to an evident operations on sites (which can be formulated as filling the hole in a context with a single hole), and constraints work by the respective versions of the Inf clause. We mention in particular the generalized form discussed in \cite{GGP2024-MFPS}:
\begin{equation*}
   \begin{array}{rcl}
   \Gamma \Vdash^{S(\cdot)}_\baseB \phi 
   & \mbox{iff} & 
   \mbox{for all $\baseC \supseteq \baseB$ and  
   all $U \in R(A)$, if  
   $\Vdash^U_\baseC \Gamma$, then  
   $\Vdash^{S(U)}_\baseC \phi$}  
   \end{array}
   \tag*{(Gen-Inf)}
\end{equation*}
where we would take $R(A)$ to be the set of sites and $S(\cdot)$ is a context with a hole. Work remains to be done in fully developing the underlying theory in the case of inferonic atoms and in exploring richer examples.

\section{Discussion} \label{sec:discussion}

We have presented a first step towards an inferentialist account  of information that is grounded in proof-theoretic semantics. Through such a theory we seek to capture the fundamentally inferential nature of information in 
a way that is both direct and a strong foundation for further development. 

We can identify scope for further development in three main directions: metaphysics, logic, and systems. First, in metaphysics, we might seek to characterize the essence of the philosophy of information through employing the razor of proof-theoretic semantics. 

One important aspect, as  mentioned in Section~\ref{sec:metaphysics} and in Remarks \ref{rem:verid-1},  \ref{rem:verid-2}, and \ref{rem:verid-3} is an inferentialist veridicality thesis and its influence on our understanding of information. Such a thesis should take account of `assertion' and `denial' and make clear the role the distinction between proof-theoretic validity and truth. Such a project would be quite substantial and is beyond our present scope.  

Second, in the logical theory, it is clear that there is a wide range of logical variants to be explored, including classical, second-order, substructural, modal (recall, for example,  Remark~\ref{remark:modalities}), and epistemic logics of inferons. Comparisons with model-theoretic approaches such as discussed in \cite{Mares1996,Restall1996} may be useful. While the basic analyses of such logics in terms of base-extension semantics are now mostly available, their integration in logics of inferons remains to be explored. As indicated in Section~\ref{sec:metaphysics}, we foresee, in future work, substantive connections with van Benthem's work on logical dynamics and information \cite{vanBLDII2011,vanBenthemTracking2016}. 

In \cite{vanBenthemChu2000}, van~Benthem considers how Chu spaces --- which, as described by Barwise and Seligman in \cite{BarwiseSeligman1997}, provide a model of information flow --- can be seen as models for a two-sorted first-order logic. He identifies a class of formul{\ae} of this logic, called `flow formul{\ae}', that are preserved by Chu transforms between Chu spaces. A natural question is whether there is a corresponding analysis in our setting; that is, for inferomorphisms. 

Furthermore, the computational-logic-related ideas (in the context of base-extension  semantics, see \cite{GP-DF-NAF-B-eS-2023}) of closed world assumptions, negation-as-failure, and their interaction assertion/denial in inferons would seem to offer a useful perspective. Moreover, the role of proof-relevance, and so proof-theoretic validity, in the setting of logics of inferons remains unexplored.  

Third, the use of the ideas we have introduced as tools for modelling systems seems to offer both theoretical and pragmatic challenges. The starting point for an inferonic account of system modelling would be to capture systems a the level of generality suggested by Figure~\ref{fig:generic}, supported by (Gen-Inf). The inferentialist account of information flow we have begun to develop would seem to be well-aligned with this objective, and would seem to have good prospects of integrating well with 
inferentialist accounts of the meaning of models, such as those suggested by Kuorikoski and Reijula  \cite{KR2023} and Bueno \cite{Bueno2014}.

% \begin{itemize}  
% \item Modal and substructural logics, including BI 
% \item Epistemic logics  
% \item Richer polarities 
% \item Negation and polarities, various  
% \item Compound inferons --- combining 
% \item Second-order
% \item More developed account of modelling 
% \item Modelling dynamics 
% \item Inferomorphisms in larger examples (such as airport)
% \end{itemize}

\subsection*{Acknowledgements} This work has been partially supported by a gift from Amazon Web Services. We thank many of the attendees of the 6th Symposium on Proof-theoretic Semantics (London, February 2026) for their comments, to Johan van Benthem for his encouragement of our exploration of an inferentialist account of information, and to  Gabriele Brancati, Tim Button, Alexander Gheorghiu, and Tao Gu for %substantively 
discussing these ideas with us. Pym is grateful to University College London for its support of his sabbatical leave at the University of London's Institute of Philosophy, and to the Institute for its generous hosting. We are grateful the anonymous referees for their very helpful suggestions. 

\subsection*{Funding Statement} \label{sec:funding-statement }

This work has been partially supported by a gift from Amazon Web Services.

\bibliographystyle{siam}  
\bibliography{sitbib} 

\end{document}